\documentclass[12pt]{amsart}
\usepackage[mathcal]{eucal}
\usepackage{amssymb, graphicx, labelfig}
\usepackage[all]{xy}
\usepackage{hyperref} 

\newtheorem{thm}{Theorem}
\newtheorem{prop}[thm]{Proposition}
\newtheorem{lem}[thm]{Lemma}

\newtheorem{cor}[thm]{Corollary}
\newtheorem{comp}[thm]{Complement}

\theoremstyle{remark}
\newtheorem{rem}[thm]{Remark}

\theoremstyle{definition}

\newcommand{\C}{\mathbb C}

\newcommand{\Z}{\mathbb Z}

\newcommand{\HH}{\mathbb H}
\newcommand{\CP}{\mathbb {CP}^1}
\newcommand{\Id}{\mathrm{Id}}
\newcommand{\End}{\mathrm{End}}

\newcommand{\T}{\mathcal T}
\newcommand{\SSS}{\mathcal S^A}
\newcommand{\TT}{\mathcal T^\omega}
\newcommand{\RR}{\mathcal R_{\SL(\C)}}

\newcommand{\RP}{\mathcal R_{\PSL(\C)}}
\newcommand{\ZZ}{\mathcal Z^\omega}
\newcommand{\ZZZ}{\mathcal Z^\iota}
\newcommand{\SL}{\mathrm{SL}_2}
\newcommand{\PSL}{\mathrm{PSL}_2}
\newcommand{\Tr}{\mathrm{Tr}}
\newcommand{\E}{\mathrm{e}}
\newcommand{\I}{\mathrm{i}}

\newcommand{\db}{/\kern -4pt/}

\renewcommand{\leq}{\leqslant}
\renewcommand{\geq}{\geqslant}
\renewcommand{\phi}{\varphi}
\renewcommand{\epsilon}{\varepsilon}

\title
[Representations of the skein algebra: closed surfaces]
{Representations\\ of the Kauffman bracket  skein algebra III:  \\ closed surfaces and naturality}

\author{Francis Bonahon}
\address {Department
of Mathematics,  University of
Southern California, Los Angeles
CA~90089-2532, U.S.A.}
\email{fbonahon@math.usc.edu}

\author{Helen Wong}
\address{Department
of Mathematics, Carleton College, Northfield MN 55057, U.S.A.}
\email{hwong@carleton.edu}

\thanks{This research was partially supported by grants DMS-0604866, DMS-1105402, DMS-1105692, DMS-1406559 from the U.S. National Science Foundation, and by a mentoring grant from the Association for Women in Mathematics. In addition, parts of this article were written while the first author was a Simons  Fellow (grant 301050 from the Simons Foundation) in 2014-15, as well as a Simons Visiting Professor at the Mathematical Sciences Research Institute in Berkeley, California, (NSF grant 09032078000) in the Spring 2015 semester.}
\date{\today}

\begin{document}

\begin{abstract}
This is the third article in the series begun with \cite{BonWon3, BonWon4}, devoted to finite-dimensional representations of the Kauffman bracket skein algebra of an oriented surface $S$. In \cite{BonWon3} we associated a \emph{classical shadow} to an irreducible representation $\rho$ of the skein algebra, which is a character $r_\rho \in \RR(S)$ represented by a group homomorphism $\pi_1(S) \to \SL(\C)$. The main result of the current article is that, when the surface $S$ is closed, every character $r\in \RR(S)$ occurs as the classical shadow of an irreducible representation of the Kauffman bracket skein algebra. We also prove that the construction used in our proof is natural, and associates to each group homomorphism $r\colon \pi_1(S) \to \SL(\C)$ a representation of the skein algebra $\SSS(S)$ that is uniquely determined up to isomorphism. 
\end{abstract}

 \maketitle

This article is the third  in the series begun with \cite{BonWon3, BonWon4},  devoted to the analysis and construction of  finite-dimensional representations of the skein algebra of a surface. See also \cite{BonWon2} for a description of the corresponding program. 

The Kauffman bracket skein algebra $\SSS(S)$ of an oriented surface $S$ of finite topological type takes its origins in the construction of the Jones polynomial invariant \cite{Jones1, Jones2, Kauffman1, Kauffman2} of knots and links. It can be interpreted  \cite{Tur, BFK1, BFK2, PrzS} as a quantization of the character variety
$$
\RR(S) = \{ \text{group homomorphisms }r\colon  \pi_1(S) \to \SL(\C) \}\db \SL(\C)
$$
with respect to its Atiyah-Bott-Goldman \cite{AtiBott, Gold1, Gold2} Poisson structure. More accurately, the points of such a quantization are representations of the algebra $\SSS(S)$. 

When the parameter $A=\mathrm e^{-\pi\mathrm i \hbar}$ is a root of unity, a celebrated example of finite-dimensional representation of the skein algebra $\SSS(S)$ arises from Witten's quantum field interpretation of the Jones polynomial \cite{Witten}, and more precisely from  the Witten-Reshetikhin-Turaev topological quantum field theory associated to the fundamental representation of the quantum group $\mathrm U_q (\mathfrak{sl}_2)$ \cite{Witten, ReshTur, BHMV, TuraevBook, BonWon5}. In the current article, we construct a large family of new finite-dimensional representations of $\SSS(S)$, while providing a converse to the results of \cite{BonWon3}. 

This article is mostly concerned with the case where the surface $S$ is closed. The  case where $S$ has at least one puncture is much easier (at least assuming the results of \cite{BonWon1} and \cite{BonWon3}), and was treated in \cite{BonWon4}. The current closed surface case require many more ideas, and also involves several very surprising properties. 

More specifically, when $A^2$ is a primitive $N$--root of unity with $N$ odd, we identified in  \cite{BonWon3} certain  invariants for irreducible  representations $\rho \colon \SSS(S) \to \End(E)$. It is easier to restrict attention to the case where $A^N=-1$; this is no loss of generality, as \cite[\S 5]{BonWon4} indicates how to deduce the case $A^N=+1$ from this one, by using spin structures on the surface. When the surface is closed, there is only one invariant, consisting of a point in the character variety $\RR(S)$. By definition \cite{Mum} of the geometric invariant theory quotient  involved in the definition of $\RR(S)$, two homomorphisms $r$, $r'\colon \pi_1(S) \to \SL(\C)$ represent the same point of $\RR(S)$ if and only if they induce the same trace functions, namely if and only if $\Tr\,r(\gamma) = \Tr\,r'(\gamma) $ for every $\gamma \in \pi_1(S)$.

\begin{thm}[\cite{BonWon3}]
\label{thm:InvariantsExist}
Let $S$ be a closed oriented surface,  let $A$ be a primitive $N$--root of $-1$ with $N$ odd, and let $T_N(x)$ be the $N$--th normalized Chebyshev polynomial of the first kind, characterized by the trigonometric identity that $2\cos N\theta = T_N(2\cos \theta)$. For every irreducible finite-dimensional representation  $\rho \colon \SSS(S) \to \End(E)$  of the Kauffman bracket skein algebra, there exists a unique  character $r_\rho \in \RR(S)$ such that
$$
T_N \bigl( \rho ([K]) \bigr) =- \bigl( \Tr\, r_\rho(K) \bigr) \Id_E
$$
for every framed knot $K\subset S \times [0,1]$ whose projection to $S$ has no crossing and whose framing is vertical.  \qed
\end{thm}

The character  $r_\rho \in \RR$ is the \emph{classical shadow} of the irreducible representation $\rho \colon \SSS(S) \to \End(E)$. In \cite{BonWon3} we prove a stronger version of Theorem~\ref{thm:InvariantsExist}, which  is valid for all framed links $K\subset S \times [0,1]$ and involves the element $[K^{T_N}]\in \SSS(S)$ defined by threading the Chebyshev polynomial $T_N$ along all components of $K$. The above version is easier to state and sufficient for our purposes. 

See also \cite{Le} for a different approach to  the key results underlying Theorem~\ref{thm:InvariantsExist}.

The main result of this article is the following converse statement. 

\begin{thm}[Realization Theorem]
\label{thm:RealizeInvariantIntro}
Let $S$ be a closed oriented surface,  and let $A$ be a primitive $N$--root of $-1$ with $N$ odd. 
Then, every character $r\in \RR(S)$ is the classical shadow of an irreducible representation $\rho_r \colon \SSS(S) \to \End(E)$.\end{thm}

For the classical example of the Witten-Reshetikhin-Turaev representation $\rho_{\mathrm{WRT}} \colon \SSS(S) \to \End(W_{\mathrm{WRT}})$, also defined when $A$ is a primitive $2N$--root of unity with $N$ odd, the classical shadow  of $\rho_{\mathrm{WRT}}$ is the trivial character \cite{BonWon5}. We therefore construct a much broader family of representations of the skein algebra $\SSS(S)$ than this historic example. 

As explained in Theorem~\ref{thm:IndependentChoicesIntro} below, our construction is natural in the sense that it provides a representation  $\rho_r \colon \SSS(S) \to \End(E)$ that  depends only on the homomorphism $r\colon \pi_1(S) \to \SL(\C)$, up to isomorphism and other symmetries of the data. We conjecture that, when the character belongs to a Zariski dense open subset of the character variety $\RR(S)$, the representation $\rho_r$ is the only irreducible representation of $\SSS(S)$ with classical shadow $r\in \RR(S)$. This conjecture is proved by Nurdin Takenov for small punctured surfaces, such as the one-puncture torus or the four-puncture sphere \cite{Takenov}; it would be definitely false without the genericity hypothesis, as can for instance be proved by combining the results of \cite{HavPost} with the techniques of \cite{Takenov}.

The strategy for proving Theorem~\ref{thm:RealizeInvariantIntro} is somewhat unconventional. In addition to using classical hyperbolic geometry as a guide for quantum topology constructions, it  relies on the fact that punctured surfaces are easier to deal with than closed surfaces, and follows the slogan ``drill, baby, drill''\footnote{Popularized during the 2008 United  States presidential campaign  \cite[\S3]{Palin}, when the ideas behind this article were beginning to take shape.}. Namely, we drill punctures from the closed surface $S$ to obtain a punctured surface $S_\lambda$, by removing  from $S$ the vertices of a triangulation $\lambda$, and the more punctures the better.  If we are given a homomorphism $r \colon \pi_1(S) \to \SL(\C)$ representing the character $r\in \RR(S)$ and if the triangulation $\lambda$ is complicated enough, we can then choose  additional data at the punctures of $S_\lambda$ (called a \emph{$\lambda$--enhancement} of the homomorphism $r$) and  apply the results of \cite{BonWon4} to the punctured surface $S_\lambda$. This provides a representation $\rho_\lambda \colon \SSS(S_\lambda) \to \End(E_\lambda)$ of the skein algebra of the punctured surface $S_\lambda$, whose classical shadow is equal to the character $r_\lambda \in \RR(S_\lambda)$ induced by $r\in \RR(S)$ in the sense that 
$$
T_N \bigl( \rho_\lambda ([K]) \bigr) =- \bigl( \Tr\, r_\lambda(K) \bigr) \Id_{E_\lambda} =- \bigl( \Tr\, r(K) \bigr) \Id_{E_\lambda}
$$
for every framed knot $K\subset S_\lambda \times [0,1]$ whose projection to $S_\lambda$ has no crossing and whose framing is vertical. This last property, proved in \cite{BonWon4}, heavily relies on the miraculous cancellations of \cite{BonWon3}. 

However, there is no reason for $\rho_\lambda$ to induce a representation of the skein algebra $\SSS(S)$ of the closed surface $S$. Namely, if the two framed links $K$, $K'\subset S_\lambda \times [0,1]$ are isotopic in $S\times [0,1]$, by an isotopy sweeping through the punctures of $S_\lambda$, the two endomorphisms $ \rho_\lambda \bigl( [K] \bigr)$,  $ \rho_\lambda \bigl( [K'] \bigr) \in \End(E_\lambda)$ will in general be different. 
Nevertheless, we are able to identify a subspace $F_\lambda \subset E_\lambda$ where $ \rho_\lambda \bigl( [K] \bigr) $ and $  \rho_\lambda \bigl( [K'] \bigr)$ do coincide. This subspace $F_\lambda \subset E_\lambda$ is called the \emph{total off-diagonal kernel}, for reasons that are explained in  \S\S \ref{subsect:ClassicalOffDiagonalTerm} and \ref{subsect:OffDiagTerm}.

\begin{thm}
\label{thm:OffDiagonalKernelIntro}
Let the punctured surface $S_\lambda$ be obtained from the closed surface $S$ by removing the vertices of the triangulation $\lambda$ of $S$, and let
 $F_\lambda \subset E_\lambda$ be the total off-diagonal kernel of the representation $\rho_\lambda \colon \SSS(S_\lambda) \to \End(E_\lambda)$ introduced above. Then 
\begin{enumerate}
\item
$F_\lambda$ is invariant under the image $\rho_\lambda \bigl( \SSS(S_\lambda) \bigr) \subset \End(E_\lambda)$;

\item if  the two framed links $K$, $K'\subset S_\lambda \times [0,1]$ are isotopic in $S\times [0,1]$, the induced endomorphisms 
$$ \rho_\lambda \bigl( [K] \bigr)_{|F_\lambda}=   \rho_\lambda \bigl( [K'] \bigr)_{|F_\lambda} \in \End(F_\lambda)
$$
are equal. 
\end{enumerate}
\end{thm}

The definition of the total off-diagonal kernel $F_\lambda$ was devised by wishful thinking, as the largest subspace where the second conclusion of Theorem~\ref{thm:OffDiagonalKernelIntro} could hold. The really unexpected properties are that this subspace is non-trivial (see Theorem~\ref{thm:RepDimensionIntro} below) and that $F_\lambda$ is invariant under the image of $\rho_\lambda$. Indeed, although $\rho_\lambda = \mu_\lambda \circ \Tr_\lambda^\omega \colon \SSS(S_\lambda) \to \End(E_\lambda)$ is defined as a composition of the quantum trace homomorphism $\Tr_\lambda^\omega \colon \SSS(S_\lambda) \to \ZZ(\lambda)$ of \cite{BonWon1} with an \emph{irreducible} representation $\mu_\lambda \colon \ZZ(\lambda) \to \End(E_\lambda)$ of the balanced Chekhov-Fock algebra $\ZZ(\lambda)$ of the triangulation $\lambda$, the invariance of $F_\lambda$ shows that the representation $\rho_\lambda$ is  reducible. This reducibility property for $\rho_\lambda = \mu_\lambda \circ \Tr_\lambda^\omega$ would be false if we replaced $\mu_\lambda$ by an arbitrary irreducible representation of $\ZZ(\lambda)$. 

Theorem~\ref{thm:OffDiagonalKernelIntro} is proved in \S \ref{subsect:OffDiagKernelInvariant} and \S \ref{subsect:ConstructRepClosedSurface} when the triangulation $\lambda$ is sufficiently complicated, and in  \S \ref{subsect:ConstructSkeinRepArbitraryTriangulation} for general triangulations. 

A consequence of Theorem~\ref{thm:OffDiagonalKernelIntro} is that  the representation $\rho_\lambda \colon \SSS(S_\lambda) \to \End(E_\lambda)$ induces a representation $\check\rho_\lambda \colon \SSS(S) \to \End(F_\lambda)$ of the skein algebra of the closed surface $S$. This representation has a classical shadow equal to the character $r \in \RR(S)$, in the sense that 
$$
T_N \bigl( \check\rho_\lambda ([K]) \bigr) =- \bigl( \Tr\, r(K) \bigr) \Id_{F_\lambda}
$$
for every framed knot $K\subset S \times [0,1]$ whose projection to $S$ has no crossing and whose framing is vertical. 

The following property shows that our construction is very natural. For most homomorphisms $r \colon \pi_1(S) \to \SL(\C)$, this result states that the representation $\check\rho_\lambda $ is unique up to isomorphism. However, an additional ambiguity can arise for special characters that admit internal symmetries called sign-reversal symmetries. The cohomology group $H^1(S;\Z_2)$ acts on the character variety $\RR(S)$ and on the skein algebra $\SSS(S)$; see \S \ref{sect:signreversal}. A \emph{sign-reversal symmetry} for the character $r\in \RR(S)$ is a class $\epsilon \in H^1(S; \Z_2)$ that fixes $r$; characters that admit non-trivial sign-reversal symmetries are rare, and form a high codimension subset of the character variety $\RR(S)$. If $\rho \colon \SSS(S) \to \End(E)$ is a representation with classical shadow $r$, composing it with the action of a sign-reversal symmetry $\epsilon \in H^1(S;\Z_2)$ on $\SSS(S)$ gives another representation $\rho \circ \epsilon$ with classical shadow $\epsilon r = r \in \RR(S)$. Therefore, sign-reversal symmetries of $r\in \RR(S)$ are intrinsic symmetries of the problem of finding representations of $\SSS(S)$ with classical shadow $r$. 

\begin{thm}[Naturality Theorem]
\label{thm:IndependentChoicesIntro}
Up to isomorphism and up to sign-reversal symmetry of the character $r\in \RR(S)$ (if any exists), the representation $\check\rho_\lambda \colon \SSS(S) \to \End(F_\lambda)$ depends only on the group homomorphism $r \colon \pi_1(S) \to \SL(\C)$, not on the choice of the triangulation $\lambda$ or of the $\lambda$--enhancement $\xi$ used in the construction.
\end{thm}

In particular, although the dimension of $E_\lambda$  grows exponentially with  the number of punctures of the drilled surface $S_\lambda$, the dimension of the off-diagonal kernel $F_\lambda$ is independent of the topology  of $\lambda$. A consequence is that the construction is natural with respect to the action of the mapping class group of $S$. 

The proof of Theorem~\ref{thm:IndependentChoicesIntro}, given in \S \ref{subsect:IndependenceChoices}, relies on invariance under Pachner moves to go from one triangulation to another. It is a good illustration of the ``drill, baby, drill'' philosophy, as showing that two triangulations $\lambda$ and $\lambda'$  induce the isomorphic representations of $\SSS(S)$ usually involves surfaces with many more punctures than $S_\lambda$ and $S_{\lambda'}$. Here, the invariance under the face subdivision move considered in \S \ref{subsect:FaceSubdivisions}, which adds one vertex to the triangulation but does not change the representation, is probably the most surprising. 

Conjugating $r$ by an element of $\SL(\C)$ also leaves $\check\rho_\lambda$ unchanged, up to isomorphism. For a generic character $r\in\RR(S)$, two homomorphisms $\pi_1(S) \to \SL(\C)$ representing $r$ are always conjugate by an element of $\SL(\C)$ and therefore determine the same representation of $\SSS(S)$. However, for those special characters for which the property fails (namely reducible characters), we do not know if the representation $\check\rho_\lambda$ depends only on the character $r\in \RR(S)$, or on subtler properties of the specific homomorphism $\pi_1(S) \to \SL(\C)$ representing $r$ that we used in the construction. 

It is also quite possible that the need to consider sign-reversal symmetries is an artifact of our proof, and of its reliance on insights from the character variety $\RP(S)$. Indeed, the characters that admit non-trivial sign-reversal symmetries are precisely the branch points of the projection $\RR(S) \to \RP(S)$. It appears that composing our representation $\check\rho_\lambda$ with a sign-reversal symmetry of its classical shadow $r\in \RR(S)$ often produces a representation $\check\rho_\lambda \circ \epsilon$ that is isomorphic to $\check\rho_\lambda$, but we have not been able to confirm this fact in full generality. 

	At this point, we still have a major problem, which is that we do not know that the off-diagonal kernel $F_\lambda$ is different from 0. This property may even seem unlikely at first, as the off-diagonal kernel $F_\lambda$ is defined as an intersection of kernels of endomorphisms of the vector space $E_\lambda$. This question is addressed in \S \ref{bigsect:DimensionOffDiagonalKernel}, and provides another one of the surprising twists in this article. 

\begin{thm}
\label{thm:RepDimensionIntro}
If the closed oriented surface $S$ has genus $g$,  the representation $\check\rho_\lambda \colon \SSS(S) \to \End(F_\lambda)$ with classical shadow $r\in \RR(S)$ provided by Theorem~{\upshape \ref{thm:OffDiagonalKernelIntro}} has dimension
$$
\dim F_\lambda \geq
\begin{cases}
N^{3(g-1)} &\text{ if } g\geq 2\\
N &\text{ if } g=1\\
1 &\text{ if } g=0.
\end{cases}
$$ 
The above inequality is an equality for $r\in \RR(S)$ generic, namely for $r$ in an explicit Zariski dense open subset of $\RR(S)$.  
\end{thm}

In particular, the representation $\check\rho_\lambda \colon \SSS(S) \to \End(F_\lambda)$ is non-trivial. It may be reducible.   In fact, although we conjecture that $\check\rho_\lambda$ is irreducible for generic $r\in \RR(S)$, it is definitely reducible for highly non-generic homomorphisms $r \colon \pi_1(S) \to \SL(\C)$ such as the trivial homomorphism. However, restricting $\check\rho_\lambda$ to an irreducible component proves our main Theorem~\ref{thm:RealizeInvariantIntro}. 

We suspect that the inequalities of Theorem~\ref{thm:RepDimensionIntro} are always equalities. 
Our proof of Theorem~\ref{thm:RepDimensionIntro} departs from the ``drill, baby, drill'' and ``more punctures is better'' philosophy, and is based on a careful analysis of explicit triangulations $\lambda$ with a very small number of vertices. 

The subsequent article \cite{BonWon6} in the same series  proves an analogue of the Naturality Theorem~\ref{thm:IndependentChoicesIntro} for punctured surfaces. 

The results and methodology of this article were announced in \cite{BonWon2}. See the recent preprints \cite{AbdFroh1, AbdFroh2} for another construction of representations of $\SSS(S)$ with a given classical shadow $r\in \RR(S)$, valid for $r$ in a Zariski dense open subset of $\RR(S)$. The construction of \cite{AbdFroh1, AbdFroh2} is simpler, but ours is more explicit. In the few cases where the dimension of the representations of \cite{AbdFroh1, AbdFroh2} can be computed, these dimensions are significantly larger than those arising in the current article. We also believe that many of the ideas introduced in this paper are susceptible to have further applications in other contexts. 

\section{The Kauffman bracket skein algebra}

Let $S$ be an oriented surface of finite topological type without boundary. 
The \emph{Kauffman bracket skein algebra} $\SSS(S)$  depends on a parameter $A=\E^{-\pi\I \hbar}\in \C-\{0\}$, and is defined as follows: One first considers the vector space freely generated by  all isotopy classes of framed links in the thickened surface $S \times [0,1]$, and then one takes the quotient of this space by two relations. The first and main relation is the  \emph{skein relation}, which states that  
$$
[K_1] = A^{-1} [K_0] + A [K_\infty]
$$
whenever the three links $K_1$, $K_0$ and $K_\infty\subset S\times [0,1]$ differ only in a little ball where they are as represented in Figure~\ref{fig:SkeinRelation}, and where $[K]$ denotes the class of $\SSS(S)$ represented by the framed link $K$. The second relation is the \emph{trivial knot relation}, which asserts that 
$$[K\cup O] = -(A^2 +A^{-2})[K]$$
 whenever $O$ is the boundary of a disk $D \subset K \times [0,1]$ disjoint from $K$, and is endowed with a framing transverse to $D$. 
 
\begin{figure}[htbp]

\SetLabels
( .5 * -.4 ) $K_0$ \\
( .1 * -.4 )  $K_1$\\
(  .9*  -.4) $K_\infty$ \\
\endSetLabels
\centerline{\AffixLabels{\includegraphics{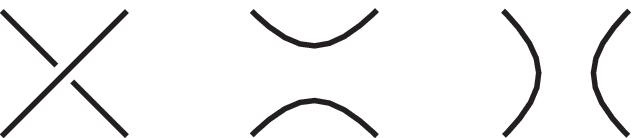}}}
\vskip 15pt
\caption{A Kauffman triple}
\label{fig:SkeinRelation}
\end{figure}

The algebra multiplication is provided by the operation of superposition, where the product $[K]\cdot [L]$ is represented by the union $[K'\cup L']$ where $K' \subset S\times [0,\frac 12]$ and $L' \subset S\times [\frac12, 1]$ are respectively obtained by rescaling the framed links $K \subset S\times [0,1]$ and $L' \subset S\times [0, 1]$.

\section{Sign-reversal symmetries}
\label{sect:signreversal}

The character variety $\RR(S)$ and the skein algebra $\SSS(S)$ both admit natural actions of the cohomology group $H^1(S;\Z_2)$. Indeed, for a character $r\in \RR(S)$ represented by a homomorphism  $r\colon \pi_1(S) \to \SL(\C)$ and a cohomology class $\epsilon \in H^1(S;\Z_2)$,  its image $\epsilon r \in \RR(S)$ is represented by the homomorphism $\epsilon r$ defined by
$$\epsilon r(\gamma) = (-1)^{\epsilon(\gamma)} r(\gamma) \in \SL(\C)$$
for every $\gamma \in \pi_1(S)$. The action of $H^1(S;\Z_2)$ on $\SSS(S)$ is similarly defined by the property that 
$$\epsilon[K] = [(-1)^{\epsilon(K)} K] \in \SSS(S)$$
  for every framed link $K \subset S\times[0,1]$ and $\epsilon \in H^1(S;\Z_2)$. 

If the character $r\in \RR(S)$ is fixed under the action of some  $\epsilon \in H^1(S;\Z_2)$, we say that  $\epsilon \in H^1(S;\Z_2)$ is a \emph{sign-reversal symmetry} for the character $r\in \RR(S)$. This is equivalent to the property that the trace $\Tr\,r(\gamma)$ is equal to $0$ for every $\gamma \in \pi_1(S)$ with $\epsilon(\gamma) \neq 0$.

Because of our assumption that $N$ is odd, the Chebyshev polynomial $T_N(x)$ is a sum of monomials of odd degree. It follows that, if the representation $\rho \colon \SSS(S) \to \End(E)$ has classical shadow $r\in \RR(S)$, its composition $\rho \circ \epsilon$ with the action of $\epsilon \in H^1(S;\Z_2)$ on $\SSS(S)$ has classical shadow $\epsilon r\in \RR(S)$. 
In particular, if the classical shadow $r\in \RR(S)$ of the representation $\rho \colon \SSS(S) \to \End(E)$ has a sign-reversal symmetry $\epsilon \in H^1(S;\Z_2)$, the representation $\rho \circ \epsilon$  also has classical shadow $\epsilon r=r$. Sign-reversal symmetries of a character $r\in\RR(S)$ are therefore intrinsic symmetries of the problem of finding representations of $\SSS(S)$ with classical shadow $r$, which explains why they will occur in  many statements of our article.

Characters with non-trivial sign-reversal symmetries exist, but are rare. For instance, they form an algebraic subset of complex dimension $2g-2$ in the $(6g-6)$--dimensional character variety $\RR(S)$, where $g$ is the genus of the surface $S$; see \cite[\S 5.1]{BonWon4}. 

\section{Constructing representations for punctured surfaces}

Throughout the article,  $A$ will be  a primitive $N$--root of $-1$ with $N$ odd. Namely, $A^N=-1$ and $N$ is the smallest positive integer with this property (and $N$ is odd). We  also use a choice of square root $\omega=\sqrt{A^{-1}}$. 

\subsection{The balanced Chekhov-Fock algebra of a triangulation}
\label{subsect:CheFock}

Let $\lambda$ be a triangulation of the closed oriented surface $S$. For most of the article, we are allowing an edge to go from one vertex to itself, as well as two edges to have the same endpoints. However,  we will always require that the sides of a face of $\lambda$ correspond to three distinct edges, for reasons that will become apparent in Remark~\ref{rem:EdgesUnivCoverDistinctEndpoints}. 

We will sometimes restrict attention to triangulations where each edge has distinct endpoints, and where distinct edges have distinct pairs of endpoints. In this case, we will say that the triangulation is \emph{combinatorial}, since this corresponds to the usual convention of combinatorial piecewise linear topology.  

Let $e_1$, $e_2$, \dots, $e_n$ be the edges of $\lambda$.   After choosing an auxiliary number $\omega$ such that $\omega^2 = A^{-1}$, the \emph{Chekhov-Fock algebra}  of $\lambda$ is the algebra $\TT(\lambda)$ defined by generators $Z_1^{\pm1}$, $Z_2^{\pm1}$, \dots, $Z_n^{\pm1}$ respectively associated to the edges  $e_1$, $e_2$, \dots, $e_n$ of $\lambda$, and by the relations
$$
Z_iZ_j = \omega^{2\sigma_{ij}} Z_jZ_i.
$$ 
where $\sigma_{ij}=a_{ij}-a_{ji}\in \{-2, -1, 0, 1, 2\}$ and where  $a_i \in \{0,1, 2\}$ is the number of times an end of the edge $e_j$ immediately succeeds an end of $e_i$ when going counterclockwise around a vertex of $\lambda$. 

An element of the Chekhov-Fock algebra $\TT(\lambda)$ is a linear combination of monomials $Z_{1}^{k_1}Z_{2}^{k_2} \dots Z_{n}^{k_n}$ in the generators $Z_i$, with $k_1$, $k_2$, \dots, $k_n \in \Z$. Such a monomial $Z_1^{k_1}Z_2^{k_2} \dots Z_n^{k_n}$ is \emph{balanced} if its exponents $k_i$  satisfy the following parity condition: for every triangle $T_j$ of the ideal triangulation $\lambda$, the sum $k_{i_1}+ k_{i_2}+ k_{i_3}$ of the exponents of the generators $Z_{i_1}$, $Z_{i_2}$, $Z_{i_3}$ associated to the sides of $T_j$ is even. 

 The \emph{balanced Chekhov-Fock algebra}  $\ZZ(\lambda)$ of the  triangulation $\lambda$ is the subalgebra of $\TT(\lambda)$ generated by all monomials satisfying this parity condition.
 
There are two reasons to be interested in the balanced Chekhov-Fock algebra $\ZZ(\lambda)$, whose combination is particularly useful for our purposes. The first one is the existence of an injective algebra homomorphism
$$
\Tr_\lambda^\omega \colon \SSS(S_\lambda) \to \ZZ(\lambda)
$$
from the skein algebra of the punctured surface $S_\lambda = S - V_\lambda$, obtained by removing from $S$ the set $V_\lambda$ of vertices of $\lambda$, to the algebra $\ZZ(\lambda)$; this \emph{quantum trace homomorphism} $\Tr_\lambda^\omega $ is constructed in \cite{BonWon1}. The second reason is that the algebraic structure of $\ZZ(\lambda)$ is fairly simple, so that its representations are easily classified (see \cite[\S 2]{BonWon4}, and the next section). This enables us to obtain representations of $\SSS(S_\lambda)$ by composing the quantum trace homomorphism $\Tr_\lambda^\omega $ with suitable representations of $\ZZ(\lambda)$. We will then show that these representations of the skein algebra $\SSS(S_\lambda)$ of the punctured surface $S_\lambda$ induce representations of the skein algebra $\SSS(S)$ of the closed surface $S$, which is the object of interest to us.

Because of the skew-commutativity relations $Z_iZ_j = \omega^{2\sigma_{ij}} Z_jZ_i$, the order of the variables in  a  monomial $Z_{i_1}^{n_1}Z_{i_2}^{n_2} \dots Z_{i_l}^{n_l}$ is quite important. We will make heavy use of the following symmetrization trick. 
The \emph{Weyl quantum ordering} for  $Z_{i_1}^{n_1}Z_{i_2}^{n_2} \dots Z_{i_l}^{n_l}$ is the monomial
$$
[Z_{i_1}^{n_1}Z_{i_2}^{n_2} \dots Z_{i_l}^{n_l}] = \omega^{-\sum_{u<v} n_un_v\sigma_{i_ui_v}} Z_{i_1}^{n_1}Z_{i_2}^{n_2} \dots Z_{i_l}^{n_l}. 
$$
The formula is specially designed that $[Z_{i_1}^{n_1}Z_{i_2}^{n_2} \dots Z_{i_l}^{n_l}]  \in \TT(\lambda)$ is invariant under any permutation of the~$Z_{i_u}^{n_u}$.

\subsection{Enhanced homomorphisms from $\pi_1(S)$ to $\SL(\C)$}
\label{subsect:Enhancements}
We are given a  character  $r\in \RR(S)$, represented by a homomorphism $ r \colon \pi_1(S) \to \SL(\C)$, and a triangulation $\lambda$ of the closed surface $S$.
Let $\widetilde S$ be the universal covering of $S$, and let $\widetilde \lambda$ be the triangulation of $\widetilde S$ induced by $\lambda$. Let $V_\lambda\subset S$ and $\widetilde V_\lambda \subset \widetilde S$ be the respective vertex sets of $\lambda$ and $\widetilde\lambda$. 

A \emph{$\lambda$--enhancement} for the group homomorphism $ r \colon \pi_1(S) \to \SL(\C)$ is a map $\xi \colon \widetilde V_\lambda \to \CP$ such that:
\begin{enumerate}
\item $\xi$ is $ r$--equivariant, in the sense that $\xi(\gamma \widetilde v) =  r (\gamma)\xi(\widetilde v)$ for every $\widetilde v \in \widetilde V_\lambda$ and every $\gamma \in \pi_1(S)$ (for the standard action of $\SL(\C)$ on the projective line $\CP$);
\item  for every edge $\widetilde e$  of $\widetilde \lambda$, the elements $\xi(\widetilde v)$ and $\xi(\widetilde v')\in \CP$ respectively associated to the end points $\widetilde v$ and $\widetilde v'$ of  $\widetilde e$ are distinct. 
\end{enumerate}

\begin{rem}
\label{rem:EdgesUnivCoverDistinctEndpoints}
Note an easy consequence of our assumption that the sides of each face of a triangulation $\lambda$ correspond to three distinct edges. If $e$ is an edge of $\lambda$ whose endpoints are equal to the same vertex $v$, an Euler characteristic argument shows that the closed loop formed by $e$ cannot bound a disk in $S$. As a consequence every edge $\widetilde e$ of $\widetilde\lambda$ has distinct endpoints, which makes Condition~(2) above more likely. Also, for the same reason, every edge $e$ of $\lambda$ whose endpoints are equal determines a non-trivial element of $\pi_1(S)$, well-defined up to conjugation. 
\end{rem}

\begin{lem}
\label{lem:EnhancementsExist}
Consider   a triangulation $\lambda$ of the surface $S$ and  a group homomorphism $r \colon \pi_1(S) \to \SL(\C)$ satisfying the following property: for every edge $e$ of $\lambda$ whose endpoints are equal to the same vertex $v$, the element $r(e)$ is different from $\pm \Id$  in $ \SL(\C)$.  Then the homomorphism  $ r \colon \pi_1(S) \to \SL(\C)$ admits a $\lambda$--enhancement. 
\end{lem}

Note that, in particular, the hypotheses of Lemma~\ref{lem:EnhancementsExist} are automatically satisfied  if every edge of $\lambda$ has distinct endpoints, or if $r$ is injective by Remark~\ref{rem:EdgesUnivCoverDistinctEndpoints}. 

\begin{proof}
To construct an $ r$--equivariant map $\xi \colon \widetilde V_\lambda \to \mathbb{CP}^1$, we  proceed orbit by orbit for the action of $\pi_1(S)$ on the vertex set $\widetilde V_\lambda$. 

For a vertex $v\in V_\lambda$, pick a point $\widetilde v \in \widetilde V_\lambda$ in its preimage. As a first approximation, define $\xi(\widetilde v)$ to be an arbitrary point of the projective line $\mathbb{CP}^1$. Then there is a unique way to $r$--equivariantly extend $\xi$ to the whole preimage of $v$, namely to the orbit $\pi_1(S) \widetilde v$ of $\widetilde v$ under the action of $\pi_1(S)$: define $\xi$ on this orbit by the property that $\xi(\gamma \widetilde v) =  r (\gamma)\xi(\widetilde v)$ for every $\gamma \in \pi_1(S)$. 

Performing this operation for each vertex $v$ of $\lambda$ defines an $r$--equivariant map  $\xi \colon \widetilde V_\lambda \to \mathbb{CP}^1$. 

In addition, we can require that, at each step, the initial point $\xi(\widetilde v) \in \mathbb{CP}^1$ is chosen to satisfy the following two conditions: $\xi(\widetilde v)$ is not in the image under $\xi$ of the orbits considered in earlier steps; for every edge $e$ of $\lambda$ whose endpoints are both equal to $v$, the point $\xi(\widetilde v)$ is not fixed by the image under $r$ of any conjugate of $e\in \pi_1(S)$. Because of our hypothesis that $r(e)\neq \pm \Id$ in the second case, these two conditions are easily satisfied by suitably choosing $\xi(\widetilde v) \in \mathbb{CP}^1$ outside of a countable number of forbidden values. 

It is then immediate that the map $\xi \colon \widetilde V_\lambda \to \mathbb{CP}^1$ so constructed in a $\lambda$--enhancement.
\end{proof}

A $\lambda$--enhancement $\xi \colon \widetilde V_\lambda \to \mathbb{CP}^1$ for the homomorphism $ r \colon \pi_1(S) \to \SL(\C)$ assigns a non-zero complex weight $x_i \in \C^*= \C-\{0\}$ to the $i$--th edge $e_i$ of $\lambda$ as follows. Lift $e_i$ to an edge $\widetilde e_i$ of the triangulation $\widetilde \lambda$ of the universal covering $\widetilde S$. Arbitrarily orient $\widetilde e_i$, and let $\widetilde v_i^+$ and $\widetilde v_i^-$ be the positive and negative endpoints of $\widetilde e_i$. Consider the two faces of $\widetilde\lambda$ that are adjacent to $\widetilde e_i$, let $\widetilde v_i^{\,\mathrm{left}}$ be the third vertex of the face to the left, and let $\widetilde v_i^{\,\mathrm{right}}$ be the third vertex of the face to the right. Then, $x_i$ is defined as minus the crossratio of the four points $\xi(\widetilde v_i^+)$, $\xi(\widetilde v_i^-)$, $\xi(\widetilde v_i^{\,\mathrm{left}})$, $\xi(\widetilde v_i^{\,\mathrm{right}})\in \mathbb{CP}^1$. More precisely, for the standard identification $\mathbb{CP}^1 \cong \C \cup \{ \infty \}$,
$$
x_i = - \frac
{\bigl(\xi(\widetilde v_i^{\,\mathrm{left}}) - \xi(\widetilde v_i^+)   \bigr)
\bigl(\xi(\widetilde v_i^{\,\mathrm{right}}) - \xi(\widetilde v_i^-)   \bigr)}
{\bigl(\xi(\widetilde v_i^{\,\mathrm{left}}) - \xi(\widetilde v_i^-)   \bigr)
\bigl(\xi(\widetilde v_i^{\,\mathrm{right}}) - \xi(\widetilde v_i^+)   \bigr)}
. $$
Note that reversing the orientation of $\widetilde e_i$  leaves $x_i$ unchanged. Also, the two conditions in the definition of $\lambda$--enhancements guarantee that  $x_i$ is a well-defined element of $\C^*$ and is independent of the choice  of the lift $\widetilde e_i$ of $e_i$, by invariance of crossratios under the action of $\SL(\C)$ on $\mathbb {CP}^1$.  

\subsection{Representations of the balanced Chekhov-Fock algebra}

We will use the results of \cite{BonWon4} to associate to each  group homomorphism $r\colon \pi_1(S) \to \SL(\C)$, endowed with a $\lambda$--enhancement $\xi \colon \widetilde V_\lambda \to \mathbb{CP}^1$, a representation  $\mu_\lambda \colon \ZZ(\lambda)\to  \End(E)$ of the balanced Chekhov-Fock algebra  $\ZZ(\lambda)$. 

This representation $\mu_\lambda$ will be uniquely determined up to isomorphism, but also up to sign-reversal symmetry of the character $r\in \RR(S)$. To make sense of this property, note that a monomial $Z_{\mathbf k} = Z_1^{k_1} Z_2^{k_2} \dots Z_n^{k_n} \in \ZZ(\lambda)$ uniquely determines a homology class $[\mathbf k] \in H_1(S_\lambda; \Z_2)$ in the punctured surface $S_\lambda$, by the property that the algebraic intersection number of $[\mathbf k]$ with each edge $e_i$ has the same parity as the exponent $k_i$ of the corresponding generator $Z_i$; see \cite[Lemma~9]{BonWon4}. A cohomology class $\epsilon \in H^1(S_\lambda; \Z_2)$ then acts on $\ZZZ(\lambda)$ by sending each $Z_{\mathbf k}$ to $(-1)^{\epsilon([\mathbf k])} Z_{\mathbf k}$. By restriction, this defines an action of $H^1(S;\Z_2)$ on $\ZZ(\lambda)$. 

Also, a vertex $v$ of $\lambda $ determines an element
 $$
 H_v =[Z_1^{k_1} Z_2^{k_2} \dots Z_n^{k_n}] = \omega^{-\sum_{i<j} k_i k_j \sigma_{ij}} Z_1^{k_1} Z_2^{k_2} \dots Z_n^{k_n} \in \ZZ(\lambda)
 $$
 where $k_i \in \{0,1,2\}$ is the number of endpoints of the edge $e_i$ that are equal to $v$, and where $[\phantom{M}]$ denotes the Weyl quantum  ordering defined in \S \ref{subsect:CheFock}.  This element $H_v$ is central in $\ZZ(\lambda)$, as proved in \cite[\S3]{BonLiu} or \cite[\S2.2]{BonWon4}. 
 
 A final observation is that the generator $Z_i \in \TT(\lambda)$ associated to the edge $e_i$ of $\lambda$ does not belong to the balanced Chekhov-Fock algebra $\ZZ(\lambda)$, as it does not satisfy the required exponent parity condition. However, its square $Z_i^2 \in \ZZ(\lambda)$ does. 
 
 We will make repeated use of the following result, borrowed from \cite{BonWon4}. 

\begin{prop}
\label{prop:ConstructRepCheFock}
For a  triangulation $\lambda$ of the surface $S$, consider a group homomorphism $r\colon \pi_1(S) \to \SL(\C)$ endowed with a $\lambda$--enhancement $\xi \colon \widetilde V_\lambda \to \mathbb{CP}^1$. Then,  up to isomorphism and up to the action of a sign-reversal symmetry of $r\in \RR(S)$ (if $r$ admits any), there exists a unique  representation $\mu_\lambda \colon \ZZ(\lambda) \to \End(E_\lambda)$ of the balanced Chekhov-Fock algebra $\ZZ(\lambda)$ with the following properties. 
\begin{enumerate}
\item The dimension of $E_\lambda$ is equal to $N^{3g+p_\lambda-3} $, where $g$ is the genus of the surface  $S$ and where $p_\lambda$ is the number of vertices of the triangulation $\lambda$.
\item For every edge $e_i$ of $\lambda$, let  $x_i \in \C^*$ be the crossratio weight associated to  $e_i$ by the enhancement $\xi$ as above, and let $Z_i$ be the corresponding generator of the Chekhov-Fock algebra $ \TT(\lambda)$. Then,
$$
\mu_\lambda (Z_i^{2N}) = x_i\, \Id_{E_\lambda}.
$$
\item For every vertex $v$ of $\lambda$ with associated central element $H_v \in \ZZ(\lambda)$, 
$$
\mu_\lambda(H_v) = -\omega^4 \,\Id_{E_\lambda}.  
$$
\item
The representation $\rho_\lambda = \mu_\lambda \circ \Tr_\lambda^\omega \colon \SSS(S_\lambda) \to \End({E_\lambda})$ has classical shadow $r \in \RR(S)$, in the sense that 
$$
T_N \bigl( \rho_\lambda ([K]) \bigr) = - \Tr\, r(K)\, \Id_{E_\lambda}
$$
for every framed knot $K \subset S_\lambda\times [0,1]$ whose projection to $S_\lambda$ has no crossing and whose framing is vertical. 
\end{enumerate}
In addition, $\mu_\lambda$ is irreducible.
\end{prop}
\begin{proof}
This is a special case of the combination of Propositions~22 and 23 of \cite{BonWon4}. The only minor difference is that these results are expressed in terms of pleated surfaces instead of $\lambda$--enhancements. 

To connect the two viewpoints, note that the triangulation $\lambda$ can also be interpreted as an ideal triangulation of the punctured surface $S_\lambda =S - V_\lambda$, obtained by removing from $S$ the vertex set $V_\lambda$ of $\lambda$. Similarly, the lift $\widetilde \lambda$ of $\lambda$ to  the universal covering $\widetilde S$ of $S$ gives an ideal triangulation of  the preimage $\widetilde S_\lambda= \widetilde S - \widetilde V_\lambda $ of $S_\lambda$ in  $\widetilde S$.  The $\lambda$--enhancement $\xi \colon \widetilde V_\lambda \to \mathbb{CP}^1$  then determines an $r$--equivariant pleated surface $\widetilde f_\lambda \colon \widetilde S_\lambda \to \HH^3$ that sends each face $\widetilde T$ of $\widetilde \lambda$ with vertices $\widetilde v_1$, $\widetilde v_2$, $\widetilde v_3 \in \widetilde V_\lambda$ to  the ideal triangle $\widetilde f_\lambda(\widetilde T) \subset \HH^3$ with vertices $\xi(\widetilde v_1)$, $\xi(\widetilde v_2)$, $\xi(\widetilde v_3) \in \CP = \partial_\infty \HH^3$. We can then lift $\widetilde f_\lambda$ to an $r_\lambda$--equivariant pleated surface $\widehat f_\lambda \colon \widehat S_\lambda \to \HH^3$, where $\widehat S_\lambda$ is the universal cover of the punctured surface $S_\lambda$ and where  $r_\lambda \colon \pi_1(S_\lambda) \to \SL(\C)$  is the composition of $r\colon \pi_1(S) \to \SL(\C)$  with the homomorphism $\pi_1(S_\lambda) \to \pi_1(S)$ induced by the inclusion map.

The pleated surface $(\widehat f_\lambda, r_\lambda)$ is exactly the setup needed to apply Proposition~23 of \cite{BonWon4} to the punctured surface $S_\lambda$. By construction, the shearbend parameter associated by this pleated surface to the edge $e_i$ of $\lambda$ is exactly the crossratio weight $x_i$ defined as above by the $\lambda$--enhancement $\xi$. 

Proposition~23 of \cite{BonWon4} has an additional degree of freedom for each puncture $v$ of $S_\lambda$. Specifically, the hypotheses of that statement require that we choose an $N$--root $h_v = \mu_\iota(H_v)^{\frac1N}$ for a certain number $ \mu_\iota(H_v)\in \C^*$ provided by \cite[Proposition~22]{BonWon4} (using the notation of \cite{BonWon4}). In addition, this number $ \mu_\iota(H_v)$ is such that $ \mu_\iota(H_v) +  \mu_\iota(H_v)^{-1} = - \Tr\,r(P_v)$, where $P_v$ is a small loop going around the puncture $v$ of $S_\lambda$.  In our case $r_\lambda(P_v)$ is the identity and consequently has trace equal to 2, so that $ \mu_\iota(H_v)=-1$. 

We can therefore apply \cite[Proposition~23]{BonWon4} to the $N$--root  $h_v=-\omega^4$ of $ \mu_\iota(H_v)=-1$, since $N$ is odd and $\omega^{4N} = A^{-2N}=1$. This provides a representation  $\mu_\lambda \colon \ZZ(\lambda) \to \End(E_\lambda)$  satisfying the conclusions of Proposition~\ref{prop:ConstructRepCheFock}. 

The uniqueness parts of  Propositions~22 and 23 of \cite{BonWon4}  show that $\mu_\lambda$ is unique up to isomorphism and up to the action of a sign-reversal symmetry $\epsilon_\lambda \in H^1(S_\lambda; \Z_2)$ of the restriction $r_\lambda \in \RR(S_\lambda)$ of $r\in \RR(S)$. For every puncture $v$ of $S_\lambda$, $\Tr\,r_\lambda(P_v)=2\neq 0$ and the sign-reversal symmetry $\epsilon_\lambda$ is consequently trivial on the loop $P_v$ going around $v$. It follows that $\epsilon_\lambda$ is the restriction of a sign-reversal symmetry $\epsilon \in H^1(S; \Z_2)$ of  $r \in \RR(S_\lambda)$. This proves the uniqueness statement for the representation  $\mu_\lambda \colon \ZZ(\lambda) \to \End(E_\lambda)$.
\end{proof}

\begin{rem}
As indicated in the above discussion, we could have replaced Conclusion~(3) of Proposition~\ref{prop:ConstructRepCheFock} by the property that $\mu_\lambda(H_v) = h_v \, \Id_{E_\lambda}$ for an arbitrary $N$--root $h_v$ of $-1$. However our subsequent applications of Proposition~\ref{prop:ConstructRepCheFock} require that $h_v=-\omega^4$ in a crucial way. 
\end{rem}

\begin{comp}
\label{comp:ConstructRepCheFockContinuous}
The representation $\mu_\lambda$ of Proposition~{\upshape\ref{prop:ConstructRepCheFock}} continuously depends on the enhanced homomorphism $(r, \xi)$ as follows. For each edge $e_i$ of $\lambda$, consider the corresponding crossratio weight $x_i\in \C-\{0\}$ as a function of the pair $(r, \xi)$, and let  $u_i=\sqrt[2N]{x_i}$ be a local determination of the $2N$--root of $x_i$ defined for $(r, \xi)$ in an  open subset $\mathcal U$ of the space of all such pairs.  Then the representation $\mu_\lambda \colon \ZZ(\lambda) \to \End({E_\lambda})$ can be chosen so that, for every monomial $Z_1^{k_1}Z_2^{k_2} \dots Z_n^{k_n} \in \ZZ(\lambda)$, 
$$
\mu_\lambda (Z_1^{k_1}Z_2^{k_2} \dots Z_n^{k_n}) = u_1^{k_1}u_2^{k_2} \dots u_n^{k_n} \, A_{k_1k_2\dots k_n} 
$$
for some linear isomorphism $A_{k_1k_2\dots k_n} \in \End({E_\lambda})$ independent of $(r,\xi)\in \mathcal U$. 
\end{comp}
\begin{proof}
This is an immediate consequence of the proofs of  Propositions~15, 22 and 23 in \cite{BonWon4} (where Proposition~15 is a key step in the proof of Proposition~23).
\end{proof}

\subsection{Representing the skein algebra of the punctured surface $S_\lambda$}
\label{subsect:RepresentingPunctSurfSkein}
We now begin our construction of an irreducible representation of the skein algebra $ \SSS(S)$ whose classical shadow is equal to the character  $r\in \RR(S)$.

Represent the character $r\in \RR(S)$ by a group homomorphism $r\colon \pi_1(S) \to \SL(\C)$. Let $\lambda$ be a triangulation of $S$ for which this homomorphism $r$ admits a $\lambda$--enhancement $\xi$. For instance, any combinatorial triangulation has this property by Lemma~\ref{lem:EnhancementsExist}. Let $S_\lambda = S- V_\lambda$ be the punctured surface obtained by removing the vertex set of $\lambda$ from $S$. 

We can then consider the representation $\mu_\lambda \colon \ZZ(\lambda) \to \End({E_\lambda})$ associated to the enhanced homomorphism $(r, \xi)$ by Proposition~\ref{prop:ConstructRepCheFock}. Composing $\mu_\lambda$ with the quantum trace homomorphism $\Tr_\lambda^\omega \colon \SSS(S_\lambda) \to \ZZ(\lambda)$ of \cite{BonWon1} now defines a representation
$$
\rho_\lambda = \mu_\lambda \circ \Tr_\lambda^\omega \colon \SSS(S_\lambda) \to \End({E_\lambda}).
$$

This is only a representation of the skein algebra $\SSS(S_\lambda)$ of the punctured surface $S_\lambda$, whereas we want to represent the skein algebra $\SSS(S)$ of the closed surface $S$. The rest of the article is devoted to showing how $\rho_\lambda$ induces  a non-trivial representation of $\SSS(S)$. 

\section{The off-diagonal kernels}

\subsection{The classical off-diagonal term of a vertex}
\label{subsect:ClassicalOffDiagonalTerm}

This section is intended to motivate the definition of the next section.

Consider a vertex $v$ of the triangulation $\lambda$. Let $e_{i_1}$, $e_{i_2}$, \dots, $e_{i_u}$  be the edges of $\lambda$ that emanate from $v$, indexed in counterclockwise  order around $v$, and with possible repetitions when the two endpoints of an edge are equal to $v$. As in \S \ref{subsect:Enhancements}, let $x_i\in \C^*$ be the crossratio weight associated to the edge $e_i$ of $\lambda$ by the enhancement $\xi$.

\begin{lem}
\label{lem:ClassicalOffDiagonal}
$$
1+ x_{i_1} + x_{i_1} x_{i_2} + \dots + x_{i_1} x_{i_2} \dots x_{i_{u-1}} =0
$$
\end{lem}

\begin{proof}
Let $P_v$ be a small  loop going around the vertex $v$, oriented counterclockwise. A standard computation (see for instance Exercises~8.5--8.7 and 10.14 in \cite{BonBook}) enables us to compute the image of any element of $ \pi_1(S_\lambda)$ under the homomorphism $ r_\lambda \colon \pi_1(S_\lambda) \to \SL(\C)$ induced by $r$, namely the homomorphism $ r_\lambda$ obtained by composing  $r\colon \pi_1(S) \to \SL(\C)$  with the homomorphism $\pi_1(S_\lambda)  \to \pi_1(S)$ induced by the inclusion map. For $P_v$, this gives that, up to conjugation,
\begin{align*}
 r_\lambda (P_k) &= \pm
\begin{pmatrix}
1&1\\0&1
\end{pmatrix}
\begin{pmatrix}
z_{i_1} & 0 \\ 0 & z_{i_1}^{-1}
\end{pmatrix}
\begin{pmatrix}
1&1\\0&1
\end{pmatrix}
\begin{pmatrix}
z_{i_2} & 0 \\ 0 & z_{i_2}^{-1}
\end{pmatrix}
\dots
\begin{pmatrix}
1&1\\0&1
\end{pmatrix}
\begin{pmatrix}
z_{i_u} & 0 \\ 0 & z_{i_u}^{-1}
\end{pmatrix}\\
&= \pm
\begin{pmatrix}
z_{i_1}z_{i_2}\dots z_{i_u} & \sum_{j=1}^{u}  z_{i_1} z_{i_2} \dots z_{i_{j-1}}z_{i_j}^{-1}z_{i{j+1}}^{-1} \dots z_{i_u}^{-1} \\
0 & z_{i_1}^{-1}  z_{i_2}^{-1} \dots z_{i_u}^{-1}
\end{pmatrix}
\in \SL(\C)
\end{align*}
for arbitrary choices of square roots $z_i = \sqrt{x_i}$. The $\pm$ sign depends on these choices of square roots. 

Since $P_v$ is homotopic to 0 in $S$, $ r_\lambda (P_v) =  r(P_v) = \Id \in \SL(\C)$.  Therefore, $z_{i_1} z_{i_2} \dots z_{i_u} = \pm 1$ and the off-diagonal term
\begin{align*}
 \sum_{j=1}^{u}  z_{i_1} z_{i_2} \dots z_{i_{j-1}}  z_{i_j}^{-1}  z_{i{j+1}}^{-1} \dots z_{i_u}^{-1}
&= z_{i_1}^{-1}  z_{i_2}^{-1} \dots z_{i_u}^{-1} 
\biggl(
 \sum_{j=1}^{u}  z_{i_1}^2 z_{i_2}^2 \dots z_{i_{j-1}}^2
\biggr)\\
&= \pm
\biggl(
1+ z_{i_1}^2 + z_{i_1}^2z_{i_2}^2 + \dots + z_{i_1}^2z_{i_2}^2 \dots z_{i_{u-1}}^2
\biggr)\\
&= \pm
\biggl(
1+ x_{i_1} + x_{i_1} x_{i_2} + \dots + x_{i_1} x_{i_2} \dots x_{i_{u-1}}
\biggr)
\end{align*}
is equal to 0. 
\end{proof}

We will consider a quantum analogue of the equation
$$
1+ x_{i_1} + x_{i_1} x_{i_2} + \dots + x_{i_1} x_{i_2} \dots x_{i_{u-1}} =0
$$
or, equivalently,
$$
1+ z_{i_1}^2 + z_{i_1}^2z_{i_2}^2 + \dots + z_{i_1}^2z_{i_2}^2 \dots z_{i_{u-1}}^2
=0
$$
for the representation $\mu_\lambda \colon \ZZ(\lambda) \to \End({E_\lambda})$ of Proposition~\ref{prop:ConstructRepCheFock}. The major difference is that this equation will not be realized everywhere, but only on a subspace $F_v$ of ${E_\lambda}$.

\subsection{The off-diagonal term and kernel of a vertex}
\label{subsect:OffDiagTerm}

As in the previous section, we consider a vertex $v$ of the triangulation $\lambda$, and we  index the edges of $\lambda$ emerging from $v$ as $e_{i_1}$, $e_{i_2}$, \dots, $e_{i_u}$ in counterclockwise order around $v$.

Note that the indexing of the $e_{i_j}$ depends on our choice of the first edge $e_{i_1}$. For this choice of indexing, the \emph{off-diagonal term} of the puncture $v$ is the element
\begin{align*}
Q_v &=
 \sum_{j=0}^{u-1} \omega^{-4j} Z_{i_1}^2 Z_{i_2}^2 \dots Z_{i_{j}}^2\\
 &=
1+ \omega^{-4} Z_{i_1}^2 + \omega^{-8} Z_{i_1}^2Z_{i_2}^2 + \dots + \omega^{-4(u-1)} Z_{i_1}^2Z_{i_2}^2 \dots Z_{i_{u-1}}^2
\end{align*}
of $\ZZ(\lambda)$. 

For the representation $\mu_\lambda \colon \ZZ(\lambda) \to \End({E_\lambda})$ of Proposition~\ref{prop:ConstructRepCheFock}, the  \emph{off-diagonal kernel} of the vertex $v$ for the representation $\mu_\lambda$ is the subspace $F_v = \ker \mu_\lambda(Q_v)$ of ${E_\lambda}$. To relate this definition to the relation of Lemma~\ref{lem:ClassicalOffDiagonal}, observe  that the off-diagonal kernel $F_v$ is the set of vectors $w \in E_\lambda$ such that
$$
\mu_\lambda \bigl(1+ \omega^{-4} Z_{i_1}^2 + \omega^{-8} Z_{i_1}^2Z_{i_2}^2 + \dots + \omega^{-4(u-1)} Z_{i_1}^2Z_{i_2}^2 \dots Z_{i_{u-1}}^2 \bigr)(w)=0.
$$
Note the analogy with the last displayed equation of \S \ref{subsect:ClassicalOffDiagonalTerm}.

The \emph{total off-diagonal kernel} of $\mu_\lambda$ is the intersection ${F_\lambda}= \bigcap_{v\in V_\lambda} F_v$ of the off-diagonal kernels of all  vertices of $\lambda$. 

The off-diagonal term $Q_v\in \ZZ(\lambda)$ clearly depends of the indexing of the edges of $\lambda$  around $v$. We will show in Lemma~\ref{lem:OffDiagKernelWellDefined} below that, on the contrary,  the off-diagonal kernel $F_v \subset {E_\lambda}$   depends only on the vertex $v$. 
As a first step towards the proof of that statement, we begin with a preliminary lemma. 

By invariance of the Weyl quantum ordering under permutation, the central element $H_v \in \ZZ(\lambda)$ associated to the vertex $v$ can  be written as  $H_v = [Z_{i_1}Z_{i_2}\dots Z_{i_u}]$. We want to compute the precise quantum ordering coefficient in this expression. 

\begin{lem}
\label{lem:QuantumOrderH}
 Let the edges of $\lambda$ emerging from the vertex $v$ be indexed as $e_{i_1}$, $e_{i_2}$, \dots, $e_{i_u}$ in counterclockwise order around $v$. Then, the central element $H_v \in \ZZ(\lambda)$ associated to $v$ is equal to  
$$H_v = \omega^{-u+2} Z_{i_1}Z_{i_2}\dots Z_{i_u}.$$
\end{lem}

\begin{proof}
The proof is straightforward when the edges $e_{i_k}$ are all distinct, and in particular when the triangulation $\lambda$ is combinatorial. Indeed, in this case, $Z_{i_k}Z_{i_{k+1}} = \omega^2 Z_{i_{k+1} } Z_{i_k}$ whenever $1\leq k <u$, $Z_{i_1}Z_{i_u} = \omega^2 Z_{i_u} Z_{i_1}$, and all other pairs of generators $Z_{i_k}$, $Z_{i_l}$ commute. 
The general case could be deduced from this one with the change of triangulation techniques developed in \S \S \ref{subsect:FaceSubdivisions} and \ref{subsect:DiagonalExchanges}, but we prefer to give a combinatorial proof right away. See also the very indirect argument that we used in the proof of  \cite[Lemma~18]{BonWon4}.

By definition of the Weyl quantum ordering,
$$
H_v=
[Z_{i_1}Z_{i_2}\dots Z_{i_u}]  = \omega^{-\sum_{1\leq k<l \leq u} \sigma_{i_ki_l}}
Z_{i_1}Z_{i_2}\dots Z_{i_u}
$$
where the skew commutativity coefficient $\sigma_{ij} \in \{0, \pm1, \pm2\}$ is defined as in \S \ref{subsect:CheFock}, and in particular is such that $Z_iZ_j = \omega^{2\sigma_{ij}} Z_jZ_i$.

By definition of $\sigma_{ij}$, we can write
$$
\sum_{1\leq k<l\leq u} \sigma_{i_ki_l} =
\sum_{l=2}^{u}  \sum_{\substack{1\leq k<  l\text{ and}\\
a\text{ angular sector}\\
\text{from }e_{i_k} \text{ to } e_{i_l}}}
\kern -0pt  \epsilon(k,l,a)
$$
where, for every angular sector $a$ of a triangle $T_j$ that is locally bounded by the edges $e_{i_k}$ and $e_{i_l}$ near the vertex of $a$, $\epsilon(k,l,a)$ is equal to $+1$ if one goes from  $e_{i_k}$ to $e_{i_l}$ counterclockwise at $a$, and is equal to $-1$ otherwise. The angular sectors $a$ contributing to this sum include the angular sectors $a_{1}$, $a_{2}$, \dots, $a_{u}$ that are adjacent to $v$, indexed  in such a way that $a_{k}$ is locally bounded by the edges $e_{i_k}$ and $e_{i_{k+1}}$ near $v$.  There may be contributions from additional angular sectors when the edges $e_{i_k}$ are not distinct.

Fixing an index $l$, we want to analyze its contribution $\sigma(l)=\sum_{k, a}  \epsilon(k,l,a)$  to the above sum.  
If an index $k$  contributes to $\sigma(l)$, then the edge $e_{i_k}$  is contained in a face of $\lambda$ that also contains $e_{i_l}$, and one of the two edges $e_{i_{k\pm1}}$ is also contained in the same face. Analyzing the possible configurations in the union of the two faces of $\lambda$ containing $e_{i_l}$, we see that most of the couples $(k, a)$ contributing to the sum can be grouped into pairs
\begin{enumerate}
\item $\{ (k, a_k'), (k+1, a_k'') \}$ when the angular sector $a_k$ is opposite $e_{i_l}$ in a face  of $\lambda$, and where $a_k'$ and $a_k''$ are the other two angular sectors of this face;
\item $\{( k-1, a_{k-1}), (k+1, a_k) \}$ when $e_{i_k} = e_{i_l}$.
\end{enumerate}
The first type of pair $\{ (k, a_k'), (k+1, a_k'') \}$ contributes $ \epsilon(k,l,a_k') +  \epsilon(k+1,l,a_k'')=+1-1 =0$. The second type $\{( k-1, a_{k-1}), (k+1, a_k) \}$ contributes $ \epsilon(k,l-1,a_{l-1}) +  \epsilon(k,l+1,a_l)=1-1 =0$. In particular, the corresponding terms cancel out. 

The only terms that do not cancel out in such a pair are those where the potential pair would involve an index that is not in the interval $[1, l-1]$. This always occurs for $\epsilon(l-1, l, a_{l-1})=+1$, and for $\epsilon(1, u, a_u)=-1$ when $l=u$. A more special instance arises when the angular sector $a_u$ is opposite $e_{i_l}$ in a face of $\lambda$, in which case $\epsilon(1,l, a_u'')=1$ cannot be cancelled by a term $\epsilon(u,l, a_u')=-1$. Similarly, when $e_{i_l}=e_{i_1}$ or $e_{i_l} =e_{i_u}$ with $l<u$, the terms $\epsilon(2, l, a_1)=-1$ or $\epsilon(1, l, a_u) =-1$, respectively, are not cancelled by another term.

Using our convention that the three sides of each face of the triangulation $\lambda$ are all distinct, one easily sees that these are the only terms that do not cancel out. Note that, outside of $\epsilon(l-1, l, a_{l-1})=+1$, all the other exceptions occur precisely when the face of $\lambda$ containing the angular sector $a_u$ also contains the edge $e_{i_l}$. Summing over $l$ and combining the above observations, it follows that
$$
\sum_{1\leq k<l\leq u} \sigma_{i_ki_l} = u-1 + n_1 - n_2
$$
where $n_1 \in \{0,1,2\}$ is the number of indices $l\in [2, u]$ for which the edge $e_{i_l}$ is opposite the angular sector $a_u$ in the face that contains it, and $n_2 \in \{1,2, 3 \}$ is the number of indices $l\in [2, u]$ for which the edge $e_{i_l}$ is adjacent to $a_u$. 

Now, consider the face of $\lambda$ that contains the angular sector $a_u$. There are three cases to consider, according to whether 1, 2 or 3 of the vertices of this face are equal to  $v$. An immediate count gives that $n_2=n_1+1$ in all three cases.  This proves that 
$$
\sum_{1\leq k<l\leq u} \sigma_{i_ki_l} =  u-2
$$
and completes the proof  of Lemma~\ref{lem:QuantumOrderH}.
\end{proof}

This combinatorial proof of Lemma~\ref{lem:QuantumOrderH} also enables us to compute the Weyl quantum ordering of monomials similar to the central element $H_v$. These computations will be used in \S \ref{bigsect:DimensionOffDiagonalKernel}. 

\begin{lem}
\label{lem:QuantumOrderAtVertex}
Let the edges of $\lambda$ emerging from the vertex $v$ be indexed as $e_{i_1}$, $e_{i_2}$, \dots, $e_{i_u}$ in counterclockwise order around $v$. Then, for every $k_0$ with $1<k_0<u$, 
$$
[Z_{i_1} Z_{i_2} \dots, Z_{i_{k_0}}]=
\begin{cases}
\omega^{-k_0+1} Z_{i_1} Z_{i_2} \dots Z_{i_{k_0}} &\text{ if } e_{i_{k_0}}\neq e_{i_1} \text { and } e_{i_{k_0+1}}\neq e_{i_u}\\
\omega^{-k_0+2} Z_{i_1} Z_{i_2} \dots Z_{i_{k_0}} &\text{ if } e_{i_{k_0}}= e_{i_1} \\
\omega^{-k_0} Z_{i_1} Z_{i_2} \dots Z_{i_{k_0}} &\text{ if }  e_{i_{k_0+1}}= e_{i_u}.
\end{cases}
$$

In particular, $[Z_{i_1} Z_{i_2} \dots, Z_{i_{k_0}}]= \omega^{-k_0+1} Z_{i_1} Z_{i_2} \dots, Z_{i_{k_0}}$ if the triangulation $\lambda$ is combinatorial. 
\end{lem}

\begin{proof}
By definition of the Weyl quantum ordering,
$$
[Z_{i_1}Z_{i_2}\dots Z_{i_{k_0}}]  = \omega^{-\sum_{1\leq k<l \leq k_0} \sigma_{i_ki_l}}
Z_{i_1}Z_{i_2}\dots Z_{i_{k_0}}.
$$
The same arguments as in the proof of Lemma~\ref{lem:QuantumOrderH} then give that
$$
\sum_{1\leq k<l \leq k_0} \sigma_{i_ki_l} = k_0 -1 + n_1 - n_2
$$
where $n_1 \in \{0,1,2\}$ is the number of indices $l\in [2, k_0]$ for which the edge $e_{i_l}$ is opposite the angular sector $a_u$ in a face of $\lambda$, and $n_2 \in \{0,1, 2 \}$ is the number of indices $l\in [2, k_0]$ for which the edge $e_{i_l}$ is adjacent to $a_u$. 

The fact that indices are   truncated at $k_0$  introduces minor differences with  the case of  Lemma~\ref{lem:QuantumOrderH}. More precisely, the case-by-case analysis now gives that
$$
n_2 =
\begin{cases}
n_1 &\text{ if } e_{i_{k_0}}\neq e_{i_1} \text { and } e_{i_{k_0+1}}\neq e_{i_u}\\
n_1 +1  &\text{ if } e_{i_{k_0}}= e_{i_1} \\
n_1 -1 &\text{ if }  e_{i_{k_0+1}}= e_{i_u}.
\end{cases}
$$
The stated computation immediately follows.
\end{proof}




We are now ready to prove the promised result, that the off-diagonal kernel $F_v \subset E_\lambda$ depends only on the vertex $v$. 

\begin{lem}
\label{lem:OffDiagKernelWellDefined}
The off-diagonal kernel $F_v = \ker  \mu_\lambda(Q_v) $ of  $v$ is independent of the counterclockwise  indexing of the edges $e_{i_1}$, $e_{i_2}$, \dots, $e_{i_u}$, $e_{i_{u+1}} = e_{i_1}$ of $\lambda$ around the vertex $v$. 
\end{lem}

\begin{proof}
We can clearly restrict attention to the case where we shift the indexing by 1, and start at the last edge $e_{i_u}$ instead. Then the off-diagonal term 
$$
Q_v = 1+ \omega^{-4} Z_{i_1}^2 + \omega^{-8} Z_{i_1}^2Z_{i_2}^2 + \dots + \omega^{-4(u-1)} Z_{i_1}^2Z_{i_2}^2 \dots Z_{i_{u-1}}^2
$$ 
gets replaced by
\begin{align*}
Q_v' &= 1+ \omega^{-4} Z_{i_u}^2 + \omega^{-8} Z_{i_u}^2Z_{i_1}^2 +\omega^{-12} Z_{i_u}^2Z_{i_1}^2 Z_{i_2}^2 +  \dots + \omega^{-4(u-1)} Z_{i_u}^2 Z_{i_1}^2Z_{i_2}^2 \dots Z_{i_{u-2}}^2\\
&=  1+ \omega^{-4} Z_{i_u}^2 Q_v -  \omega^{-4u} Z_{i_u}^2 Z_{i_1}^2Z_{i_2}^2 \dots Z_{i_{u-2}}^2 Z_{i_{u-1}}^2.
\end{align*}

We now have to remember that $\mu_\lambda(H_v) = -\omega^4 \, \Id_{E_\lambda}$ by choice of the representation $\mu_\lambda$ in Proposition~\ref{prop:ConstructRepCheFock},
so that
$$
\mu_\lambda(\omega^{-4u} Z_{i_u}^2 Z_{i_1}^2Z_{i_2}^2 \dots Z_{i_{u-2}}^2 Z_{i_{u-1}}^2) = \mu_\lambda(\omega^{-8} H_v^2) = \Id_{E_\lambda}
$$
by Lemma~\ref{lem:QuantumOrderH}. 
It follows that 
$$
\mu_\lambda(Q_v') = \Id_{E_\lambda} + \omega^{-4} \mu_\lambda(Z_{i_u}^2) \circ \mu_\lambda(Q_v) -Id_{E_\lambda} =  \omega^{-4} \mu_\lambda(Z_{i_u}^2) \circ \mu_\lambda(Q_v) .
$$

The element $Z_{i_u}^2$ is invertible in $\ZZ(\lambda)$. Therefore $\mu_\lambda(Z_{i_u}^2)$ is invertible in $\End({E_\lambda})$ and it follows that $\ker \mu_\lambda(Q_v') = \ker \mu_\lambda(Q_v)$, as desired.
\end{proof}

\subsection{Invariance under the action of  the skein algebra}
\label{subsect:OffDiagKernelInvariant}

The off-diagonal kernel $F_v\subset {E_\lambda}$ cannot be invariant under $\mu_\lambda \bigl( \ZZ(\lambda) \bigr)$, since the representation $\mu_\lambda \colon \ZZ(\lambda) \to \End({E_\lambda})$ is irreducible. However, it is invariant under the image of the representation  $\rho_\lambda = \mu_\lambda \circ \Tr_\lambda^\omega \colon \SSS(S_\lambda) \to \End({E_\lambda}) $. 

In this section, we restrict attention to the case where the triangulation $\lambda$ is combinatorial. We will later see, in \S \ref{subsect:ConstructSkeinRepArbitraryTriangulation}, that the property holds without this hypothesis. 

\begin{prop}
\label{prop:OffDiagKernelInvariant}
Suppose that the triangulation $\lambda$  is combinatorial, in the sense that every edge has distinct endpoints and that no two distinct edges have the same endpoints. Then, the off-diagonal kernel $F_v$ of each vertex $v$ of $\lambda$ is invariant under $\rho_\lambda \bigl( \SSS(S_\lambda) \bigr ) \subset \End({E_\lambda})$. 
\end{prop}

\begin{proof}
Let $\mathcal N(v)\subset S$ be the  neighborhood of the vertex $v$ that is the union  of the faces of $\lambda$ containing $v$. Because of our hypothesis that $\lambda$ is combinatorial,  there are no identifications on the boundary of $\mathcal N(v)$, and $\mathcal N(v)$ is homeomorphic to a disk.  We   already indexed the edges of $\lambda$ emanating from $v$ as $e_{i_1}$, $e_{i_2}$, \dots, $e_{i_u}$, going counterclockwise around $v$. Let $e_{k_1}$, $e_{k_2}$, \dots, $e_{k_u}$ denote the edges forming the boundary of the star neighborhood $\mathcal N(v)$, in such a way that $e_{i_j}$, $e_{i_{j-1}}$ and $e_{k_j}$ cobound a face of $\lambda$. See Figure~\ref{fig:VertexStar}.

\begin{figure}[htbp]

\SetLabels
(.665 *.22 ) $ e_{i_1}$ \\
( .48* .22) $ e_{i_u}$ \\
( .78* .37) $e_{i_2} $ \\
( .5* -.05) $ e_{k_1}$ \\
( .85* .07) $ e_{k_2}$ \\
( .17* .07) $ e_{k_u}$ \\
( .36*.73 ) $ e_{i_{j_l}}$ \\
( .2* .7) $a_l $ \\
(.32 *.99 ) $ e_{k_{j_l}}$ \\
( .95*.78 ) $ K$ \\
( -.04* .7) $e_{k_{j_l+1}} $ \\
\endSetLabels
\centerline{\AffixLabels{\includegraphics{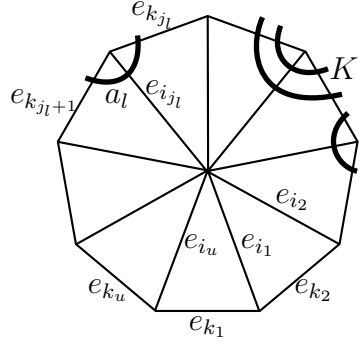}}}
\caption{The star neighborhood $\mathcal N(v)$ of the vertex $v$}
\label{fig:VertexStar}
\end{figure}

Let $K$ be a framed link in $S_\lambda \times [0,1]$. We want to show that $\rho_\lambda\bigl ([K]\bigr) = \mu \circ \Tr_\lambda^\omega\bigl ([K]\bigr)$ respects $F_v$ and, for this, we first need to understand the quantum trace $\Tr_\lambda^\omega\bigl ([K]\bigr)\in \ZZ(\lambda)$. The precise construction of  $\Tr_\lambda^\omega\bigl ([K]\bigr)$ in \cite{BonWon1} can be somewhat elaborate, and we first isotop $K$ to a position where the quantum trace is easier to analyze.

Because there are no identifications on the boundary of the neighborhood $\mathcal N(v)$, we can arrange by an isotopy that the intersection of $K$ with $\mathcal N(v) \times [0,1]$ consists of finitely many horizontal arcs $a_1\times \{*\}$, $a_2\times \{*\}$, \dots, $a_t\times \{*\}$, where each $a_l\subset \mathcal N(v)$ is an embedded arc that  turns around one of the boundary vertices of $\mathcal N(v)$, in the sense that $a_l$  goes from some boundary edge $e_{k_{j_l}}$ of $\mathcal N(v)$ to the next $e_{k_{j_l+1}}$, crosses the internal edge $e_{i_{j_l}}$ in one point, and meets no other edge of $\lambda$. See Figure~\ref{fig:VertexStar}. 
In addition, we can arrange that these horizontal arcs $a_l \times \{*\}$ sit at increasing elevations $*\in [0,1]$ as $l$ goes from 1 to $t$.

Then the  State Sum Formula of \cite{BonWon1}  expresses $\Tr_\lambda^\omega\bigl ([K]\bigr)$ as a sum of terms of the form
$$
A_1 A_2 \dots A_t B \in \ZZ(\lambda)
$$
where each term corresponds to a state for the boundary of $K \cap \bigl( \mathcal N(v) \times [0,1]\bigr)$, where $A_l$ is the contribution of the arc $a_l$, and where $B$ is the contribution of the complement of $K \cap \bigl( \mathcal N(v) \times [0,1]\bigr)$ in $K$. In addition, each non-zero $A_l$ can be of only three types 
\begin{align*}
A_l &= \omega^{2} Z_{k_{j_l}} Z_{i_{j_l}} Z_{k_{j_l+1}}  \\
\text{ or } A_l & =    Z_{k_{j_l}}^{-1} Z_{i_{j_l}} Z_{k_{j_l+1}}  +  Z_{k_{j_l}}^{-1} Z_{i_{j_l}}^{-1} Z_{k_{j_l+1}} \\
\text{ or } A_l & =  \omega^{2} Z_{k_{j_l}}^{-1} Z_{i_{j_l}}^{-1} Z_{k_{j_l+1}}^{-1} 
\end{align*}
according to the state considered, and $B$ involves only generators $Z_i$ corresponding to edges $e_i$ of $\lambda$ that are \emph{not} in $\mathcal N(v)$.

This expression has the unfortunate feature  that, although the terms $A_l$ and $B$ are elements of the Chekhov-Fock algebra $\T^\omega (\lambda)$, they do not satisfy the exponent parity condition necessary to make them elements of the balanced Chekhov-Fock algebra $\ZZ(\lambda)$. In particular, we cannot directly apply the representation $\mu_\lambda$, since terms like $\mu_\lambda(A_l)$ and $\mu_\lambda(B)$ are not  defined. 

To circumvent  this problem, we  factor out of $A_1 A_2 \dots A_v B $ the product $C=Z_{j_1}^{-1}Z_{j_2}^{-1} \dots Z_{j_w}^{-1}$, where $e_{j_1}$, $e_{j_2}$, \dots, $e_{i_w}$ are the edges of $\lambda$ crossed by $K$ (with possible repetitions). Note that the arc $a_l$ contributes, up to permutation, a term $Z_{k_{j_l}}^{-1}Z_{i_{j_l}}^{-1} Z_{k_{j_l}+1}^{-1}$ to this product $C=Z_{j_1}^{-1}Z_{j_2}^{-1} \dots Z_{j_w}^{-1}$. Set 
\begin{align*}
A_l' &=   Z_{k_{j_l}}^2 Z_{i_{j_l}}^2 Z_{k_{j_l+1}}^2  \\
\text{ or } A_l' & = \omega^{4} Z_{i_{j_l}}^2 Z_{k_{j_l+1}}^2  + Z_{k_{j_l+1}}^2  \\
\text{ or } A_l' &= 1
\end{align*}
according to which of the above three types $A_l$ belongs to. This is specially designed so that $A_l = \omega^\alpha A_l' Z_{k_{j_l}}^{-1}Z_{i_{j_l}}^{-1} Z_{k_{j_l}+1}^{-1}$ for some $\alpha \in \Z$ depending on whether the edges $e_{k_{j_l}}$ and $e_{k_{j_l}+1}$ are contained in the same face of $\lambda$ (outside of $\mathcal N(v)$) or not.

In addition, the term $Z_{k_{j_l}}^{-1}Z_{i_{j_l}}^{-1} Z_{k_{j_l}+1}^{-1}$ $\omega$--commutes with each $A_{l'}'$, in the sense that $Z_{k_{j_l}}^{-1}Z_{i_{j_l}}^{-1} Z_{k_{j_l}+1}^{-1}A_{l'}' = \omega^{\beta} A_{l'}'Z_{k_{j_l}}^{-1}Z_{i_{j_l}}^{-1} Z_{k_{j_l}+1}^{-1}$ for some $\beta\in \Z$. (The only case that requires some thought is when $A_{l'}'  = \omega^{4} Z_{i_{j_{l'}}}^2 Z_{k_{j_{l'}+1}}^2  + Z_{k_{j_{l'}+1}}^2 $, in which case it suffices to note that  $Z_{k_{j_l}}^{-1}Z_{i_{j_l}}^{-1} Z_{k_{j_l}+1}^{-1}$ commutes with $Z_{i_{j_{l'}}}^2$.) This enables us to rewrite 
$$
A_1 A_2 \dots A_t B = 
A_1' A_2' \dots A_t' B' C.
$$
where $B'$ involves only generators $Z_i$ corresponding to edges $e_i$ of $\lambda$ that are outside of $\mathcal N(v)$.

By construction, $C$ belongs to the balanced Chekhov-Fock algebra $\ZZ(\lambda)$. Each $A_l'$ is also an element of $\ZZ(\lambda)$ since all its exponents are even. Since $A_1 A_2 \dots A_t B$ belongs to $\ZZ(\lambda)$ by construction of the quantum trace, it follows that $B'$ is also in  $\ZZ(\lambda)$. In particular, we can now consider the endomorphisms $\mu_\lambda(C)$, $\mu_\lambda(A_l')$, $\mu_\lambda(B')\in \End({E_\lambda})$. We want to show that all these endomorphisms respect the off-diagonal kernel $F_v $.

Since the expression of $B'$ involves only generators $Z_i$ corresponding to edges $e_i$ of $\lambda$ that are {not} in the neiborhood $\mathcal N(v)$ of $v$, it commutes with each of the generators $Z_{i_j}$ corresponding to edges emanating from the vertex $v$. As a consequence, $B'$ commutes with the off-diagonal term $Q_v=  \sum_{j=0}^{u-1} \omega^{-4j} Z_{i_1}^2 Z_{i_2}^2 \dots Z_{i_{j}}^2$. It follows that $\mu_\lambda(B')\in \End({E_\lambda})$ respects the off-diagonal kernel $F_v = \ker \mu_\lambda(Q_v)$. 

In $C=Z_{j_1}^{-1}Z_{j_2}^{-1} \dots Z_{j_w}^{-1}$, the contribution $Z_{k_{j_l}}^{-1}Z_{i_{j_l}}^{-1} Z_{k_{j_l+1}}^{-1}$ of each arc $a_l$ commutes with each of the generators $Z_{i_j}$ corresponding to edges emanating from the vertex $v$. The remaining terms of $C$ involve only generators $Z_i$ associated to edges of $\lambda$ that are not in $\mathcal N(v)$, and therefore also commute with the $Z_{i_j}$. It  follows that $C$ also commutes with the off-diagonal term $Q_v$. This again implies that $\mu_\lambda(C)\in \End({E_\lambda})$ respects the off-diagonal kernel $F_v $. 

For $\mu_\lambda(A_l')$, we need to distinguish cases. There is nothing to prove when $A_l'=1$, since $\mu(A_l')=\Id_{E_\lambda}$ clearly respects $F_v$. Also, when $A_l' = Z_{k_{j_l}}^2 Z_{i_{j_l}}^2 Z_{k_{j_l+1}}^2 $, it commutes with all generators $Z_{i_j}$ corresponding to edges $e_{i_j}$ emanating from the vertex  $v$, and therefore commutes with the off-diagonal term $Q_v=  \sum_{j=0}^{u-1} \omega^{-4j} Z_{i_1}^2 Z_{i_2}^2 \dots Z_{i_{j}}^2$; therefore $\mu_\lambda(A_l')$ respects the off-diagonal kernel $F_v$ in this case as well. 

The case 
 $A_l' = \omega^4 Z_{i_{j_l}}^2 Z_{k_{j_l+1}}^2 + Z_{k_{j_l+1}}^2  $  requires more work. Remember from Lemma~\ref{lem:OffDiagKernelWellDefined} that we have some flexibility in the choice of $Q_v$. In particular, $F_v$ is also the kernel of
\begin{align*}
 Q_v' &= 1+ \omega^{-4} Z_{i_{j_l+1}}^2 + \omega^{-8} Z_{i_{j_l+1}}^2Z_{i_{j_l+2}}^2 + \dots \\
 &\qquad \dots + \omega^{-4(u-2)}Z_{i_{j_l+1}}^2Z_{i_{j_l+2}}^2 \dots Z_{i_{j_l-2}}^2
 + \omega^{-4(u-1)}Z_{i_{j_l+1}}^2Z_{i_{j_l+2}}^2 \dots Z_{i_{j_l-2}}^2  Z_{i_{j_l-1}}^2
\end{align*}

 Because there are no identifications between the edges on the boundary of $\mathcal N(v)$,   we observe that $Z_{i_{j_l+1}}^2 Z_{k_{j_l+1}}^2 = \omega^8 Z_{k_{j_l+1}}^2 Z_{i_{j_l+1}}^2 $, and  that $Z_{k_{j_l+1}}^2 $ commutes with all $Z_{i_j}$ with $j \neq j_l$, $j_l+1$.  Therefore
 $$
( Q_v'-1) Z_{k_{j_l+1}}^2= \omega^8  Z_{k_{j_l+1}}^2 ( Q_v'-1)
 $$
 and
 $$
 Q_v'  Z_{k_{j_l+1}}^2= \omega^8  Z_{k_{j_l+1}}^2  Q_v' + (1-\omega^8) Z_{k_{j_l+1}}^2
 $$
 
 Similarly, $Z_{i_{j_l}}^2 Z_{k_{j_l+1}}^2$ commutes with all but the last term of $Q_v'$, and 
\begin{multline*}
( Q_v' -\omega^{-4(u-1)}Z_{i_{j_l+1}}^2Z_{i_{j_l+2}}^2 \dots  Z_{i_{j_l-1}}^2 ) Z_{i_{j_l}}^2 Z_{k_{j_l+1}}^2\\
  = Z_{i_{j_l}}^2 Z_{k_{j_l+1}}^2 ( Q_v'-\omega^{-4(u-1)}Z_{i_{j_l+1}}^2Z_{i_{j_l+2}}^2 \dots  Z_{i_{j_l-1}}^2)
\end{multline*}
Reordering terms, we conclude that
\begin{align*}
Q_v ' Z_{i_{j_l}}^2 Z_{k_{j_l+1}}^2  &=  Z_{i_{j_l}}^2 Z_{k_{j_l+1}}^2 Q_v' + (1-\omega^{-8})\omega^{-4(u-1)}Z_{k_{j_l+1}}^2  Z_{i_{j_l}}^2Z_{i_{j_l+1}}^2Z_{i_{j_l+2}}^2 \dots  Z_{i_{j_l-1}}^2 \\
&= Z_{i_{j_l}}^2 Z_{k_{j_l+1}}^2 Q_v' + (1-\omega^{-8})\omega^{-4}Z_{k_{j_l+1}}^2 H_v^2
\end{align*}
where $H_v= [Z_{i_{j_l}}Z_{i_{j_l+1}}Z_{i_{j_l+2}} \dots  Z_{i_{j_l-1}}]= \omega^{-u+2} Z_{i_{j_l}}Z_{i_{j_l+1}}Z_{i_{j_l+2}} \dots  Z_{i_{j_l-1}}$ is the central element of $\ZZ(\lambda)$ associated to the vertex $v$.

Therefore, for every  vector $w\in F_v = \ker \mu_\lambda(Q_v')$, 
\begin{align*}
\mu_\lambda(Q_v') \circ \mu_\lambda\bigl( A_l')(w) &= \omega^4 \mu_\lambda(Q_v') \circ \mu_\lambda\bigl( Z_{i_{j_l}}^2 Z_{k_{j_l+1}}^2 \bigr)(w)  +\mu_\lambda(Q_v') \circ  \mu_\lambda \bigl( Z_{k_{j_l+1}}^2 \bigr)(w)\\
&= \omega^4 \mu_\lambda\bigl( Z_{i_{j_l}}^2 Z_{k_{j_l+1}}^2 \bigr) \circ \mu_\lambda(Q_v') (w) \\
&\qquad
+ (1-\omega^{-8})  \mu_\lambda(Z_{k_{j_l+1}}^2) \circ \mu_\lambda(H_v^2) (w)\\
&\qquad +
\omega^8 \mu_\lambda(Z_{k_{j_l+1}}^2) \circ \mu_\lambda(Q_v') (w) + (1-\omega^8) \mu_\lambda(Z_{k_{j_l+1}}^2)(w)\\
&= (1-\omega^{-8})  \mu_\lambda(Z_{k_{j_l+1}}^2) (\omega^8  w)+ (1-\omega^8) \mu_\lambda(Z_{k_{j_l+1}}^2)(w) = 0
\end{align*}
since $\mu_\lambda(Q_v')(w)=0$ and $\mu_\lambda(H_v) = -\omega^4 \Id_V$. As a consequence, the image of $w\in F_v = \ker \mu_\lambda(Q_v')$ under $\mu_\lambda(A_l')$ is also in $F_v$. 

This proves that $\mu_\lambda(A_l')$ respects the off-diagonal kernel $F_v$ in all three cases.

As a summary, for every skein $[K] \in \SSS(S_\lambda)$, we showed that the linear map $\rho_\lambda\bigl([K] \bigr)\in \End({E_\lambda})$ is a sum of terms 
$$
\mu_\lambda(A_1' ) \circ\mu_\lambda(A_2')\circ \dots \circ\mu_\lambda(A_v')\circ\mu_\lambda( B') \circ \mu_\lambda(C )
$$
and that each factor in this composition respects the off-diagonal kernel $F_v$. This proves that the image $\rho_\lambda\bigl( \SSS(S_\lambda) \bigr) \subset \End(E_\lambda)$ respects $F_v$, and  completes the proof of Proposition~\ref{prop:OffDiagKernelInvariant}. 
\end{proof}

\begin{rem}
\label{rem:OffDiagKernelInvariantWeakerHyp}
Although the hypotheses of Proposition~\ref{prop:OffDiagKernelInvariant} require that the triangulation $\lambda$ be combinatorial, the proof shows that this statement is valid under the weaker hypothesis that no edge of $\lambda$ connects the vertex $v$ to itself, and that no two distinct edges connect $v$ to the same vertex of $\lambda$. We will use this observation in \S \ref{subsect:FaceSubdivisions}. 

\end{rem}

\subsection{Constructing a representation of the skein algebra $\SSS(S)$}
\label{subsect:ConstructRepClosedSurface}

We  now consider the total off-diagonal kernel
$
{F_\lambda} = \bigcap_{v\in V_\lambda} F_{v} \subset {E_\lambda}
$.

Let
$
\rho_\lambda = \mu_\lambda \circ \Tr_\lambda^\omega \colon \SSS(S_\lambda) \to \End({E_\lambda})
$
be the representation associated in \S \ref{subsect:RepresentingPunctSurfSkein} to a homomorphism $r\colon \pi_1(S) \to \SL(\C)$ endowed with a $\lambda$--enhancement $\xi$. Assuming that the triangulation $\lambda$ is combinatorial,  Proposition~\ref{prop:OffDiagKernelInvariant} shows that the  total off-diagonal kernel ${F_\lambda}$ is invariant under the image $\rho_\lambda\bigl( \SSS(S_\lambda) \bigr) \subset \End({E_\lambda})$. For every framed link $K \subset S_\lambda\times [0,1]$, we can therefore consider the restriction $\rho_\lambda \bigl ( [K] \bigr)_{|{F_\lambda}}\in \End(F_\lambda)$. 

We now show that $\rho_\lambda \bigl ( [K] \bigr)_{|{F_\lambda}}\in \End(F_\lambda)$ remains invariant if we modify $K$ by an isotopy in $S\times[0,1]$, not just in $S_\lambda \times [0,1]$.

\begin{prop}
\label{prop:VertexSweep}
Suppose that the triangulation $\lambda$ is combinatorial, in the sense that each edge has distinct endpoints and that no two distinct edges have the same endpoints. 
Let the two framed links $K$, $K' \subset S_\lambda \times [0,1]$ be isotopic in $S\times [0,1]$. Then $\rho_\lambda \bigl ( [K] \bigr)_{|{F_\lambda}} = \rho_\lambda \bigl ( [K'] \bigr)_{|{F_\lambda}}$ in $\End({F_\lambda})$. 
\end{prop}
Again, the hypothesis that $\lambda$ is combinatorial is here only to simplify the proof, and the property holds without this condition; see Theorem~\ref{thm:ConstructRepSkeinAlgebraClosedSurface} in \S \ref{subsect:ConstructSkeinRepArbitraryTriangulation}. 

\begin{proof}
We can choose the isotopy from $K$ to $K'$ so that it sweeps through punctures of $S_\lambda$ at only finitely many times. This reduces the problem to the case where the isotopy sweeps only once through a puncture. Let $v$ be the vertex of the triangulation $\lambda$ corresponding to this puncture. 

We will be using the same labeling conventions as in the proof of Proposition~\ref{prop:OffDiagKernelInvariant}. In particular, $\mathcal N(v)$ denotes the union of the faces of $\lambda$ that contain $v$. The edges emanating from $v$ are indexed as $e_{i_1}$, $e_{i_2}$, \dots, $e_{i_u}$ in counterclockwise order around $v$. The edges of the boundary of $\mathcal N(v)$ are $e_{k_1}$, $e_{k_2}$, \dots, $e_{k_u}$, where $e_{k_j}$ is the third side of the face containing $e_{i_{j-1}}$ and $e_{i_j}$. See Figure~\ref{fig:VertexSweep}. 

Since the skeins $[K]$, $[K'] \in \SSS(S_\lambda)$ are invariant under isotopy in $S_\lambda \times [0,1]$, we can restrict attention to the case of  Figure~\ref{fig:VertexSweep}, where the two pieces of $K$ and $K'$ represented are endowed with the vertical framing, and where the remaining portions of $K$ and $K'$ coincide and are located  in $S_\lambda \times [0,1]$ at lower elevations  than the pieces represented. 

\begin{figure}[htbp]

\SetLabels
(.69 *.16 ) $ e_{i_1}$ \\
( .31* .17) $ e_{i_u}$ \\
( .16* .38) $e_{i_{u-1}} $ \\
( .84* .35) $e_{i_2} $ \\
( .5* -.01) $ e_{k_1}$ \\
( .86* .11) $ e_{k_2}$ \\
( .18* .11) $ e_{k_u}$ \\
( .57*.68 ) $ K$ \\
( .49*.31 ) $ K'$ \\
\endSetLabels
\centerline{\AffixLabels{\includegraphics{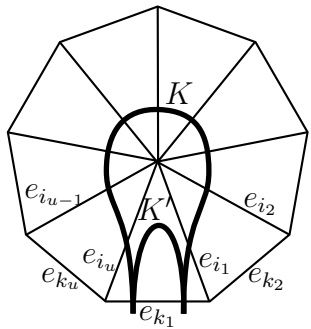}}}
\caption{Sweeping through a puncture}
\label{fig:VertexSweep}
\end{figure}

These elevation conventions greatly simplify the computation of the quantum trace $\Tr_\lambda^\omega\bigl([K]\bigr)$, because we do not have to worry about correction factors coming from biangles. The construction of the quantum trace in \cite{BonWon1} then gives that
\begin{align*}
\Tr_\lambda^\omega \bigl ( [K] \bigr)
=& B_{++} Z_{k_1} \biggl( \omega^{3-u} Z_{i_1} Z_{i_2} \dots Z_{i_u} + \omega^{3-u} \sum_{j=1}^{u-1} Z_{i_1}Z_{i_2}\dots Z_{i_j}Z_{i_{j+1}}^{-1}Z_{i_{j+2}}^{-1}\dots Z_{i_u}^{-1} 
\biggr) Z_{k_1} \\
&+ B_{+-} Z_{k_1} \biggl( \omega^{5-u} \sum_{j=1}^{u-1} Z_{i_1}Z_{i_2}\dots Z_{i_j}Z_{i_{j+1}}^{-1}Z_{i_{j+2}}^{-1}\dots Z_{i_u}^{-1} 
\biggr) Z_{k_1}^{-1}\\
&+ B_{-+} Z_{k_1}^{-1} \biggl( \omega^{1-u}Z_{i_1} Z_{i_2} \dots Z_{i_u} + \omega^{1-u}Z_{i_1}^{-1} Z_{i_2}^{-1} \dots Z_{i_u}^{-1}\\
&\qquad\qquad\qquad\qquad \qquad\qquad\qquad + \omega^{1-u}\sum_{j=1}^{u-1} Z_{i_1}Z_{i_2}\dots Z_{i_j}Z_{i_{j+1}}^{-1}Z_{i_{j+2}}^{-1}\dots Z_{i_u}^{-1} 
\biggr) Z_{k_1} \\
&+ B_{--} Z_{k_1}^{-1} \biggl( \omega^{3-u} Z_{i_1}^{-1} Z_{i_2}^{-1} \dots Z_{i_u}^{-1}  + \omega^{3-u}\sum_{j=1}^{u-1} Z_{i_1}Z_{i_2}\dots Z_{i_j}Z_{i_{j+1}}^{-1}Z_{i_{j+2}}^{-1}\dots Z_{i_u}^{-1} 
\biggr) Z_{k_1}^{-1}
\end{align*}
where the terms $B_{\pm\pm}$ are contributions of the parts of the link $K$ that are not represented on Figure~\ref{fig:VertexSweep}. The domino diagrams of Exercises~8.5--8.7 and 10.14 in \cite{BonBook} may be here useful to list all possible terms. The order of terms is dictated by our condition on relative elevations. 

We now move all terms $Z_{k_1}$ together. Because of our hypothesis that $\lambda$ is combinatorial, there are no extraneous identifications between the edges represented in Figure~\ref{fig:VertexSweep}. It follows that $Z_{k_1}$ commutes with all $Z_{i_j}$ with $1<j<u$, and that $Z_{k_1}Z_{i_1} = \omega^{-2} Z_{i_1}Z_{k_1}$ and $Z_{k_1}Z_{i_u}=\omega^2 Z_{i_u}Z_{k_1}$. This gives:
\begin{align*}
\Tr_\lambda^\omega \bigl ( [K] \bigr)
=& B_{++} Z_{k_1}^2 \biggl( \omega^{3-u} Z_{i_1} Z_{i_2} \dots Z_{i_u} + \omega^{7-u} \sum_{j=1}^{u-1} Z_{i_1}Z_{i_2}\dots Z_{i_j}Z_{i_{j+1}}^{-1}Z_{i_{j+2}}^{-1}\dots Z_{i_u}^{-1} 
\biggr) \\
&+ B_{+-} \biggl( \omega^{1-u} \sum_{j=1}^{u-1} Z_{i_1}Z_{i_2}\dots Z_{i_j}Z_{i_{j+1}}^{-1}Z_{i_{j+2}}^{-1}\dots Z_{i_u}^{-1} 
\biggr) \\
&+ B_{-+}  \biggl( \omega^{1-u}Z_{i_1} Z_{i_2} \dots Z_{i_u} + \omega^{1-u}Z_{i_1}^{-1} Z_{i_2}^{-1} \dots Z_{i_u}^{-1} \\
&\qquad\qquad\qquad\qquad \qquad\qquad\qquad+ \omega^{5-u}\sum_{j=1}^{u-1} Z_{i_1}Z_{i_2}\dots Z_{i_j}Z_{i_{j+1}}^{-1}Z_{i_{j+2}}^{-1}\dots Z_{i_u}^{-1} 
\biggr) \\
&+ B_{--} Z_{k_1}^{-2} \biggl( \omega^{3-u} Z_{i_1}^{-1} Z_{i_2}^{-1} \dots Z_{i_u}^{-1} + \omega^{-1-u}\sum_{j=1}^{u-1} Z_{i_1}Z_{i_2}\dots Z_{i_j}Z_{i_{j+1}}^{-1}Z_{i_{j+2}}^{-1}\dots Z_{i_u}^{-1} 
\biggr) .
\end{align*}

We now recognize several of the terms in this sum. For instance,  the central element $H_v\in \ZZ(\lambda)$ associated to the vertex $v$ is 
$$
H_v = [Z_{i_1} Z_{i_2} \dots Z_{i_u}] = \omega^{2-u} Z_{i_1} Z_{i_2} \dots Z_{i_u}.
$$
Similarly,
$$H_v^{-1} =  [Z_{i_1}^{-1} Z_{i_2}^{-1} \dots Z_{i_u}^{-1}] = \omega^{2-u} Z_{i_1}^{-1} Z_{i_2}^{-1} \dots Z_{i_u}^{-1}.
$$
Also.
\begin{align*}
\sum_{j=1}^{u-1} Z_{i_1}Z_{i_2}\dots Z_{i_j}Z_{i_{j+1}}^{-1}Z_{i_{j+2}}^{-1}\dots Z_{i_u}^{-1}  
&= \biggl(
\sum_{j=1}^{u-1} \omega^{-4(j-1)} Z_{i_1}^2 Z_{i_2}^2 \dots Z_{i_j}^2
\biggr)
Z_{i_1}^{-1} Z_{i_2}^{-1} \dots Z_{i_u}^{-1}\\
&=  \biggl(
\sum_{j=2}^{u} \omega^{-4(j-2)} Z_{i_1}^2 Z_{i_2}^2 \dots Z_{i_{j-1}}^2
\biggr)
\omega^{u-2} H_v^{-1}\\
&= \omega^{u+2} (Q_v-1)H_v^{-1}
\end{align*}
where $Q_v$ is the off-diagonal term of the vertex $v$ defined in \S \ref{subsect:OffDiagTerm}. 

This gives
\begin{align*}
\Tr_\lambda^\omega \bigl ( [K] \bigr)
=& B_{++} Z_{k_1}^2 \bigl( \omega H_v + \omega^{9} (Q_v-1) H_v^{-1}
\bigr) + B_{+-} \omega^{3} (Q_v-1) H_v^{-1} \\
&+ B_{-+}  \bigl( \omega^{-1}H_v+ \omega^{-1} H_v^{-1} + \omega^{7} (Q_v-1)H_v^{-1}
\bigr) \\
&+ B_{--} Z_{k_1}^{-2} \bigl( \omega H_v^{-1} + \omega (Q_v-1)H_v^{-1} 
\bigr) .
\end{align*}

Note that all the terms arising in this expression belong to the balanced Chekhov-Fock algebra $\ZZ(\lambda)$. We can therefore apply the representation $\mu_\lambda\colon \ZZ(\lambda) \to \End(E_\lambda)$. Remembering that $\mu_\lambda(H_v) = -\omega^4\, \Id_{E_\lambda}$, 
\begin{align*}
\rho_\lambda\bigl( [K] \bigr) =& \mu_\lambda \circ \Tr_\lambda^\omega \bigl ( [K] \bigr)\\
=& \mu_\lambda(B_{++} Z_{k_1}^2) \circ \bigl(- \omega^{5} \mu_\lambda(Q_v)
\bigr) +   \mu_\lambda(B_{+-}) \circ \bigl(-\omega^{-1} \mu_\lambda(Q_v)+\omega^{-1} \Id_{E_\lambda} \bigr) \\
&+ \mu_\lambda(B_{-+} )\circ  \bigl(- \omega^{-5}\,\Id_{E_\lambda} -  \omega^{3} \mu_\lambda(Q_v)
\bigr) +  \mu_\lambda(B_{--} Z_{k_1}^{-2}) \circ \bigl( - \omega^{-3}  \mu_\lambda (Q_v)
\bigr) .
\end{align*}

This expression greatly simplifies when we restrict it to the total off-diagonal kernel ${F_\lambda}\subset \ker \mu_\lambda(Q_v)$, and 
$$
 \rho_\lambda\bigl( [K] \bigr)_{|{F_\lambda}} = \omega^{-1}  \mu_\lambda(B_{+-})_{|{F_\lambda}} - \omega^{-5} \mu_\lambda(B_{-+} )_{|{F_\lambda}}.
$$

We now perform the same computations for the skein $[K']\in \SSS(S)_0)$. The principles are the same, but everything is much simpler because the framed knot $K'$ meets many fewer edges of $\lambda$. In particular, the expression of  $ \Tr_\lambda^\omega \bigl ( [K'] \bigr)$ is much less cumbersome, and gives
$$
\Tr_\lambda^\omega \bigl ( [K'] \bigr)
=  B_{+-} Z_{k_1}( -\omega^{-1} 
) Z_{k_1}^{-1} +  B_{-+} Z_{k_1}^{-1}( \omega^{-5}) Z_{k_1} 
= \omega^{-1}  B_{+-}  -  \omega^{-5} B_{-+} 
$$
where the terms $B_{\pm\pm}$ are contributions of the parts of the link $K'$ that are not represented on Figure~\ref{fig:VertexSweep}, and are the same as those that appeared in the computation of $ \Tr_\lambda^\omega \bigl ( [K] \bigr)$ since these ``hidden parts'' of the links $K$ and $K'$ coincide.

As a consequence,
$$
 \rho_\lambda\bigl( [K'] \bigr)_{|{F_\lambda}} =\mu_\lambda \circ \Tr_\lambda^\omega \bigl ( [K'] \bigr)_{|{F_\lambda}} = \omega^{-1}  \mu_\lambda(B_{+-})_{|F_\lambda} - \omega^{-5} \mu_\lambda(B_{-+} )_{|{F_\lambda}}.
$$

Comparing the two formulas, we see that $ \rho_\lambda\bigl( [K] \bigr)_{|{F_\lambda}} =  \rho_\lambda\bigl( [K'] \bigr)_{|{F_\lambda}} $, which completes the proof of Proposition~\ref{prop:VertexSweep}.
 \end{proof}
 
 \begin{rem}
 \label{rem:VertexSweepWeakerHyp}
As in Remark~\ref{rem:OffDiagKernelInvariantWeakerHyp}, the proof of  Proposition~\ref{prop:VertexSweep} is valid under a weaker hypothesis than the requirement that the triangulation $\lambda$ be combinatorial. Indeed, the following condition is sufficent for the statement to hold:  no edge of $\lambda$ connects the vertex $v$ to itself, and  no two distinct edges connect $v$ to the same vertex of $\lambda$. We will use this observation in \S \ref{subsect:FaceSubdivisions}. 
\end{rem}
 
 A consequence of Proposition~\ref{prop:VertexSweep} is that the representation $
\rho_\lambda  \colon \SSS(S_\lambda) \to \End({E_\lambda})
$ induces a representation $
\check \rho_\lambda  \colon \SSS(S) \to \End({F_\lambda})
$ by the property that 
$$
\check \rho_\lambda\bigl( [K] \bigr) =\rho_\lambda\bigl( [K] \bigr)_{|{F_\lambda}} \in  \End({F_\lambda}). 
$$
  
 \begin{prop}
\label{prop:ClassicalShadowOKCombinatorial}
The above representation $\check \rho_\lambda \colon \SSS(S) \to \End({F_\lambda})$ has classical shadow $r\in \RR(S)$, in the sense that 
 $$
T_N \bigl( \check \rho_\lambda ([K]) \bigr) = - \Tr\, r(K)\, \Id_{F_\lambda}
$$
for every framed knot $K \subset S_\lambda\times [0,1]$ whose projection to $S_\lambda$ has no crossing and whose framing is vertical.
\end{prop}

\begin{proof} Let $K \subset S_\lambda\times [0,1]$ be a framed knot whose projection to $S_\lambda$ has no crossing and whose framing is vertical. By definition of the representation $\mu_\lambda \colon \ZZ(\lambda) \to \End({E_\lambda})$ and of $\rho_\lambda =\mu_\lambda \circ \Tr_\lambda^\omega $, and in particular by Condition~(4) of Proposition~\ref{prop:ConstructRepCheFock}, 
 $$
T_N \bigl( \rho_\lambda ([K]) \bigr) = - \Tr\, r(K)\, \Id_{E_\lambda}.
$$
In particular, by restriction to the off-diagonal kernel $F_\lambda$, 
\begin{equation*}
T_N \bigl( \check \rho_\lambda ([K]) \bigr) = T_N \bigl( \rho_\lambda ([K]) \bigr)_{|{F_\lambda}} =- \Tr\, r(K)\, \Id_{F_\lambda}. \qedhere
\end{equation*}
\end{proof}

At this point, it looks like we are almost done with the proof of Theorem~\ref{thm:RealizeInvariantIntro}. The only problem is that we do not know that the total off-diagonal kernel ${F_\lambda}= \bigcap_{v\in V_\lambda} F_v$ is non-trivial. In fact, we don't even know that any of the off-diagonal kernels $F_v$ is non-trivial. The rest of the article is devoted to estimating the dimension of ${F_\lambda}$. At the same time, we will prove that the representation $\check \rho_\lambda$ is, up to isomorphism, independent of all the choices that we have made.

\section{Changing triangulations}
\label{bigsect:ChangingTriangulation}

In this section, we introduce two moves that modify the triangulation $\lambda$ without changing the isomorphism class of the representation $\check \rho_\lambda \colon \SSS(S) \to \End({F_\lambda})$ constructed above. We will then use these moves to prove that, up to isomorphism and sign-reversal symmetry, $\check \rho_\lambda$ is independent of the choice of $\lambda$ and of the $\lambda$--enhancement $\xi$. 

Unlike in the previous section, the triangulations that we are considering here are not assumed any more to be combinatorial. 

\subsection{Subdividing faces}
\label{subsect:FaceSubdivisions}

Let $\lambda$ be a triangulation of the surface $S$. Let $\lambda'$ be the triangulation obtained from $\lambda$ by subdividing the face $T$ into three triangles as in Figure~\ref{fig:TriangSubdivis}. In particular, the vertex set of $\lambda'$ consists of the vertices of $\lambda$ plus one vertex $v_0$ in the interior of $T$. 

\begin{figure}[htbp]
\centerline
{\SetLabels
( .5*-.08 ) $e_1$ \\
( .8*.5 ) $e_2$ \\
( 0.2* .5) $e_3$ \\
( 0.5* .38) $T$ \\
(0.5* -.25) (a)\\
\endSetLabels
\AffixLabels{\includegraphics{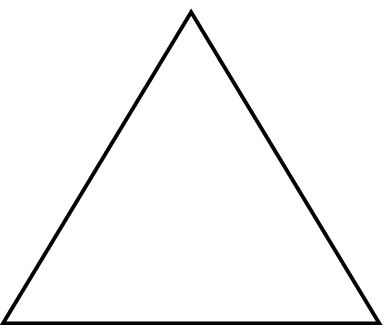}}
\hskip 1cm
\SetLabels
( .5* -.09 ) $e_1' $ \\
( .8*.5 ) $e_2' $ \\
( 0.19* .5) $e_3' $ \\
( .6* .55) $ e_{n+1}'$ \\
(.38 * .17) $e_{n+2}' $ \\
( .65* .18) $e_{n+3}' $ \\
(0.5* -.25) (b)\\
\endSetLabels
\AffixLabels{\includegraphics{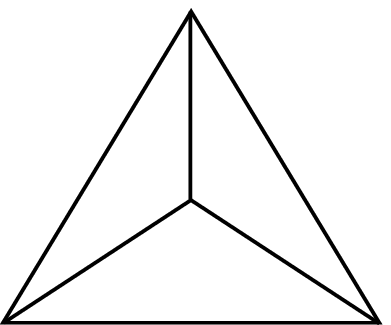}}}
\vskip 15pt
\caption{Subdividing a face}
\label{fig:TriangSubdivis}
\end{figure}

For convenience in the notation, index the edges $e_1$, $e_2$, \dots, $e_n$ of $\lambda$ and the edges $e_1'$, $e_2'$, \dots $e_{n+3}'$ of $\lambda'$  in such a way that:
\begin{enumerate}
\item the sides of the face $T$ of $\lambda$ are $e_1$, $e_2$, $e_3$, in this order as one goes counterclockwise around $T$;
\item for $i\leq n$, the edge $e_i'$  of $\lambda'$ coincides with  the edge $e_i$ of $\lambda$;
\item $e_{n+1}'$, $e_{n+2}'$, $e_{n+3}'$ are the ``new'' edges of $\lambda'$ that are not edges of $\lambda$, and each $e_{n+j}'$ is opposite $e_j'$ in $T$ as in Figure~\ref{fig:TriangSubdivis}(b), in the sense that no face of $\lambda'$ contains both $e_{n+j}'$ and $e_j'$. 
\end{enumerate}

We assume that we are given a $\lambda'$--enhancement $\xi' \colon \widetilde V_{\lambda'} \to \mathbb{CP}^1$  for  the  homomorphism $r\colon \pi_1(S) \to \SL(\C)$. By restriction, $\xi'$ defines a $\lambda$--enhancement $\xi \colon \widetilde V_{\lambda} \to \mathbb{CP}^1$ for $r$. 

We want to compare the two  irreducible representations  $\mu_{\lambda} \colon \ZZ(\lambda) \to \End({E_\lambda})$ and $\mu_{\lambda'} \colon \ZZ(\lambda') \to \End(E_{\lambda'})$  respectively associated to the enhanced homomorphisms $(r, \xi)$ and $(r, \xi')$  by Proposition~\ref{prop:ConstructRepCheFock}. For this, we first construct a natural algebra  homomorphism $\ZZ(\lambda) \to \ZZ(\lambda')$.

Let $\Phi \colon \ZZ(\lambda) \to \ZZ(\lambda')$ be the linear map defined by the property that
$$
\Phi \bigl( [Z_1^{k_1} Z_2^{k_2} \dots Z_n^{k_n} ]\bigr)
=
\bigl[ Z_1^{\prime\, k_1} Z_2^{\prime\, k_2} \dots Z_n^{\prime\, k_n} Z_{n+1}^{\prime\, \frac{k_2 + k_3 - k_1}2} Z_{n+2}^{\prime\, \frac{k_1 + k_3 - k_2}2}Z_{n+3}^{\prime\, \frac{k_1 + k_2 - k_3}2}  \bigr]
$$
for every monomial $[Z_1^{k_1} Z_2^{k_2} \dots Z_n^{k_n} ]\in \ZZ(\lambda) $ (where $[\phantom{m} ]$ denotes the Weyl quantum ordering). By definition of the Weyl quantum ordering, this is equivalent to the property that
$$
\Phi \bigl( Z_1^{k_1} Z_2^{k_2} \dots Z_n^{k_n} \bigr)
=\omega^{k_1k_3-k_1k_2 - k_2k_3} 
Z_1^{\prime\, k_1} Z_2^{\prime\, k_2} \dots Z_n^{\prime\, k_n} 
\bigl[ Z_{n+1}^{\prime\, \frac{k_2 + k_3 - k_1}2} Z_{n+2}'^{\prime\, \frac{k_1 + k_3 - k_2}2}Z_{n+3}^{\prime\, \frac{k_1 + k_2 - k_3}2}  \bigr],
$$
which is a little easier to handle since this second formula does not require us to consider the skew-commutativity properties of the $Z_i$ and $Z_i'$ with $1 \leq i \leq n$. The definition may look a little mysterious at this point, but will become clearer with the proof of Lemma~\ref{lem:SubdivisAlgebraHom} below.

Note that the definition  makes sense because, if the exponents $(k_1, k_2, \dots, k_n)$ satisfy the parity condition required for $ Z_1^{k_1} Z_2^{k_2} \dots Z_n^{k_n} $ to be in the balanced Chekhov-Fock algebra $\ZZ(\lambda)$, the  parity of the exponents  $$\textstyle (k_1, k_2, \dots, k_n, {\frac{k_2 + k_3 - k_1}2}, {\frac{k_1 + k_3 - k_2}2}, {\frac{k_1 + k_2 - k_3}2})$$ 
also guarantees that $Z_1^{\prime\, k_1} Z_2^{\prime\, k_2} \dots Z_n^{\prime\, k_n} Z_{n+1}^{\prime\, \frac{k_2 + k_3 - k_1}2} Z_{n+2}^{\prime\, \frac{k_1 + k_3 - k_2}2}Z_{n+3}^{\prime\, \frac{k_1 + k_2 - k_3}2}$ belongs to $\ZZ(\lambda')$.

\begin{lem}
\label{lem:SubdivisAlgebraHom}
The map $\Phi \colon \ZZ(\lambda) \to \ZZ(\lambda')$ is an algebra homomorphism. 
\end{lem}

\begin{proof}
This is a simple consequence of the description, given in \cite[\S 2.2]{BonWon4},  of the algebraic structure of $\ZZ(\lambda)$ in terms of  the Thurston intersection form on a train track $\tau_\lambda$ associated to $\lambda$. The train track $\tau_\lambda \subset S$ is defined by the property that, on each face of $\lambda$, it consists of three edges as in Figure~\ref{fig:TrainTracks}(a). In particular, there is a one-to-one correspondence between the switches of $\tau_\lambda$ and the edges $e_i$ of $\lambda$. 

\begin{figure}[htbp]

\centerline{
\SetLabels
( .5* .48) $ \tau_\lambda$ \\
( .5* -.15) (a) \\
\endSetLabels
{\AffixLabels{\includegraphics{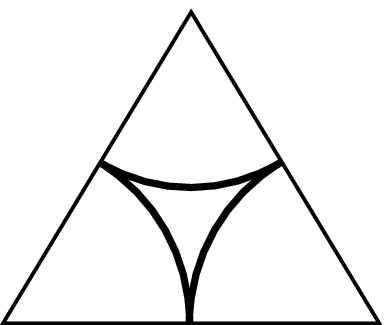}}}
\hskip 1cm
\SetLabels
(.4 *.5 ) $\tau_{\lambda'} $ \\
( .5* -.15) (b) \\
\endSetLabels
{\AffixLabels{\includegraphics{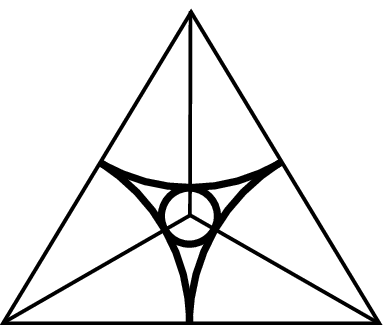}}}
}
\vskip 15pt

\caption{The train tracks $\tau_{\lambda} $ and $\tau_{\lambda'} $}
\label{fig:TrainTracks}
\end{figure}

In \cite[\S 2]{BonWon4}, we interpret the monomials $[Z_1^{k_1} Z_2^{k_2} \dots Z_n^{k_n} ]\in \ZZ(\lambda)$  as integer edge weight systems for $\tau_\lambda$ satisfying the usual switch conditions. Namely, at the switch of $\tau_\lambda$ located on the edge $e_i$ of $\lambda$, the weights of the two edges of $\tau_\lambda$ incoming on one side of that switch are required to add up to the same number $k_i$ as the weights of the two edges outgoing on the other side. The exponents $(k_1, k_2, \dots, k_n)$ of $Z_1^{k_1} Z_2^{k_2} \dots Z_n^{k_n}$ satisfy the parity condition required for $Z_1^{k_1} Z_2^{k_2} \dots Z_n^{k_n}$ to be in the balanced Chekhov-Fock algebra $\ZZ(\lambda)$ if and only if they are associated in this way to an integer edge weight systems for $\tau_\lambda$; in addition, the edge weight system is then uniquely determined. 

This enables us to identify the set $\mathcal W(\tau_\lambda; \Z)$ of  integer edge weight systems for $\tau_\lambda$ to the set of exponent $n$--tuples $\mathbf k = (k_1, k_2, \dots, k_n)$ satisfying the required parity condition, and therefore to the set of Weyl quantum ordered monomials $[Z_1^{k_1} Z_2^{k_2} \dots Z_n^{k_n} ]\in \ZZ(\lambda)$. 

The set $\mathcal W(\tau_\lambda; \Z)$ of edge weight systems for a train track carries a natural  bilinear form, the \emph{Thurston intersection form}, which provides an antisymmetric bilinear form 
$$
\Omega \colon\mathcal W(\tau_\lambda; \Z) \times \mathcal W(\tau_\lambda; \Z) \to \Z. 
$$
Lemma~10 of \cite{BonWon4} then describes the algebraic structure of $\ZZ(\lambda)$ by the property that
$$
[Z_1^{k_1} Z_2^{k_2} \dots Z_n^{k_n} ]\ [Z_1^{k_1'} Z_2^{k_2'} \dots Z_n^{k_n'} ]
= \omega^{2\Omega(\mathbf k, \mathbf {k'})} [Z_1^{k_1+k_1'} Z_2^{k_2+k_2'} \dots Z_n^{k_n+k_n'} ]
$$
for every $\mathbf k = (k_1, k_2, \dots, k_n)$ and $\mathbf {k'} = (k_1', k_2', \dots, k_n')\in \mathcal W(\tau_\lambda; \Z)$.

The key observation is now that there is a natural embedding $\tau_\lambda \to \tau_{\lambda'}$, identifying $\tau_\lambda$ to the complement in $\tau_{\lambda'}$ of the three edges that are adjacent to the central vertex $v_0 \in T$. This embedding provides a map $\phi \colon \mathcal \mathcal W(\tau_\lambda; \Z) \to\mathcal W(\tau_{\lambda'}; \Z)$, which is expressed  in terms of switch weights as
$$\textstyle
\phi(k_1, k_2, \dots, k_n) =\bigl (k_1, k_2, \dots, k_n, \frac{k_2+k_3-k_1}2,  \frac{k_1+k_3-k_2}2,  \frac{k_1+k_2-k_3}2\bigr).
$$
for every $(k_1, k_2, \dots, k_n)\in \mathcal W(\lambda)$. As a consequence, identifying each element $(k_1, k_2, \dots, k_n)\in\mathcal W(\tau_\lambda; \Z)$ to the corresponding monomial $[Z_1^{k_1} Z_2^{k_2} \dots Z_n^{k_n} ]\in \ZZ(\lambda)$, the map $\Phi\colon \ZZ(\lambda) \to \ZZ(\lambda')$ is the unique linear extension of $\phi$.

Because $\phi \colon \mathcal \mathcal W(\tau_\lambda; \Z) \to\mathcal W(\tau_{\lambda'}; \Z)$ is induced by the embedding $\tau_\lambda \to \tau_{\lambda'}$, the classical homological interpretation of the Thurston intersection form as a homological intersection number is an orientation covering (see for instance Lemma~28 of \cite{BonWon4}) shows that $\phi$ sends the Thurston form of $\tau_\lambda$ to the Thurston form of $\tau_{\lambda'}$. From the description of the algebraic structure of $\ZZ(\lambda)$ and $\ZZ(\lambda)$ in terms of  Thurston intersection forms, it follows that $\Phi$ is an algebra homomorphism. 
\end{proof}

Now that we know that $\Phi$ is an algebra homomorphism, we can consider the composition
$$
\mu_{\lambda'} \circ \Phi \colon \ZZ(\lambda) \to \End(E_{\lambda'})
$$
of $\Phi$ with the representation $\mu_{\lambda'}\colon \ZZ(\lambda') \to \End(E_{\lambda'})$ associated by Proposition~\ref{prop:ConstructRepCheFock} to the $\lambda'$--enhancement $\xi'$. 

Recall that $v_0$ is the vertex of $\lambda'$ that is not a vertex of $\lambda$, namely the one that was added in the interior of the face $T$ of $\lambda$.

\begin{lem}
\label{lem:SubdivRespectsOffDiagKernel}
The image  $\mu_{\lambda'} \circ \Phi\bigl( \ZZ(\lambda) \bigr) \subset \End(E_{\lambda'})$ respects the off-diagonal kernel $F_{v_0}' \subset E_{\lambda'}$ of the vertex $v_0 \in V_{\lambda'} - V_\lambda$. 
\end{lem}

\begin{proof}
By the exponent parity condition defining $\ZZ(\lambda)$, each monomial of $\ZZ(\lambda)$ is a product of constants, monomials of the form $Z_1Z_2Z_4^{k_4}Z_5^{k_5} \dots Z_n^{k_n}$, $Z_1Z_3Z_4^{k_4}Z_5^{k_5} \dots Z_n^{k_n}$, $Z_2Z_3Z_4^{k_4}Z_5^{k_5} \dots Z_n^{k_n} \in \ZZ(\lambda')$  and their inverses. It therefore suffices to show that the image of each of these elementary monomials respects $F_{v_0}'$. 

Consider for instance $Z=Z_1Z_2Z_4^{k_4}Z_5^{k_5} \dots Z_n^{k_n}$. Its image under $\Phi$ is
$$
\Phi(Z)= 
\Phi\bigl(
Z_1Z_2Z_4^{k_4}Z_5^{k_5} \dots Z_n^{k_n}
\bigr)
= \omega^{-1} Z_1'Z_2' Z_4^{\prime\, k_4}Z_5^{\prime\, k_5} \dots Z_n^{\prime\, k_n} Z_{n+3}'.
$$
As a consequence, $\Phi(Z)$ commutes with each of the elements $Z_{n+1}^{\prime\,2}$, $Z_{n+2}^{\prime\,2}\kern 5pt$ and $Z_{n+3}^{\prime\,2} \in \ZZ(\lambda')$ associated to the edges of $\lambda'$ emanating from the vertex $v_0$. In particular, $\Phi(Z)$ commutes with the off-diagonal term $Q_{v_0} = 1+\omega^{-4} Z_{n+1}^{\prime\, 2} + \omega^{-8} Z_{n+1}^{\prime\, 2} Z_{n+2}^{\prime\, 2}$ of the vertex $v_0$. It follows that $\mu_{\lambda'} \circ \Phi(Z)$ respects the kernel $F_{v_0}'$ of $\mu_{\lambda'}(Q_{v_0})$. 

The same argument holds for the other  two monomial types  $Z_1Z_3Z_4^{k_4}Z_5^{k_5} \dots Z_n^{k_n}$ and\linebreak   $Z_2Z_3Z_4^{k_4}Z_5^{k_5} \dots Z_n^{k_n}$, and proves the required result. 
\end{proof}

The following result plays a critical role in our arguments. Its proof uses the non-quantum context, and in particular the off-diagonal equality of Lemma~\ref{lem:ClassicalOffDiagonal}, in a crucial way.

\begin{lem}
\label{lem:SubdivisionDimOffDiagKernel}
The dimension of the off-diagonal kernel $F_{v_0}' \subset E_{\lambda'}$ is equal to the dimension $N^{3g+p_\lambda-3}$ of $E_\lambda$, where $g$ is the genus of the surface $S$ and $p_\lambda$ is the number of vertices of the triangulation $\lambda$. 
\end{lem}

\begin{proof}
By construction of the representation $\mu_{\lambda'}$ by Proposition~\ref{prop:ConstructRepCheFock}, $\dim E_{\lambda'}= N^{3g+p_{\lambda'}-3}$. Since $p_\lambda = p_{\lambda'}-1$, it therefore suffices to show that  $ F_{v_0}' $ has dimension $ \frac1N \dim E_{\lambda'}$. 

 Consider the off-diagonal term
$$
Q_{v_0}  = 1+ \omega^{-4} Z_{n+1}^{\prime\,2} + \omega^{-8} Z_{n+1}^{\prime\, 2} Z_{n+2}^{\prime\, 2}.
$$
Because $Z_{n+1}^{\prime\, 2} $ $\omega^8$--commutes with $Z_{n+1}^{\prime\, 2} Z_{n+2}^{\prime\, 2}$ and because $\omega^8= A^{-4}$ is a primitive $N$--root of unity (we here use the fact that $N$ is odd), the Quantum Binomial Formula (see for instance \cite[\S IV.2]{Kassel}) gives that
$$
(Q_{v_0} -1)^N=  \bigl(\omega^{-4} Z_{n+1}^{\prime\, 2} + \omega^{-8} Z_{n+1}^{\prime\, 2} Z_{n+2}^{\prime\, 2} \bigr)^N = Z_{n+1}^{\prime\,2N}+  Z_{n+1}^{\prime\, 2N} Z_{n+2}^{\prime\, 2N}.
$$
Applying $\mu_{\lambda'}$ then gives that
$$
\mu_{\lambda'}(Q_{v_0} -1)^N=  \mu_{\lambda'}(Z_{n+1}^{\prime\, 2N})+  \mu_{\lambda'}(Z_{n+1}^{\prime\, 2N} Z_{n+2}^{\prime\, 2N}) = x_{n+1}'  \, \Id_{E_{\lambda'}} + x_{n+1}'  x_{n+2}' \Id_{E_{\lambda'}} =-\Id_{E_{\lambda'}}, 
$$
where the $x_i'\in \C^*$ are the crossratio weights associated by the enhancement $\xi'$ to the edges $e_i'$ of $\lambda'$, and where the last equality comes from Lemma~\ref{lem:ClassicalOffDiagonal}. 

It follows that  $\mu_{\lambda'}(Q_{v_0} -1) \in \End(E_{\lambda'})$ is diagonalizable, and that its eigenvalues are $N$--roots of $-1$, namely are all of the form $-\omega^{8k}$ with $k\in \Z$. 

Now consider the element $Z_2^{\prime\, 2}\in \ZZ(\lambda')$ associated to the edge $e_2'$ of $\lambda'$. Since $Z_2^{\prime\, 2}(Q_{v_0} -1) = \omega^{-8} (Q_{v_0} -1)Z_2^{\prime\, 2}$, the linear isomorphism $\mu_{\lambda'}(Z_2^{\prime\, 2} ) \in \End(E_{\lambda'})$ sends the $(-\omega^{8k})$--eigenspace of $\mu_{\lambda'}(Q_{v_0} -1)$ to the $(-\omega^{8k+8})$--eigenspace. It follows that all numbers $-\omega^{8k}$  occur as eigenvalues  of $\mu_{\lambda'}(Q_{v_0} -1)$, and that the corresponding eigenspaces all have the same dimension. Since there are $N$ such eigenspaces, their  dimension is 
$
\frac 1N \dim E_{\lambda'}
$.

In particular, $F_{v_0}' = \ker \mu_{\lambda'}(Q_{v_0} )$ has dimension $\frac1N \dim E_{\lambda'}$, since it is the $(-1)$--eigenspace of $\mu_{\lambda'}(Q_{v_0} -1)$. This concludes the proof.
\end{proof}

At this point, we have two representations $\SSS(S_{\lambda'}) \to \End(E_{\lambda'})$. The first one is our usual
$$
\rho_{\lambda'} = \mu_{\lambda'} \circ \Tr_{\lambda'}^\omega \colon \SSS(S_{\lambda'}) \to \End(E_{\lambda'}).
$$
The second representation comes from the composition
$$
 \SSS(S_{\lambda'}) 
 \stackrel{I}{\longrightarrow}
 \SSS(S_{\lambda}) 
 \stackrel{\Tr_\lambda^\omega\kern2pt}{\longrightarrow}
 \ZZ(\lambda)
 \stackrel{\Phi}{\longrightarrow}
 \ZZ(\lambda')
 \stackrel{\mu_{\lambda'}}{\longrightarrow}
 \End(E_{\lambda'}),
$$
where the first homomorphism $I \colon \SSS(S_{\lambda'}) \to \SSS(S_\lambda)$ is induced by the inclusion map $S_{\lambda'} \to S_\lambda$, which gives a new representation
$$
\widehat \rho_{\lambda'} = \mu_{\lambda'} \circ \Phi \circ \Tr_\lambda^\omega \circ I \colon \SSS(S_{\lambda'}) \to \End(E_{\lambda'}).
$$
Note that Lemma~\ref{lem:SubdivisAlgebraHom} is here required to guarantee that $\widehat \rho_{\lambda'} $ is an algebra homomorphism. 

The images of these two representations respect the off-diagonal kernel $F_{v_0}' \subset E_{\lambda'}$, by Lemma~\ref{lem:SubdivRespectsOffDiagKernel} for $\widehat \rho_{\lambda'} $ and by Proposition~\ref{prop:OffDiagKernelInvariant} for $\rho_{\lambda'} $. 
 Actually, because $\lambda'$ is not necessarily combinatorial, we need to refer here to the strengthened version of Proposition~\ref{prop:OffDiagKernelInvariant} provided by Remark~\ref{rem:OffDiagKernelInvariantWeakerHyp}. 

As a consequence, $\rho_{\lambda'}$ and $\widehat \rho_{\lambda'}$ induce two representations $\SSS(S_{\lambda'}) \to \End(F_{v_0}')$. We now show that these induced representations coincide. 

\begin{lem}
\label{lem:SubdivisSkein}
The two representations $\rho_{\lambda'}$, $\widehat\rho_{\lambda'} \colon \SSS(S_{\lambda'}) \to \End(E_{\lambda'})$ above are such that
$$
\rho_{\lambda'} \bigl( [K] \bigr) _{|F_{v_0}'} = \widehat \rho_{\lambda'} \bigl( [K] \bigr) _{|F_{v_0}'} 
$$
for every framed link $K \subset S_{\lambda'} \times [0,1]$. 
\end{lem}

\begin{proof}
As in the proof of Proposition~\ref{prop:OffDiagKernelInvariant}, we can arrange that the projection of $K$ to $S$ meets the face $T$ of $\lambda$ along a family of arcs $a_1$, $a_2$, \dots, $a_t$, where $a_l$ meets only the edges $e_{\sigma_l(1)}$, $e_{\sigma_l(2)}$, $e_{n+\sigma_l(3)}'$ of $\lambda'$, for a cyclic permutation $\sigma_l$ of the indices $\{1,2,3\}$. Namely, the situation is as illustrated in Figure~\ref{fig:VertexStar} with $u=3$. Then, still as in the proof of Proposition~\ref{prop:OffDiagKernelInvariant}, the quantum trace $\Tr_{\lambda'}^\omega\bigl([K] \bigr) \in \ZZ(\lambda')$ is a sum of terms of the form
$$
A_1' A_2' \dots A_t' B' C' \in \ZZ(\lambda')
$$
where each term correspond to a state for the boundary of $K\cap T\times [0,1]$, where $C'$ is equal to $C'=[Z_{i_1'}^{\prime\,-1} Z_{i_2'}^{\prime\,-1} \dots Z_{i'_{w'}}^{\prime\, -1}]$ if $K$ crosses the edges $e_{i_1'}'$, $e_{i_2'}'$, \dots, $e_{i_{w'}'}'$ of $\lambda'$, where $B'$ involves only generators $Z_i'$ with $4\leq i \leq n$ (corresponding to edges $e_i'$ of $\lambda'$ contained in the complement of $T$), and where $A_l'$ is the contribution of the arc $a_l$ and is of one of the following three types:
\begin{align*}
A_l' &= Z_{\sigma_l(1)}^{\prime\,2} Z_{\sigma_l(2)}^{\prime\,2}Z_{n+\sigma_l(3)}^{\prime\,2} \\
\text{or } A_l' &= \omega^{-4} Z_{\sigma_l(2)}^{\prime\,2} Z_{n+\sigma_l(3)}^{\prime\,2} +  Z_{\sigma_l(2)}^{\prime\,2}\\
\text{or } A_l' &=1.
\end{align*}

Similarly, $\Tr_{\lambda}^\omega\bigl([K] \bigr) \in \ZZ(\lambda)$ is a sum of terms
$$
A_1 A_2 \dots A_t B C \in \ZZ(\lambda)
$$
corresponding to states for the boundary of $K\cap T\times [0,1]$, where $C$ is equal to $C=[Z_{i_1}^{-1} Z_{i_2}^{-1} \dots Z_{i_w}^{-1}]$ if $K$ crosses the edges $e_{i_1}$, $e_{i_2}$, \dots, $e_{i_{w}}$ of $\lambda$, where $B$ involves only generators $Z_i$ with $4\leq i \leq n$ (corresponding to edges $e_i$ of $\lambda$ contained in the complement of $T$), and where $A_l$ is the contribution of the arc $a_l$ and is of one of the following three types:
\begin{align*}
A_l &= \omega^4 Z_{\sigma_l(1)}^2  Z_{\sigma_l(2)}^2 \\
\text{or } A_l &=  Z_{\sigma_l(2)}^2\\
\text{or } A_l &=1.
\end{align*}

In order to show that $\rho_{\lambda'} \bigl( [K] \bigr) _{|F_{v_0}'} = \mu_{\lambda'} \circ \Tr_{\lambda'}^\omega\bigl([K] \bigr)_{|F_{v_0}'} $ is equal to $ \widehat \rho_{\lambda'} \bigl( [K] \bigr) _{|F_{v_0}'} = \mu_{\lambda'} \circ \Phi \circ \Tr_\lambda^\omega \bigl( [K] \bigr) _{|F_{v_0}'} $, we will compare the respective contributions to these quantities of the terms $A_1' A_2' \dots A_t' B' C'$ and $A_1 A_2 \dots A_t B C$ associated to the same state for the boundary of $K\cap T\times [0,1]$. 

Let $A_1' A_2' \dots A_t' B' C' \in \ZZ(\lambda')$ and $A_1 A_2 \dots A_t B C\in \ZZ(\lambda)$ be the terms of  $\Tr_{\lambda'}^\omega\bigl([K] \bigr)$ and $ \Tr_{\lambda}^\omega\bigl([K] \bigr)$ respectively associated to the same state for the boundary of $K\cap T\times [0,1]$. 

From the definition of the homomorphism $\Phi \colon \ZZ(\lambda) \to \ZZ(\lambda')$, is is immediate $\Phi(B)=B'$. From the observation that each arc $a_l$ contributes a monomial $Z_{\sigma_l(1)}^{\prime\, -1}Z_{\sigma_l(2)}^{\prime\,-1}Z_{n+\sigma_l(3)}^{\prime\, -1}$ to $C'$ and a monomial $Z_{\sigma_l(1)}^{-1}Z_{\sigma_l(2)}^{-1}$ to $C$, it also easily follows from the definition of $\Phi$ that $\Phi(C) = C'$ and $\Phi(B)=B'$. As a consequence, $\mu_{\lambda'} (B')_{|F_{v_0}'} =  \mu_{\lambda'} \circ \Phi(B) _{|F_{v_0}'} $ and $\mu_{\lambda'} (C')_{|F_{v_0}'} =  \mu_{\lambda'} \circ \Phi(C) _{|F_{v_0}'} $

 We need to compare each $\mu_{\lambda'} (A_l')_{|F_{v_0}'} $ to the corresponding $ \mu_{\lambda'} \circ \Phi(A_l) _{|F_{v_0}'} $.

In the case when $A_l'=1$, then $A_l=1$ and of course $\Phi(A_l) = \Phi(1) = 1 =A_l'$. In particular, $\mu_{\lambda'} (A_l')_{|F_{v_0}'} $ is equal to the corresponding $ \mu_{\lambda'} \circ \Phi(A_l) _{|F_{v_0}'} $ in this simple case. 

The case where $A_l'= Z_{\sigma_l(1)}^{\prime\,2} Z_{\sigma_l(2)}^{\prime\,2} Z_{n+\sigma_l(3)}^{\prime\,2}$ is barely more complicated, as the corresponding term is $A_l= \omega^4 Z_{\sigma_l(1)}^2  Z_{\sigma_l(2)}^2 $. Indeed, $\Phi(A_l) = A_l'$, so that $\mu_{\lambda'} (A_l')_{|F_{v_0}'} =\mu_{\lambda'} \circ \Phi(A_l) _{|F_{v_0}'} $ in this case as well. 

The case where $A_l'= \omega^{-4} Z_{\sigma_l(2)}^{\prime\,2} Z_{n+\sigma_l(3)}^{\prime\,2} +  Z_{\sigma_l(2)}^{\prime\,2}$ and $A_l= Z_{\sigma_l(2)}^2$ is much more interesting, because $\Phi(A_l) = Z_{\sigma_l(2)}^{\prime\,2}[Z_{n+\sigma_l(1)}^{\prime} Z_{n+\sigma_l(2)}^{\prime\,-1} Z_{n+\sigma_l(3)}^{\prime}  ]$ looks very different from $A_l'$. We can rewrite these terms as
\begin{align*}
A_l'&=Z_{\sigma_l(2)}^{\prime\,2} (1 + \omega^{-4}  Z_{n+\sigma_l(3)}^{\prime\,2}) = Z_{\sigma_l(2)}^{\prime\,2} Q_{v_0}' -  \omega^{-8} Z_{\sigma_l(2)}^{\prime\,2}Z_{n+\sigma_l(3)}^{\prime\,2}Z_{n+\sigma_l(1)}^{\prime\,2} \\
\text{and }  \Phi(A_l) &= \omega^{-4}  H_{v_0}^{\prime\,-1} Z_{\sigma_l(2)}^{\prime\,2} Z_{n+\sigma_l(3)}^{\prime\,2}Z_{n+\sigma_l(1)}^{\prime\,2}
\end{align*}
for  the off-diagonal term $Q_{v_0}' = 1 + \omega^{-4}  Z_{n+\sigma_l(3)}^{\prime\,2} + \omega^{-8} Z_{n+\sigma_l(3)}^{\prime\,2}Z_{n+\sigma_l(1)}^{\prime\,2}$ and the central element $H_{v_0}' = [Z_{n+\sigma_l(1)}' Z_{n+\sigma_l(2)}'Z_{n+\sigma_l(3)}']$ associated to the vertex $v_0$. Using the properties that $\mu_{\lambda'}(Q_{v_0})_{|F_{v_0}'} =0$ and $\mu_{\lambda'}(H_{v_0}') = -\omega^4 \Id_{E_{\lambda'}}$, it follows that 
$$
\mu_{\lambda'} (A_l')_{|F_{v_0}'} 
=\mu_{\lambda'} ( - \omega^{-8} Z_{\sigma_l(2)}^{\prime\,2}Z_{n+\sigma_l(3)}^{\prime\,2}Z_{n+\sigma_l(1)}^{\prime\,2})_{|F_{v_0}'}
=\mu_{\lambda'} \bigl( \Phi(A_l) \bigr) _{|F_{v_0}'}.
$$

This proves that 
$$
\mu_{\lambda'}(A_1' A_2' \dots A_t' B' C' )_{|F_{v_0}'} 
=\mu_{\lambda'} \circ  \Phi(A_1 A_2 \dots A_t B C)  _{|F_{v_0}'}
$$
whenever the terms $A_1' A_2' \dots A_t' B' C' \in \ZZ(\lambda')$  of  $\Tr_{\lambda'}^\omega\bigl([K] \bigr)$ and $A_1 A_2 \dots A_t B C\in \ZZ(\lambda)$ of $ \Tr_{\lambda}^\omega\bigl([K] \bigr)$ correspond to the same state for the boundary of $K\cap T\times [0,1]$. As a consequence,
\begin{equation*}
\rho_{\lambda'} \bigl( [K] \bigr) _{|F_{v_0}'} = \mu_{\lambda'} \circ \Tr_{\lambda'}^\omega\bigl([K] \bigr)_{|F_{v_0}'}  = \mu_{\lambda'} \circ \Phi \circ \Tr_\lambda^\omega \bigl( [K] \bigr) _{|F_{v_0}'}= \widehat \rho_{\lambda'} \bigl( [K] \bigr) _{|F_{v_0}'}. \qedhere
\end{equation*}

\end{proof}

We now return to the  irreducible representations  $\mu_{\lambda} \colon \ZZ(\lambda) \to \End({E_\lambda})$ and $\mu_{\lambda'} \colon \ZZ(\lambda') \to \End(E_{\lambda'})$, respectively associated by Proposition~\ref{prop:ConstructRepCheFock} to the $\lambda$--enhancement $\xi \colon \widetilde V_{\lambda} \to \mathbb{CP}^1$ and the $\lambda'$--enhancement $\xi' \colon \widetilde V_{\lambda'} \to \mathbb{CP}^1$ for  the  homomorphism $r\colon \pi_1(S) \to \SL(\C)$. Recall that $\xi$ is just the restriction of  $\xi'$ to $ \widetilde V_{\lambda} \subset \widetilde V_{\lambda'} $.

By Lemma~\ref{lem:SubdivRespectsOffDiagKernel}, the composition $\mu_{\lambda'} \circ \Phi \colon \ZZ(\lambda) \to \End(E_{\lambda'})$ respects the off-diagonal kernel $F_{v_0}' \subset E_{\lambda'}$, and therefore induces a representation $\bar \mu_\lambda \colon \ZZ(\lambda) \to \End(F_{v_0}')$ by the property that $\bar\mu_\lambda(Z) = \mu_{\lambda'} \circ \Phi(Z)_{F_{v_0}'}$ for every $Z\in \ZZ(\lambda)$.

\begin{lem}
\label{lem:SubdivIsomorphicCheFockReps}
After pre-composing $\mu_\lambda$ with the action of a sign-reversal symmetry of $r\in \RR(S)$ if necessary, the representations $\mu_{\lambda} \colon \ZZ(\lambda) \to \End({E_\lambda})$ and $\bar \mu_\lambda \colon \ZZ(\lambda) \to \End(F_{v_0}')$ are isomorphic. 
\end{lem}

\begin{proof}
By the uniqueness statement in Proposition~\ref{prop:ConstructRepCheFock}, it suffices to check that $\bar \mu_\lambda \colon \ZZ(\lambda) \to \End(F_{v_0}')$ satisfies the following four properties, which characterize $\mu_\lambda$:
\begin{enumerate}
\item $\dim F_{v_0}' = N^{3g+p_\lambda-3}$;

\item  $\bar \mu_{\lambda} (Z_i^{2N}) = x_i \,\Id_{F_{v_0}'}$ for each $i=1$, $2$, \dots, $n$, where $x_i$ is the crossratio weight associated to the edge $e_i$ of $\lambda$ by the $\lambda$--enhancement $\xi$;

\item $\bar \mu_{\lambda} (H_v) = -\omega^4\, \Id_{F_{v_0}'}$ for every vertex $v$ of $\lambda$;

\item $ T_N \Bigl(\bar  \mu_{\lambda} \circ \Tr_\lambda^\omega \bigl([K] \bigr)\Bigr) = - \Tr\, r(K) \, \Id_{F_{v_0}'} $ for every framed knot $K\subset S_\lambda \times [0,1]$ whose projection to $K$ has no crossing and whose framing is vertical.  

\end{enumerate}

The first property (1) is proved in Lemma~\ref{lem:SubdivisionDimOffDiagKernel}. 

For the second property (2) the case where $i>3$, namely the case where the edge $e_i$ is not a side of the face $T$ that is being subdivided, is somewhat trivial. Indeed, $\bar \mu_\lambda (Z_i^{2N}) = \mu_{\lambda'} \circ \Phi(Z_i^{2N}) =  \mu_{\lambda'}(Z_i^{\prime\, 2N}) = x_i' \, \Id_{F_{v_0}'} =x_i \, \Id_{F_{v_0}'}$ as the enhancements $\xi$ and $\xi'$ associate the same crossratio weight $x_i=x_i'$ to the edge $e_i$. 

The cases where $i\leq 3$ require a geometric argument. For instance,
\begin{align*}
\bar \mu_\lambda (Z_1^{2N}) &= \mu_{\lambda'} \circ \Phi(Z_1^{2N}) =  \mu_{\lambda'}\bigl([Z_1^{\prime\, 2N} Z_{n+1}^{\prime\, -N} Z_{n+2}^{\prime\, N}  Z_{n+3}^{\prime\, N} ]  \bigr  )\\
&=  \mu_{\lambda'}\bigl( Z_1^{\prime\, 2N} Z_{n+1}^{\prime\, -2N} H_{v_0}'^{\prime\, N}\bigr  ) = - x_1' x_{n+1}^{\prime\, -1}\, \Id_{F_{v_0}'}
\end{align*}
since $\mu_{\lambda'}( Z_1^{\prime\, 2N})=x_1'  \, \Id_{F_{v_0}'}$, $\mu_{\lambda'}(  Z_{n+1}^{\prime\, 2N}) = x_{n+1}'  \, \Id_{F_{v_0}'}$, $\mu_{\lambda'}(H_{v_0}' ) =-\omega^4 \, \Id_{F_{v_0}'}$ and $\omega^{4N}=1$. Going back to the definition of the crossratio weights, a computation shows that $x_1' x_{n+1}^{\prime\, -1}= -x_1$. It follows that $\bar \mu_\lambda (Z_1^{2N}) = x_1  \, \Id_{F_{v_0}'}$, as required. 

Identical computations show that $\bar \mu_\lambda (Z_2^{2N}) = x_2  \, \Id_{F_{v_0}'}$ and $\bar \mu_\lambda (Z_3^{2N}) = x_3  \, \Id_{F_{v_0}'}$, and complete the proof of (2) in all cases. 

By definition of the homomorphism $\Phi $, it sends the central element $H_v \in \ZZ(\lambda)$ associated to a vertex $v$ of $\lambda$ to the central element $H_v' \in \ZZ(\lambda')$ associated to $v$ considered as a vertex of $\lambda'$. It follows that $\bar \mu_{\lambda} (H_v) =  \mu_{\lambda'} (H_v') =-\omega^4\, \Id_{F_{v_0}'}$. This proves the third property (3). 

Finally, (4) is a consequence of Lemma~\ref{lem:SubdivisSkein}. Indeed, for every framed knot $K\subset S_{\lambda'} \times [0,1]$ whose projection to $K$ has no crossing and whose framing is vertical
\begin{align*}
 T_N \Bigl(\bar  \mu_{\lambda} \circ \Tr_\lambda^\omega \bigl([K] \bigr)\Bigr)
 &= \bar  \mu_{\lambda} \circ \Tr_\lambda^\omega \Bigl(T_N \bigl([K] \bigr)  \Bigr)
 =  \widehat \rho_{\lambda'}\Bigl(T_N \bigl([K] \bigr)  \Bigr) _{|F_{v_0}'}
  =   \rho_{\lambda'}\Bigl(T_N \bigl([K] \bigr)  \Bigr) _{|F_{v_0}'}\\
   &=   \mu_{\lambda'} \circ \Tr_{\lambda'}^\omega\Bigl(T_N \bigl([K] \bigr)  \Bigr) _{|F_{v_0}'}
   =
   T_N \Bigl(  \mu_{\lambda'} \circ \Tr_{\lambda'}^\omega \bigl([K] \bigr)\Bigr)_{|F_{v_0}'}
   =  - \Tr\, r(K) \, \Id_{F_{v_0}'}
\end{align*}
where the first and fifth equalities come from the fact that all maps involved are algebra homomorphisms, where the second equality comes from the definitions of the representations $\widehat \rho_{\lambda'}$ and $\bar \mu_\lambda$, where the third equality is provided by  Lemma~\ref{lem:SubdivisSkein}, and where the last equality is part of the definition of $\mu_{\lambda'}$ by Proposition~\ref{prop:ConstructRepCheFock}. 

This proves that the representation $\bar \mu_\lambda \colon \ZZ(\lambda) \to \End(F_{v_0}')$ satisfies the properties (1--4) listed above. By Proposition~\ref{prop:ConstructRepCheFock}, it follows that $\bar \mu_\lambda$ is isomorphic to $\mu_\lambda$. 
\end{proof}

Our last step is to show that the isomorphism provided by Lemma~\ref{lem:SubdivIsomorphicCheFockReps} is compatible with off-diagonal kernels.

\begin{lem}
\label{lem:SubdivisOffDiagKernelConstant}
For  a vertex $v$ of the triangulation $\lambda$,  let  $F_v \subset {E_\lambda}$ and $F_{v}' \subset E_{\lambda'}$ be the respective off-diagonal kernels of $v$ for the representations $\mu_\lambda$ and $\mu_{\lambda'}$, defined by considering $v$ as a vertex of both $\lambda$ and $\lambda'$. Then,  the isomorphism $E_\lambda \to F_{v_0}'$ between the representations $\mu_{\lambda} \colon \ZZ(\lambda) \to \End({E_\lambda})$ and $\bar \mu_\lambda \colon \ZZ(\lambda) \to \End(F_{v_0}')$ provided by Lemma~{\upshape\ref{lem:SubdivIsomorphicCheFockReps}} sends $F_v$ to $ F_{v}' \cap F_{v_0}'$. 
\end{lem}

\begin{proof} We first simplify the situation a little. The representation $\mu_{\lambda} \colon \ZZ(\lambda) \to \End({E_\lambda})$ is only defined up to isomorphism and up to sign-reversal symmetry by Proposition~\ref{prop:ConstructRepCheFock}. Modifying it by a sign-reversal symmetry if necessary, and by the isomorphism of Lemma~\ref{lem:SubdivIsomorphicCheFockReps}, we can consequently assume that it is equal, not just isomorphic, to $\bar \mu_\lambda \colon \ZZ(\lambda) \to \End(F_{v_0}')$. In particular, $E_\lambda = F_{v_0}'$, $\mu_\lambda= \bar \mu_\lambda$ and the isomorphism is the identity.

Note that the modification of $\mu_\lambda$ by a sign-reversal symmetry does not change the off-diagonal kernel $F_v$, as the off-diagonal term $Q_v\in \ZZ(\lambda)$ involves only even powers of the generators $Z_i$. We consequently have to show that  $F_v =  F_{v}' \cap F_{v_0}'$ once we have arranged that the representations $\mu_\lambda$ and $\bar \mu_\lambda$ coincide.  

If $v$ is not one of the vertices of the face $T$ of $\lambda$ that is being subdivided, the expression of the off-diagonal term $Q_v\in \ZZ(\lambda)$ involves only generators $Z_i$ with $i>3$, and  $\Phi (Q_v) \in \ZZ(\lambda')$ is obtained from $Q_v\in \ZZ(\lambda)$ by replacing each generator $Z_i$ by $Z_i'$. Then we can choose the off-diagonal term $Q_{v}' \in \ZZ(\lambda')$ to be equal to $\Phi (Q_v)$. Then, 
\begin{align*}
F_v &= \ker \mu_\lambda(Q_v)  = \ker \bar \mu_\lambda (Q_v)= \ker \mu_{\lambda'} \circ \Phi(Q_v)_{ |F_{v_0}'}\\
& = \ker  \mu_{\lambda'} (Q_v')_{ |F_{v_0}'} = \bigl( \ker  \mu_{\lambda'} (Q_v') \bigr)\cap F_{v_0}'
=F_v' \cap F_{v_0}'.
\end{align*}

When $v$ belongs to  the face $T$ of $\lambda$, this case splits into three subcases according to whether $v$ corresponds to 1, 2 or 3 vertices of the triangle $T$. We  restrict our discussion to the subcase where $v$ corresponds to two vertices of $T$. The other two subcases are very similar. 

Without loss of generality, we can choose the edge indexing of Figure~\ref{fig:TriangSubdivis} so that both endpoints of the edge $e_1$ are equal to the vertex $v$. Then,  the off-diagonal term of $v$ starting at the edge $e_2$ can be written as
$$
Q_v = 1 + \omega^{-4} Z_2^2 + Z_2^2Z_1^2B + Z_2^2Z_1^2CZ_1^2 + Z_2^2Z_1^2CZ_1^2 Z_3^2 D
$$
where $B$, $C$ and $D$ are polynomials in the variables $Z_i^2$ with $4\leq i\leq n$; namely, these $Z_i^2$ correspond to edges of $\lambda$ that are not contained in the face $T$. 

Similarly, if we start from the edge $e_2'$, the off-diagonal term of the vertex $v$ in $\ZZ(\lambda')$ is
\begin{align*}
Q_{v'}' = 1 + \omega^{-4} Z_2^{\prime\, 2} &+ \omega^{-8} Z_2^{\prime\, 2}Z_{n+3}^{\prime\, 2} + \omega^{-4}Z_2^{\prime\, 2} Z_{n+3}^{\prime\, 2} Z_1^{\prime\, 2} B' \\
&+  \omega^{-4}Z_2^{\prime\, 2} Z_{n+3}^{\prime\, 2} Z_1^{\prime\, 2}  C' Z_1^{\prime\, 2} + \omega^{-8}Z_2^{\prime\, 2} Z_{n+3}^{\prime\, 2} Z_1^{\prime\, 2}  C' Z_1^{\prime\, 2} Z_{n+2}^{\prime\, 2}  \\
& +\omega^{-8}Z_2^{\prime\, 2} Z_{n+3}^{\prime\, 2}Z_1^{\prime\, 2}  C' Z_1^{\prime\, 2} Z_{n+2}^{\prime\, 2}  Z_3^{\prime\, 2} D'
\end{align*}
where $B'$, $C'$, $D'\in \ZZ(\lambda')$ are respectively obtained from $B
$, $C$, $D \in \ZZ(\lambda)$ by replacing each $Z_i^2$ with $Z_i^{\prime\, 2}$. 

By definition of the homomorphism $\Phi \colon \ZZ(\lambda) \to \ZZ(\lambda')$, 
\begin{align*}
\Phi(Z_2^2) &= \omega^{3} Z_2^{\prime\, 2} Z_{n+1}' Z_{n+2}^{\prime\, -1} Z_{n+3}' 
& \Phi (\omega^{-4}Z_2^2Z_1^2) &=\omega^{-8}  Z_2^{\prime\, 2} Z_{n+3}^{\prime\, 2} Z_1^{\prime\, 2}\\
\Phi(Z_1^2) &= \omega^{3} Z_1^{\prime\, 2} Z_{n+3}' Z_{n+1}^{\prime\, -1} Z_{n+2}' 
& \Phi (\omega^{-4} Z_1^2Z_3^2) &= \omega^{-8} Z_1^{\prime\, 2} Z_{n+2}^{\prime\, 2} Z_3^{\prime\, 2}\\
\Phi(B) & = B'  \quad \quad\quad\quad\Phi(C) =C' & \Phi(D)&= D'.
\end{align*}
Therefore,
\begin{align*}
\Phi(Q_v) = 1 
&+ \omega^{-1} Z_2^{\prime\, 2} Z_{n+1}' Z_{n+2}^{\prime\, -1} Z_{n+3}' 
+ \omega^{-4} Z_2^{\prime\, 2} Z_{n+3}^{\prime\, 2} Z_1^{\prime\, 2} B'\\
& + \omega^{-1} Z_2^{\prime\, 2} Z_{n+3}^{\prime\, 2} Z_1^{\prime\, 2} C' Z_1^{\prime\, 2} Z_{n+3}' Z_{n+1}^{\prime\, -1} Z_{n+2}' \\
&  + \omega^{-8} Z_2^{\prime\, 2} Z_{n+3}^{\prime\, 2} Z_1^{\prime\, 2} C'  Z_1^{\prime\, 2} Z_{n+2}^{\prime\, 2} Z_3^{\prime\, 2} D' .
\end{align*}
The above expressions of $ Q_{v'}' $ and $ \Phi(Q_v)$ share several terms, and their difference can therefore be expressed as  
\begin{align*}
Q_{v'}' - \Phi(Q_v)& = Z_2^{\prime\, 2} ( \omega^{-4} + \omega^{-8}Z_{n+3}^{\prime\, 2} -  \omega^{-1} Z_{n+1}' Z_{n+2}^{\prime\, -1} Z_{n+3}'  ) \\
&\qquad\qquad +  Z_2^{\prime\, 2} Z_{n+3}^{\prime\, 2} Z_1^{\prime\, 2} C' Z_1^{\prime\, 2} 
(\omega^{-4}+ \omega^{-8}  Z_{n+2}^{\prime\, 2} 
-  \omega^{-1}  Z_{n+3}' Z_{n+1}^{\prime\, -1} Z_{n+2}') \\
&=  Z_2^{\prime\, 2} Z_{n+2}^{\prime\, -2}( \omega^{-4}Z_{n+2}^{\prime\, 2} + \omega^{-8}Z_{n+2}^{\prime\, 2} Z_{n+3}^{\prime\, 2} -  \omega^{-3}  Z_{n+2}' Z_{n+1}' Z_{n+3}'  ) \\
&\qquad\qquad +  Z_2^{\prime\, 2} Z_{n+3}^{\prime\, 2} Z_1^{\prime\, 2} C' Z_1^{\prime\, 2} Z_{n+1}^{\prime\, -2}
(\omega^{-4} Z_{n+1}^{\prime\, 2}  + \omega^{-8}  Z_{n+1}^{\prime\, 2}  Z_{n+2}^{\prime\, 2} \\
&\qquad\qquad\qquad\qquad\qquad\qquad\qquad\qquad\qquad\qquad -  \omega^{-3}  Z_{n+1}' Z_{n+3}'  Z_{n+2}') \\
&= Z_2^{\prime\, 2} Z_{n+2}^{\prime\, -2}( Q_{v_0}' -1 -\omega^{-4} H_{v_0}') \\
&\qquad\qquad +  Z_2^{\prime\, 2} Z_{n+3}^{\prime\, 2} Z_1^{\prime\, 2} C' Z_1^{\prime\, 2} Z_{n+1}^{\prime\, -2}
( Q_{v_0}'' -1 -\omega^{-4} H_{v_0}') 
\end{align*}
for
\begin{align*}
Q_{v_0}' &= 1+  \omega^{-4}Z_{n+2}^{\prime\, 2} + \omega^{-8}Z_{n+2}^{\prime\, 2} Z_{n+3}^{\prime\, 2}\\
Q_{v_0}'' &= 1+  \omega^{-4} Z_{n+1}^{\prime\, 2}  + \omega^{-8}  Z_{n+1}^{\prime\, 2}  Z_{n+2}^{\prime\, 2}\\
H_{v_0}' &= \omega^{-1} Z_{n+2}' Z_{n+1}' Z_{n+3}' = \omega^{-1} Z_{n+1}' Z_{n+3}'  Z_{n+2}'.
\end{align*}

Note that $Q_{v_0}'$ and $Q_{v_0}''\in \ZZ(\lambda')$ are two off-diagonal terms for the vertex $v_0$, corresponding to different indexings of the edges around this vertex. As a consequence, $\mu_{\lambda'}(Q_{v_0}')_{|F_{v_0}'} = \mu_{\lambda'}(Q_{v_0}')_{|F_{v_0}'} = 0$ while $\mu_{\lambda'}(H_{v_0}')_{|F_{v_0}'} = -\omega^4 \, \Id_{F'_{v_0}}$. It consequently follows from the above computation that
$$
\mu_{\lambda'}(Q_{v'}')_{|F_{v_0}'}  - \mu_{\lambda'} \circ \Phi(Q_v)_{|F_{v_0}'}  =0. 
$$
 Then, as in the first case considered,  
\begin{align*}
F_v &= \ker \mu_\lambda(Q_v)  = \ker \bar \mu_\lambda (Q_v)= \ker \mu_{\lambda'} \circ \Phi(Q_v)_{ F_{v_0}'}\\
& = \ker  \mu_{\lambda'} (Q_v')_{ F_{v_0}'} = \bigl( \ker  \mu_{\lambda'} (Q_v') \bigr)\cap F_{v_0}'
=F_v' \cap F_{v_0}'.
\end{align*}

This concludes the proof of Lemma~\ref{lem:SubdivisOffDiagKernelConstant} when the vertex $v$ corresponds to two vertices of the triangle $T$. The  cases where it corresponds to one or three vertices of $T$ are very similar, and we omit the corresponding  proofs. 
\end{proof}

We now gather the results of this section in the following statement, which we state in inductive form for later use in \S \ref{subsect:ConstructSkeinRepArbitraryTriangulation}. 

It is convenient to introduce some terminology. If the representation $\rho_\lambda \colon \SSS(S_\lambda) \to \End(E_\lambda)$ respects the total off-diagonal kernel $F_\lambda \subset E_\lambda$, we say that $\rho_\lambda$ \emph{induces a representation} $\check\rho_\lambda \colon \SSS(S) \to \End(F_\lambda)$ if $\rho_\lambda\bigl( [K] \bigr)_{|F_\lambda}= \rho_\lambda\bigl( [K'] \bigr)_{|F_\lambda}$ whenever the two framed links $K$, $K' \subset S_\lambda \times [0,1]$ are isotopic in $S\times [0,1]$. 

For instance, when the triangulation $\lambda$ is combinatorial, Proposition~\ref{prop:OffDiagKernelInvariant} shows that the representation $\rho_\lambda  \colon \SSS(S_\lambda) \to \End(E_\lambda)$ respects $F_\lambda$, and Proposition~\ref{prop:VertexSweep} implies that $\rho_\lambda$ {induces a representation} $\check\rho_\lambda \colon \SSS(S) \to \End(F_\lambda)$.

\begin{prop}
\label{prop:SubdivisionInvariance}
Let $\lambda'$ be obtained from the triangulation  $\lambda$ of the surface $S$ by subdividing a face into three triangles as in Figure~{\upshape\ref{fig:TriangSubdivis}}, let  $\xi' \colon \widetilde V_{\lambda'} \to \mathbb{CP}^1$  be a $\lambda'$--enhancement for  the  homomorphism $r\colon \pi_1(S) \to \SL(\C)$, and let $\xi \colon \widetilde V_{\lambda} \to \mathbb{CP}^1$ be the  $\lambda$--enhancement defined by restriction of $\xi'$ to $\widetilde V_\lambda \subset \widetilde V_{\lambda'}$. Let $\mu_\lambda \colon \ZZ(\lambda) \to \End({E_\lambda})$ and $\mu_{\lambda'} \colon \ZZ(\lambda') \to \End(E_{\lambda'})$ be the representations respectively associated to $\xi$ and $\xi'$ by Proposition~{\upshape\ref{prop:ConstructRepCheFock}}. 

Suppose in addition that $\rho_{\lambda'} = \mu_{\lambda'} \circ \Tr_{\lambda'}^\omega \colon \SSS(S_{\lambda'}) \to \End(E_{\lambda'})$ respects the total off-diagonal kernel $F_{\lambda'} \subset E_{\lambda'}$ of $\mu_{\lambda'}$, and induces a representation $\check\rho_{\lambda'} \colon \SSS(S) \to \End(F_{\lambda'})$ as above. Then, $\rho_{\lambda} = \mu_{\lambda} \circ \Tr_{\lambda}^\omega \colon \SSS(S_{\lambda}) \to \End(E_{\lambda})$ respects the total off-diagonal kernel ${F_\lambda} \subset {E_\lambda}$ of $\mu_\lambda$, and induces a representation $\check\rho_\lambda \colon \SSS(S) \to \End({F_\lambda})$. Moreover, $\check\rho_\lambda$  is isomorphic to $\check\rho_{\lambda'}$ after a possible pre-composition with the action of a sign-reversal symmetry of $r\in \RR(S)$ on $\SSS(S)$. 
\end{prop}

\begin{proof} After pre-composition with the action of a sign-reversal symmetry of $r\in \RR(S)$ on $\ZZ(
\lambda)$, Lemma~\ref{lem:SubdivIsomorphicCheFockReps} provides an isomorphism between the two representations  $\mu_{\lambda} \colon \ZZ(\lambda) \to \End({E_\lambda})$ and $\bar \mu_\lambda \colon \ZZ(\lambda) \to \End(F_{v_0}')$. Note that this modification of $\mu_\lambda$ does not change its total off-diagonal kernel $F_\lambda$, as a sign-reversal symmetry respects each off-diagonal term $Q_v\in \ZZ(\lambda)$. 

As in the beginning of the proof of Lemma~\ref{lem:SubdivisOffDiagKernelConstant}, we can arrange without loss of generality that this isomorphism is the identity, so that $\mu_\lambda=\bar\mu_\lambda$. Under these conditions, we want to prove that $F_\lambda = F_{\lambda'}$, and that $\check\rho_\lambda = \check\rho_{\lambda'}$.

We first compare the two total off-diagonal kernels $F_\lambda $ and $ F_{\lambda'}$. Lemma~\ref{lem:SubdivisOffDiagKernelConstant} shows that $F_v =  F_{v}' \cap F_{v_0}'$  for every vertex $v$ of $\lambda$. Then
$$
F_\lambda = \bigcap_{v \in V_\lambda}F_v = \bigcap_{v \in V_\lambda} (F_{v}' \cap F_{v_0}') 
= F_{v_0}' \cap  \bigcap_{v \in V_\lambda} F_{v}' = \bigcap_{v' \in V_{\lambda'}} F_{v'}' = F_{\lambda'}.
$$

Also, by our assumption that the isomorphism between $\mu_\lambda$ and $\bar\mu_\lambda$  is the identity, $E_\lambda = F_{v_0}' \subset E_{\lambda'}$. For every framed link $K \subset S_{\lambda'} \times [0,1]$, the fact that $\mu_\lambda = \bar\mu_\lambda$ and the definition of $\bar\mu_\lambda$ imply that 
$$
\rho_\lambda \bigl( [K] \bigr) = \mu_{\lambda} \circ \Tr_{\lambda}^\omega  \bigl( [K] \bigr) = 
\mu_{\lambda'} \circ \Phi \circ \Tr_{\lambda}^\omega  \bigl( [K] \bigr) _{|F_{v_0}'} = \rho_{\lambda'} \bigl( [K] \bigr)_{|F_{v_0}'},
$$
where the last equality is provided by  Lemma~\ref{lem:SubdivisSkein}, and where we use  the same notation for the skeins $[K] \in \SSS(S_{\lambda'})$ and $[K] = I \bigl( [K] \bigr) \in \SSS(S_\lambda)$. 

In particular, since $\rho_{\lambda'} \bigl( [K] \bigr)\in \End(E_{\lambda'})$ respects the total off-diagonal kernel $F_{\lambda'}$ by hypothesis, then $\rho_\lambda \bigl( [K] \bigr) = \rho_{\lambda'} \bigl( [K] \bigr)_{|F_{v_0}'}$ respects $F_\lambda$ since $F_\lambda = F_{\lambda'} \subset F_{v_0}' \subset E_{\lambda'}$. 

The same equality $\rho_\lambda \bigl( [K] \bigr) = \rho_{\lambda'} \bigl( [K] \bigr)_{|F_{v_0}'}$ shows that $\rho_\lambda \bigl( [K] \bigr)_{|F_\lambda} = \rho_{\lambda'} \bigl( [K] \bigr)_{|F_\lambda}$ since $F_\lambda  \subset F_{v_0}' $. Therefore,  if $K$, $K' \subset S_{\lambda'} \times [0,1]$ are isotopic in $S\times [0,1]$, 
$$
\rho_\lambda \bigl( [K] \bigr)_{|F_\lambda} = \rho_{\lambda'} \bigl( [K] \bigr)_{|F_\lambda}
= \rho_{\lambda'} \bigl( [K'] \bigr)_{|F_\lambda} = \rho_{\lambda} \bigl( [K'] \bigr)_{|F_\lambda}
$$
where the second equality comes from the hypothesis that $\rho_{\lambda'}$ induces a representation $\check\rho_{\lambda'} \colon \SSS(S) \to \End(F_{\lambda'})$ and the fact that $F_\lambda = F_{\lambda'}$. As a consequence, $\rho_{\lambda}$ induces a representation $\check\rho_{\lambda} \colon \SSS(S) \to \End(F_{\lambda})$. 

Finally, the properties that $F_\lambda = F_{\lambda'}$ and $\rho_\lambda \bigl( [K] \bigr)_{|F_v} = \rho_{\lambda'} \bigl( [K] \bigr)_{|F_v}$ show that  $\check\rho_{\lambda}=\check\rho_{\lambda'}$.
\end{proof}

\subsection{Diagonal exchanges} 
\label{subsect:DiagonalExchanges}
Diagonal exchanges (also called flips) are triangulation moves that occur in many different contexts. The arguments in this section are very similar to those used for earlier results in quantum Teichm\"uller theory \cite{CheFoc1, CheFoc2, BonLiu, Liu}. In particular, this section is conceptually and technically much simpler than the previous one. 

Let $\lambda$ and $\lambda'$ be two triangulations of $S$ which have the same vertices, and which differ only in one edge. We can index the edges of $\lambda$ as $e_1$, $e_2$, \dots, $e_n$, and the edges of $\lambda'$ as $e_1'$, $e_2'$, \dots, $e_n'$ in such a way that $e_i = e_i'$ when $i\geq 2$. Then, the two faces of $\lambda$ containing the edge $e_1$ form a ``square'' $Q$ as in Figure~\ref{fig:DiagExchange}, and $e_1'$ is the other diagonal of the square $Q$. In this case, we say that $\lambda$ and $\lambda'$ differ by a \emph{diagonal exchange}.

\begin{figure}[htbp]
\vskip 5pt
\SetLabels
( .45* .55) $e_1 $ \\
(.5 *1.05 ) $ e_2$ \\
( 1.1*.5 ) $e_3 $ \\
(.5 * -.1) $e_4 $ \\
( -.1*.5 ) $e_5$ \\
(.5 *-.3 ) The triangulation $ \lambda$ \\
\endSetLabels
\centerline{\AffixLabels{\includegraphics{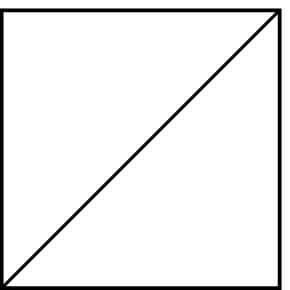}}
\hskip 60pt
\SetLabels
( .55* .55) $e_1' $ \\
(.5 *1.05 ) $ e_2'$ \\
( 1.1*.5 ) $e_3' $ \\
(.5 * -.1) $e_4' $ \\
( -.1*.5 ) $e_5' $ \\
(.5 *-.3 ) The triangulation $ \lambda'$ \\
\endSetLabels
\AffixLabels{\rotatebox{90}{\includegraphics{DiagExchange.eps}}}}
\vskip 20pt
\caption{}
\label{fig:DiagExchange}
\end{figure}

Let $\xi \colon \widetilde V_\lambda \to \CP$ be a $\lambda$--enhancement for the homomorphism  $ r \colon \pi_1(S) \to \SL(\C)$. 

We assume that the following conditions are satisfied:
\begin{enumerate}
\item The $\lambda$--enhancement $\xi$ is also a $\lambda'$--enhancement for $r$. Since the triangulations $\lambda$ and $\lambda'$ have the same vertex sets $V_\lambda =V_{\lambda'} \subset S$, this just means that $\xi \colon \widetilde V_\lambda \to \CP$ assigns distinct values  to the  endpoints of an arbitrary  lift of $e_1'$.

\item The four sides of the square $Q$, formed by the two faces of $\lambda$ containing the edge $e_1$, correspond to distinct edges of $\lambda$.
\end{enumerate}

The second condition is not essential, but will simplify our exposition by dispensing us from the need to consider many cases,  as was required in \cite{BonLiu, Liu}. Note that we are allowing identifications between the corners of $Q$ which, for instance, could very well correspond to the same vertex of $\lambda$. The first condition is really critical.

This second condition enables us to index the edges of $\lambda$ and $\lambda'$ so that the sides of the square $Q$ are $e_2=e_2'$, $e_3=e_3'$, $e_4=e_4'$, $e_5=e_5'$, as in Figure~\ref{fig:DiagExchange}.

In \cite[\S6]{Hiatt}, Chris Hiatt constructs a natural isomorphism $\Theta_{\lambda\lambda'} \colon \widehat{\mathcal Z}^\omega(\lambda') \to  \widehat{\mathcal Z}^\omega(\lambda)$ between the fraction algebras $\widehat{\mathcal Z}^\omega(\lambda') $ and $\widehat{\mathcal Z}^\omega(\lambda) $ of the balanced Chekhov-Fock algebras $\ZZ(\lambda') $ and $\ZZ(\lambda) $.  The elements of $\widehat{\mathcal Z}^\omega(\lambda) $ are formal fractions $UV^{-1}$ with $U$, $V\in \ZZ(\lambda)$ and $V\neq 0$, and are manipulated with the usual rules for fractions (except that the noncommutative context can greatly complicate computations, in particular when one needs to reduce two fractions to a common denominator in order to add them). 

The homomorphism $\Theta_{\lambda\lambda'} \colon \widehat{\mathcal Z}^\omega(\lambda') \to  \widehat{\mathcal Z}^\omega(\lambda)$ is defined as follows. Because of the exponent parity condition defining the balanced Chekhov-Fock algebra, every monomial $Z'$ of $\ZZ(\lambda')$ can be uniquely written as 
$$ Z'= (Z_1'Z_2'Z_4')^{k_1} (Z_2'Z_3')^{k_2} (Z_4'Z_5')^{k_3} Z_2^{\prime\, 2k_4} Z_4^{\prime\, 2k_5} B'$$
for exponents $k_i\in \Z$ and for a monomial $B'$ involving only generators $Z_i'$ with $i>5$.  Then $\Theta_{\lambda\lambda'} $ is uniquely determined by the property that 
\begin{multline*}
\Theta_{\lambda\lambda'} (Z') = \bigl( \omega^4 Z_1Z_2Z_4 +  Z_1^{-1} Z_2Z_4 \bigr)^{k_1} (\omega Z_1Z_2Z_3) ^{k_2} (\omega Z_1Z_4Z_5)^{k_3}\\
 \bigl(Z_2^2 + \omega^4 Z_1^2 Z_2^2\bigr)^{k_4} \bigl(Z_4^2 + \omega^4 Z_1^2 Z_4^2\bigr)^{k_5} B
\end{multline*}
where $B$ is obtained from $B'$ by replacing each generator $Z_i'$ with $i>5$ by $Z_i$. The fact that this really defines an algebra  homomorphism  $\Theta_{\lambda\lambda'} \colon \widehat{\mathcal Z}^\omega(\lambda') \to  \widehat{\mathcal Z}^\omega(\lambda)$ is proved in \cite{Hiatt}. 

\begin{lem}
\label{lem:CheFockCoordinateChanges}
\begin{align*}
 \Theta_{\lambda\lambda'}(Z_1^{\prime\, 2} )&= Z_1^{-2}  {}\kern -2cm{}
 &
\Theta_{\lambda\lambda'}(Z_2^{\prime\, 2}) &=  \bigl(1+ \omega^4 Z_1^2 \bigr) Z_2^2 
\\
\Theta_{\lambda\lambda'}(Z_3^{\prime\, 2} )&= \bigl(1+ \omega^4 Z_1^{-2} \bigr)^{-1}Z_3^2 
&
\Theta_{\lambda\lambda'}(Z_4^{\prime\, 2}) &=  \bigl(1+ \omega^4 Z_1^2 \bigr) Z_4^2 
\\
\Theta_{\lambda\lambda'}(Z_5^{\prime\, 2} )&= \bigl(1+ \omega^4 Z_1^{-2} \bigr)^{-1}Z_5^2 
&
\text{and } \Theta_{\lambda\lambda'}(Z_i^{\prime\, 2}) &= Z_i^2 \text{ for every } i>5. 
\end{align*}

\end{lem}
\begin{proof}
This is a simple computation based on the formula defining $\Theta_{\lambda\lambda'}$: Use the property that $Z_1^{\prime\, 2} = \omega^{a+4} (Z_1'Z_2'Z_4')^2 Z_2^{\prime\, -2} Z_4^{\prime\, -2}$, $Z_3^{\prime\, 2}= \omega^{b-6}(Z_2'Z_3')^2 Z_2^{\prime\, -2}$ and $Z_5^{\prime\, 2}= \omega^{c-6}(Z_4'Z_5')^2 Z_4^{\prime\, -2}$, where the integers $a$, $b$, $c \in \Z$ are determined by the faces of $\lambda$ that are adjacent to several faces of the square $Q$ and are not contained in $Q$ (and contribute additional terms to the skew-commutativity relations between the $Z_i'$, and between the $Z_i$). See also \cite[\S 6]{Hiatt}, which explains that $\Theta_{\lambda\lambda'}$ was designed as a `square root' of the Chekhov-Fock coordinate change of \cite{CheFoc1, CheFoc2, BonLiu, Liu}. 
\end{proof}

Let $\mu_\lambda \colon \ZZ(\lambda)\to \End({E_\lambda})$ and $\mu_{\lambda'} \colon \ZZ(\lambda') \to \End(E_{\lambda'})$  be the representations associated to the enhanced character $(r, \xi)$ by Proposition~\ref{prop:ConstructRepCheFock}. We would like to consider the representation $\mu_{\lambda} \circ \Theta_{\lambda\lambda'} \colon \ZZ(\lambda') \to \End({E_\lambda})$. But this composition is not immediately defined, as $\Theta_{\lambda\lambda'}$ is valued in the fraction algebra $ \widehat{\mathcal Z}^\omega(\lambda)$, whereas $\mu_\lambda$ is only defined on the Chekhov-Fock algebra $\ZZ(\lambda) \subset  \widehat{\mathcal Z}^\omega(\lambda)$. 

\begin{lem}
\label{lem:CrossratiosDiagExchange}
There is a well-defined representation $\mu_{\lambda} \circ \Theta_{\lambda\lambda'} \colon \ZZ(\lambda') \to \End({E_\lambda})$, defined as follows.
\begin{enumerate}
\item For every $Z'\in \ZZ(\lambda')$, there exists $U_1$, $V_1$, $U_2$, $V_2\in \ZZ(\lambda)$ such that $$\Theta_{\lambda\lambda'}(Z')=U_1V_1^{-1}=V_2^{-1}U_2\in \widehat{\mathcal Z}^\omega(\lambda)$$ and $\mu_\lambda(V_1)$ and $\mu_\lambda(V_2) $ are invertible in $\End({E_\lambda})$. 
\item For every decomposition $\Theta_{\lambda\lambda'}(Z')=U_1V_1^{-1}=V_2^{-1}U_2$ as above,
$$
\mu_{\lambda} \circ \Theta_{\lambda\lambda'} (Z') = \mu_\lambda(U_1)\mu_\lambda(V_1)^{-1}=\mu_\lambda(V_2)^{-1} \mu_\lambda(U_2) \in \End({E_\lambda}). 
$$
\end{enumerate}
\end{lem}

\begin{proof}
When $Z'$ is a monomial $ (Z_1'Z_2'Z_4')^{k_1} (Z_2'Z_3')^{k_2} (Z_4'Z_5')^{k_3} Z_3^{\prime\, 2k_4} Z_5^{\prime\, 2k_5} B'$, the non-monomial terms occurring in the definition of $\Theta_{\lambda\lambda'}(Z')$ can be written as
\begin{align*}
 \omega^4 Z_1Z_2Z_4 +  Z_1^{-1} Z_2Z_4 &= (1+\omega^4 Z_1^2) Z_1^{-1}Z_2Z_4\\
 Z_2^2 + \omega^4 Z_1^2 Z_2^2 &= (1+\omega^4 Z_1^2) Z_2^2 \\
 Z_4^2 + \omega^4 Z_1^2 Z_4^2 & =  (1+\omega^4 Z_1^2) Z_4^2
\end{align*}
The skew commutativity properties then enable us to write $\Theta_{\lambda\lambda'}(Z')=U_1V_1^{-1}=V_2^{-1}U_2$ for some $U_1$, $V_1$, $U_2$, $V_2\in \ZZ(\lambda)$ where the denominators $V_1$ and $V_2$ are products of terms $1+\omega^{4k}Z_1^2$ with $k\in \Z$.

The same holds for any $Z'\in \ZZ(\lambda')$ by decomposing $Z'$ as a sum of monomials, applying the above argument to each monomial, and reducing to a common denominator. (The reduction to a common denominator is here trivial, because all denominators commute with each other.)

By definition of the representation $\mu_\lambda$ in Proposition~\ref{prop:ConstructRepCheFock}, $\mu_\lambda(Z_1^2)^N = x_1\, \Id_{E_\lambda}$ where $x_1$ is the crossratio weight associated to the edge $e_1$ of $\lambda$ by the enhancement $\xi$. In particular, $\mu_\lambda(Z_1^2)$ is diagonalizable and its eigenvalues are $N$--roots of $x_1$. Also, $x_1$ is different from $-1$ because $\xi$ sends the end points of each lift of the edge $e_1'$ to different points of $\mathbb{CP}^1$. Because $N$ is odd, it follows that the eigenvalues of $\mu_\lambda(Z_1^2)$ are never of the form $-\omega^{-4k}$ with $k\in \Z$, and therefore that $\mu_\lambda(1+\omega^{4k}Z_1^2)$ is invertible for each such $k$. 

This proves that the image of each $Z'\in \ZZ(\lambda)$ under $\Theta_{\lambda\lambda'}$ can be decomposed as
 $$\Theta_{\lambda\lambda'}(Z')=U_1V_1^{-1}=V_2^{-1}U_2\in \widehat{\mathcal Z}^\omega(\lambda)$$
for some  $U_1$, $V_1$, $U_2$, $V_2\in \ZZ(\lambda)$ with $\mu_\lambda(V_1)$ and $\mu_\lambda(V_2) $  invertible in $\End({E_\lambda})$. 

An elementary algebraic manipulation shows that $\mu_\lambda(U_1)\mu_\lambda(V_1)^{-1}$ is equal to $\mu_\lambda(V_2)^{-1} \mu_\lambda(U_2) $ in $ \End({E_\lambda}) 
$, and that this endomorphism is independent of the above  decomposition. We can therefore define a map $\mu_{\lambda} \circ \Theta_{\lambda\lambda'} \colon \ZZ(\lambda') \to \End({E_\lambda})$ by the property that
$$
\mu_{\lambda} \circ \Theta_{\lambda\lambda'} (Z') = \mu_\lambda(U_1)\mu_\lambda(V_1)^{-1}=\mu_\lambda(V_2)^{-1} \mu_\lambda(U_2).
$$
for every such $Z' \in \ZZ(\lambda')$. 

The property that the map  $\mu_{\lambda} \circ \Theta_{\lambda\lambda'} \colon \ZZ(\lambda') \to \End({E_\lambda})$ is an algebra homomorphism easily follows from a couple more easy algebraic manipulations. 
\end{proof}

\begin{lem}
\label{lem:DiagExRespectsCentralElts}
For a vertex $v$ of the triangulations $\lambda$ and $\lambda'$, consider its associated central elements $H_v\in \ZZ(\lambda)$ and $H_v'\in \ZZ(\lambda')$.  Then,  the triangulation change homomorphism $\Theta_{\lambda\lambda'} \colon \widehat{\mathcal Z}^\omega(\lambda') \to \widehat{\mathcal Z}^\omega(\lambda)$ sends  $H_v'\in \ZZ(\lambda')$ to $H_v\in \ZZ(\lambda)$.
\end{lem}

\begin{proof} If we  index the edges of $\lambda'$ meeting $v$ as $e_{i_1'}$, $e_{i_2'}$, \dots, $e_{i_u'}$ in counterclockwise order around $v$ and if we suitably choose the starting point of this indexing, each corner of the square $Q$ that is equal to $v$ contributes a block $Z_2'Z_3'$, $Z_3'Z_1'Z_4'$, $Z_4'Z_5'$ or $Z_5'Z_1'Z_2'$ to the expression 
$$
H_v' = \omega^{-u+2} Z_{i_1}' Z_{i_2}' \dots Z_{i_u}'
$$
provided by Lemma~\ref{lem:QuantumOrderH}. The formula defining $\Theta_{\lambda\lambda'}$ then show that $\Theta_{\lambda\lambda'}(H_v')$ is obtained from this expression by replacing each block $Z_2'Z_3'$, $Z_3'Z_1'Z_4'$, $Z_4'Z_5'$ or $Z_5'Z_1'Z_2'$ by $\omega^{-1}Z_2Z_1Z_3$, $\omega Z_3Z_4$, $\omega^{-1}Z_4Z_1Z_5$ or $\omega Z_5Z_2$, respectively, and by replacing each $Z_i'$ with $i>5$ by $Z_i$. (The only case requiring an argument is that of the blocks $Z_3'Z_1'Z_4' = \omega^{a+6}(Z_1'Z_2'Z_4')(Z_2'Z_3') Z_2^{\prime\, -2}$ and $Z_5'Z_1'Z_2' = \omega^{b-6}(Z_1'Z_2'Z_4')(Z_4'Z_5') Z_4^{\prime\, -2}$, where $a$, $b\in \Z$ depend on the faces of $\lambda$ that are adjacent to several faces of the square $Q$ and are not contained in $Q$, if any.) The result immediately follows from this computation and from the application of Lemma~\ref{lem:QuantumOrderH} to $H_v$. 
\end{proof}

\begin{lem}
\label{lem:DiagExRespectsCheFockReps}
After pre-composing $\mu_{\lambda'}$ with the action on $\ZZ(\lambda')$ of a sign-reversal symmetry of $r\in \RR(S)$ if necessary, the representations $\mu_{\lambda} \circ \Theta_{\lambda\lambda'} \colon \ZZ(\lambda') \to \End({E_\lambda})$ and  $\mu_{\lambda'} \colon \ZZ(\lambda') \to \End(E_{\lambda'})$ are isomorphic. 
\end{lem}

\begin{proof} By Proposition~\ref{prop:ConstructRepCheFock},  the representation $\mu_{\lambda'} \colon \ZZ(\lambda') \to \End(E_{\lambda'})$  is characterized up to isomorphism and sign-reversal symmetry by the following properties.
\begin{enumerate}
\item The dimension of $E_{\lambda'}$ is equal to $N^{3g+p_{\lambda'}-3} $, where $g$ is the genus of the surface  $S$ and where $p_{\lambda'}$ is the number of vertices of the triangulation $\lambda'$;
\item For every edge $e_i'$ of $\lambda'$, let  $x_i' \in \C^*$ be the crossratio weight associated to  $e_i'$ by the enhancement $\xi$, and let $Z_i'$ be the corresponding generator of the Chekhov-Fock algebra $ \TT(\lambda)$. Then,
$$
\mu_{\lambda'} (Z_i^{\prime\, 2N}) = x_i'\, \Id_{E_{\lambda'}}.
$$
\item For every vertex $v$ of $\lambda'$, with associated central element $H_v' \in \ZZ(\lambda')$, 
$$
\mu_{\lambda'}(H_v') = -\omega^4 \,\Id_{E_{\lambda'}}.  
$$
\item
The representation $\rho_{\lambda'} = \mu_{\lambda'} \circ \Tr_{\lambda'}^\omega \colon \SSS(S_\lambda) \to \End({E_\lambda})$ has classical shadow $r \in \RR(S)$, in the sense that 
$$
T_N \bigl( \rho_{\lambda'} ([K]) \bigr) = - \Tr\, r(K)\, \Id_{E_{\lambda'}}
$$
for every knot $K \subset S_{\lambda'}\times [0,1]$ whose projection to $S_{\lambda'}$ has no crossing and whose framing is vertical. 
\end{enumerate}

It therefore suffices to show that the representation $\mu_{\lambda} \circ \Theta_{\lambda\lambda'} \colon \ZZ(\lambda') \to \End({E_\lambda})$ satisfies the same conditions. 

The triangulations $\lambda$ and $\lambda'$ have the same vertex set, so that $p_\lambda=p_{\lambda'}$. The dimension of the space ${E_\lambda}$ is equal to $N^{3g+p_{\lambda}-3}=N^{3g+p_{\lambda'}-3}$ by Proposition~\ref{prop:ConstructRepCheFock} applied to $\mu_\lambda$, which proves the first condition. 

The second condition is checked by several computations. The first elementary computation is that the crossratio weights $x_i$ and $x_i'$ respectively associated to the edges of $\lambda$ and $\lambda'$ by the enhancement $\xi$ are related by the property that $x_i'$ is equal to  $x_1^{-1}$ if $ i=1$, to $(1+x_1) x_2$ if $ i=2$, to $ x_1(1+x_1)^{-1} x_3$ if $ i=3$, to $(1+x_1) x_4$ if $ i=4$, to $x_1(1+x_1)^{-1} x_5$ if $i=5$, and to $x_i $ if $ i>5$.
See for instance \cite[\S2]{Liu} or \cite[\S8]{BonLiu}. 

Then, for instance,
\begin{align*}
\mu_{\lambda} \circ \Theta_{\lambda\lambda'} (Z_3^{\prime\, 2N}) 
&= \mu_{\lambda} \circ \Theta_{\lambda\lambda'} \bigl( \omega^{2N(2N+1)}(Z_2'Z_3')^{2N} Z_2^{\prime\, -2N} \bigr)\\
&= \mu_{\lambda} \bigl( \omega^{2N(2N+1)} (\omega Z_1 Z_2Z_3)^{2N} (Z_2^2+ \omega^4 Z_1^2 Z_2^2)^{-N} \bigr)\\
&= \mu_{\lambda} \bigl( Z_1^{2N} Z_2^{2N} Z_3^{2N} (Z_2^{2N}+ Z_1^{2N} Z_2^{2N})^{-1} \bigr)\\
&= x_1x_2x_3(x_2+x_1x_2)^{-1} \Id_{E_\lambda} = x_3' \Id_{E_\lambda},
\end{align*}
where the third equality uses the relation $Z_2^2(Z_1^2Z_2^2) = \omega^4 (Z_1^2Z_2^2) Z_2^2$, the Quantum Binomial Formula  \cite[\S IV.2]{Kassel} and the fact that $\omega^4$ is a primitive $N$--root of unity. 

Similar computations show that $\mu_{\lambda} \circ \Theta_{\lambda\lambda'} (Z_i^{\prime\, 2N}) =  x_i' \Id_{E_\lambda}$ for every $i$. See also  \cite[\S\S7--8]{BonLiu}. This proves the second condition. 

The third condition is an immediate consequence of Lemma~\ref{lem:DiagExRespectsCentralElts}. 

Finally, the fourth condition is a consequence of the property,  proved in Theorem~28 of \cite{BonWon1}, that $\Theta_{\lambda\lambda'} \circ \Tr_{\lambda'}^\omega = \Tr_{\lambda}^\omega$. 
\end{proof}

Because the triangulations $\lambda$ and $\lambda'$ have the same vertex sets $V_\lambda = V_{\lambda'}$, the associated punctured surfaces $S_\lambda = S - V_\lambda$ and $S_{\lambda'} =S - V_{\lambda'}$ are equal. As a consequence, the homomorphisms $\rho_\lambda = \mu_{\lambda} \circ \Tr_{\lambda}^\omega$ and $\rho_{\lambda'} = \mu_{\lambda'} \circ \Tr_{\lambda'}^\omega$ associated to the enhanced character $(r, \xi)$ provide representations of the same skein algebra $\SSS(S_\lambda) = \SSS(S_{\lambda'})$. 

\begin{cor}
After pre-composing $\mu_{\lambda'}$ with the action on $\ZZ(\lambda')$ of a sign-reversal symmetry of $r\in \RR(S)$ if necessary, the representations $\rho_\lambda = \mu_{\lambda} \circ \Tr_{\lambda}^\omega \colon \SSS(S_\lambda) \to \End({E_\lambda})$ and $\rho_{\lambda'} = \mu_{\lambda'} \circ \Tr_{\lambda'}^\omega \colon \SSS(S_\lambda) \to \End({E_\lambda})$ are isomorphic. 
\end{cor}
\begin{proof}
This is an immediate consequence of Lemma~\ref{lem:DiagExRespectsCheFockReps} and of the fact, proved in Theorem~28 of \cite{BonWon1}, that $\Theta_{\lambda\lambda'} \circ \Tr_{\lambda'}^\omega = \Tr_{\lambda}^\omega$. 
\end{proof}

\begin{lem}
\label{lem:DiagExIsomRespectsOffDiagKernel}
Every isomorphism $\phi \colon {E_\lambda} \to E_{\lambda'}$ between the representations $\mu_{\lambda} \circ \Theta_{\lambda\lambda'} \colon \ZZ(\lambda') \to \End({E_\lambda})$ and  $\mu_{\lambda'} \colon \ZZ(\lambda') \to \End(E_{\lambda'})$,  as in Lemma~{\upshape\ref{lem:DiagExRespectsCheFockReps}}, sends the off-diagonal kernel $F_v \subset {E_\lambda}$ of each  vertex $v\in V_\lambda= V_{\lambda'}$ to the off-diagonal kernel $F_v' \subset E_{\lambda'}$. 
\end{lem}

\begin{proof} 
There are again several cases to consider according to which corners of the square $Q$ correspond to the vertex $v$. We will give the proof in the case when $v$ corresponds to two corners of $Q$, the one where $e_2$ and $e_3$ meet and the corner where $e_5$ and $e_2$ meet. 

As usual, index the edges of $\lambda'$ around the vertex $v$ as $e_{i_1}'$, $e_{i_2}'$, \dots, $e_{i_u}'$, in counterclockwise order around $v$. We can choose the starting point of the indexing at $e_2'$, so that $e_{i_1}'=e_2'$, $e_{i_2}'=e_3'$, $e_{i_{s-1}}' =e_5'$, $e_{i_{s}}'=e_1'$, and $e_{i_{s+1}}'=e_2'$ for some index $s$. To avoid having to worry about whether $s+1 = u$ or not, it is convenient to shift the indexing by 1 and to consider the off-diagonal element 
$$
Q_v' = 1 + \omega^{-4} Z_{i_2}^{\prime\, 2} + \omega^{-8} Z_{i_2}^{\prime\, 2} Z_{i_3}^{\prime\, 2} + \dots +  \omega^{-4u} Z_{i_2}^{\prime\, 2}  Z_{i_3}^{\prime\, 2} \dots   Z_{i_u}^{\prime\, 2} \in \ZZ(\lambda'). 
$$

Then,
\begin{align*}
Z_{2}^{\prime\, 2} Q_v'=
Z_{i_1}^{\prime\, 2} Q_v' = Z_{i_1}^{\prime\, 2} &+ \omega^{-4} Z_{i_1}^{\prime\, 2} Z_{i_2}^{\prime\, 2} + \omega^{-8} Z_{i_1}^{\prime\, 2} Z_{i_2}^{\prime\, 2} Z_{i_3}^{\prime\, 2} + \dots +  \omega^{-4u} Z_{i_1}^{\prime\, 2} Z_{i_2}^{\prime\, 2}  Z_{i_3}^{\prime\, 2} \dots   Z_{i_u}^{\prime\, 2} \\
= Z_{2}^{\prime\, 2} &+ \omega^{-4} Z_{2}^{\prime\, 2} Z_{3}^{\prime\, 2} B' \\
& + \omega^{-4(s-2)} Z_{2}^{\prime\, 2} Z_{3}^{\prime\, 2} Z_{i_3}^{\prime\, 2} \dots Z_{i_{s-2}}^{\prime\, 2} ( Z_{5}^{\prime\, 2} + \omega^{-4} Z_{5}^{\prime\, 2}Z_{1}^{\prime\, 2} +\omega^{-8} Z_{5}^{\prime\, 2} Z_{1}^{\prime\, 2} Z_{2}^{\prime\, 2})\\
&  + \omega^{-4s} Z_{2}^{\prime\, 2} Z_{3}^{\prime\, 2} Z_{i_3}^{\prime\, 2} \dots Z_{i_{s-2}}^{\prime\, 2}  Z_{5}^{\prime\, 2} Z_{1}^{\prime\, 2} Z_{2}^{\prime\, 2} C'
\end{align*}
where $B'$, $C'\in \ZZ(\lambda')$ are polynomials in the variables $Z_{i}^{\prime\, 2}$ with $i>5$, corresponding to edges of $\lambda'$ that are not contained in the square $Q$.

Similarly, we can index the edges of $\lambda$ counterclockwise around $v$ as $e_{j_1}$, $e_{j_2}$, \dots, $e_{j_u}$, in such a way that $e_{j_1}=e_2$, $e_{j_2}=e_1$, $e_{j_3}=e_3$, $e_{i_{s}}=e_5$, and $e_{i_{s+1}}=e_2$. Then,
\begin{align*}
Z_2^2 Q_v=
Z_{j_1}^2 Q_v
= Z_{2}^2 &+ \omega^{-4} Z_{2}^2 Z_{1}^2 + \omega^{-8} Z_{2}^2 Z_{1}^2 Z_{3}^2 B \\
& + \omega^{-4(s-1)} Z_{2}^2 Z_{1}^2 Z_{3}^2 Z_{j_4}^2 \dots Z_{j_{s-1}}^2 ( Z_{5}^2 + \omega^{-4} Z_{5}^2  Z_{2}^2)\\
&  + \omega^{-4s} Z_{2}^2 Z_{1}^2 Z_{3}^2 Z_{j_4}^2 \dots Z_{j_{s-1}}^2 Z_{5}^2  Z_{2}^2 C
\end{align*}
where $B$, $C\in \ZZ(\lambda)$ are respectively obtained from $B'$, $C' \in \ZZ(\lambda')$ by replacing each term $Z_{i}^{\prime\, 2}$ (with $i>5$) with the corresponding $Z_i^2$. 

The computations of Lemma~\ref{lem:CheFockCoordinateChanges} show that 
\begin{align*}
&\Theta_{\lambda\lambda'} (Z_{2}^{\prime\, 2}) = Z_{2}^2 + \omega^{-4} Z_{2}^2 Z_{1}^2
&
&\Theta_{\lambda\lambda'} (Z_{2}^{\prime\, 2}Z_{3}^{\prime\, 2}) = \omega^{-4} Z_{2}^2 Z_1^2 Z_3^2
\\
&\Theta_{\lambda\lambda'} (Z_{5}^{\prime\, 2} + \omega^{-4} Z_{5}^{\prime\, 2}Z_{1}^{\prime\, 2}) = Z_{5}^{\prime\, 2}
&
&\Theta_{\lambda\lambda'} (Z_{5}^{\prime\, 2} Z_{1}^{\prime\, 2} Z_{2}^{\prime\, 2}) = \omega^{4} Z_{5}^2  Z_{2}^2.
\end{align*}
This follows that $\Theta_{\lambda\lambda'} (Z_{2}^{\prime\, 2} Q_v') = Z_2^2 Q_v$.

As a consequence, the isomorphism $\phi$ sends the kernel of $\mu_\lambda(Z_2^2 Q_v) = \mu_\lambda \circ \Theta_{\lambda\lambda'}(Z_{2}^{\prime\, 2} Q_v')$ to the kernel of $\mu_{\lambda'}(Z_{2}^{\prime\, 2} Q_v')$. Since $Z_2^2$ is invertible in $\ZZ(\lambda)$, the kernel of $\mu_\lambda(Z_2^2 Q_v) = \mu_\lambda(Z_2^2) \circ \mu_\lambda (Q_v)$ is equal to the kernel of $ \mu_\lambda (Q_v)$, namely to the off-diabonal kernel $F_v \subset {E_\lambda}$. Similarly, the kernel of $\mu_{\lambda'}(Z_{2}^{\prime\, 2} Q_v')$ is equal to the off-diagonal kernel $F_v' \subset E_{\lambda'}$. 

This concludes the proof in the case when $v$ corresponds to the corners of $Q$ where $e_2$ and $e_3$ meet as well as to the corner where $e_5$ and $e_2$ meet. The other cases are essentially identical to this one. 
\end{proof}

We summarize the discussion and results of this section in the following statement. Let the triangulations $\lambda$ and $\lambda'$ differ from each other by a diagonal exchange as in Figure~\ref{fig:DiagExchange}. Recall that we are assuming that the sides of the square $Q$ where the diagonal exchange takes place are  distinct; however, the triangulations $\lambda$ and $\lambda'$ are not necessarily assumed to be combinatorial. Since $\lambda$ and $\lambda'$ have the same vertex set $V_\lambda = V_{\lambda'}$, the punctured surfaces $S_{\lambda} = S- V_\lambda$ and $S_{\lambda'} = S- V_{\lambda'}$ are equal.

\begin{prop}
\label{prop:DiagExRespectSkeinReps}
Let the triangulations $\lambda$ and $\lambda'$ differ from each other by a diagonal exchange as in Figure~{\upshape\ref{fig:DiagExchange}},  let $\xi \colon \widetilde V_\lambda = \widetilde V_{\lambda'} \to \CP$ be simultaneously  a $\lambda$-- and a $\lambda'$--enhancement for the homomorphism $r \colon \pi_1(S) \to \SL(\C)$, and let $\mu_\lambda \colon \ZZ(\lambda) \to \End({E_\lambda})$ and $\mu_{\lambda'} \colon \ZZ(\lambda')\to \End(E_{\lambda'})$ be the representations associated to this data by Proposition~{\upshape\ref{prop:ConstructRepCheFock}}. 

Suppose in addition that $\rho_{\lambda'} = \mu_{\lambda'} \circ \Tr_{\lambda'}^\omega \colon \SSS(S_{\lambda'}) \to \End({E_\lambda})$ respects the total off-diagonal kernel ${F_{\lambda'}} \subset E_{\lambda'}$ of $\mu_{\lambda'}$ and induces a representation $\check\rho_{\lambda'} \colon \SSS(S) \to \End({F_{\lambda'}})$, as defined above Proposition~{\upshape\ref{prop:SubdivisionInvariance}}. Then, $\rho_{\lambda} = \mu_{\lambda} \circ \Tr_{\lambda}^\omega \colon \SSS(S_{\lambda}) \to \End(E_{\lambda'})$ respects the total off-diagonal kernel ${F_\lambda} \subset {E_\lambda}$ of $\mu_\lambda$ and induces a representation $\check \rho_\lambda \colon \SSS(S) \to \End({F_\lambda})$. Moreover,  $\check \rho_\lambda$ is isomorphic to $\check\rho_{\lambda'} $ after a possible pre-composition with the action of a sign-reversal symmetry of $r\in \RR(S)$ on $\SSS(S)$. 
\end{prop}

\begin{proof}
Lemma~\ref{lem:DiagExRespectsCheFockReps} provides an isomorphism $\phi\colon {E_\lambda} \to E_{\lambda'}$ between the Chekhov-Fock algebra representations  $\mu_{\lambda}\circ \Theta_{\lambda\lambda'} \colon \ZZ(\lambda') \to \End({E_\lambda})$ and  $\mu_{\lambda'} \colon \ZZ(\lambda') \to \End(E_{\lambda'})$. Also, Theorem~28 of \cite{BonWon1} states that  the quantum trace homomorphisms $\Tr_{\lambda}^\omega \colon \SSS(S_\lambda) \to \ZZ(\lambda)$ and $\Tr_{\lambda'}^\omega \colon \SSS(S_\lambda) \to \ZZ(\lambda')$ are compatible with the isomorphism $\Theta_{\lambda\lambda'} \colon \widehat{\mathcal Z}^\omega(\lambda') \to \widehat{\mathcal Z}^\omega (\lambda)$ in the sense that $\Theta_{\lambda\lambda'} \circ \Tr_{\lambda'}^\omega = \Tr_{\lambda}^\omega$. Therefore, $\phi\colon {E_\lambda} \to E_{\lambda'}$ provides an isomorphism between the representations $ \mu_{\lambda}\circ \Theta_{\lambda\lambda'}  \circ \Tr_{\lambda'}^\omega =   \mu_{\lambda}\circ  \Tr_{\lambda}^\omega = \rho_\lambda$ and $ \mu_{\lambda'} \circ \Tr_{\lambda'}^\omega = \rho_{\lambda'}$.

By Lemma~\ref{lem:DiagExIsomRespectsOffDiagKernel}, the isomorphism $\phi\colon {E_\lambda} \to E_{\lambda'}$ sends the  total off-diagonal kernel ${F_\lambda} = \bigcap_{v\in V_\lambda} F_v$ to the total off-diagonal kernel ${F_{\lambda'}} = \bigcap_{v\in V_\lambda} F_v'$. Since the representation $\rho_{\lambda'}$ respects ${F_{\lambda'}}$ by hypothesis, it follows that $\rho_\lambda$ respects ${F_\lambda}$. 

Finally, the property that $\rho_{\lambda'}$ induces a representation $\bar\rho_{\lambda'}\colon \SSS(S) \to \End({F_{\lambda'}})$ means that $\rho_{\lambda'}\bigl( [K] \bigr)_{|{F_\lambda}} = \rho_{\lambda'}\bigl( [K'] \bigr)_{|{F_\lambda}}$ whenever the two framed links $K$, $K' \subset S_\lambda \times [0,1]$ are isotopic in $S\times [0,1]$. The isomorphism $\phi$ again shows that the same property holds for $\rho_\lambda$. 
\end{proof}

\subsection{Constructing representations of the skein algebra of a closed surface using arbitrary triangulations}
\label{subsect:ConstructSkeinRepArbitraryTriangulation}

\begin{lem}
\label{lem:MakeTriangCombin}
Let $\lambda$ be a triangulation of $S$. Then one can apply to $\lambda$ a sequence of face subdivisions and diagonal exchanges, as in  \S\S {\upshape\ref{subsect:FaceSubdivisions}} and {\upshape\ref{subsect:DiagonalExchanges}}, to obtain a new triangulation $\lambda'$ that is combinatorial, in the sense that each edge of $\lambda'$ has distinct endpoints and no two edges have the same endpoints. 

In addition, any $\lambda$--enhancement $\xi \colon \widetilde V_\lambda \to \mathbb{CP}^1$ for the group homomorphism $r\colon \pi_1(S) \to \SL(\C)$ can be extended to a  $\lambda'$--enhancement $\xi' \colon \widetilde V_{\lambda'} \to \mathbb{CP}^1$. (Note that the vertex set $V_\lambda$ of $\lambda$ is contained in the vertex set $V_{\lambda'}$). 
\end{lem}

\begin{proof} By subdividing a few  faces if necessary, we can arrange that any two faces of $\lambda$ have at most one edge in common. 

After this preliminary step, let $\lambda''$ be obtained from $\lambda$ by subdividing each face, and let $\lambda'$ be obtained by performing a diagonal exchange along each edge of $\lambda''$ that is also an edge of  $\lambda$; see Figure~\ref{fig:MakeCombinatorial}. All edges of a face of $\lambda$ are distinct by our general convention for triangulations,  and we had arranged at the beginning fo the proof that any two faces have at most one edge in common. It easily follows that each edge of the resulting triangulation $\lambda'$ joins, either a vertex of $V_\lambda$ to a vertex of  $V_{\lambda'} - V_\lambda$, or two distinct vertices of $V_{\lambda'} - V_\lambda$. As a consequence,  $\lambda'$ is combinatorial.

\begin{figure}[htbp]

\includegraphics{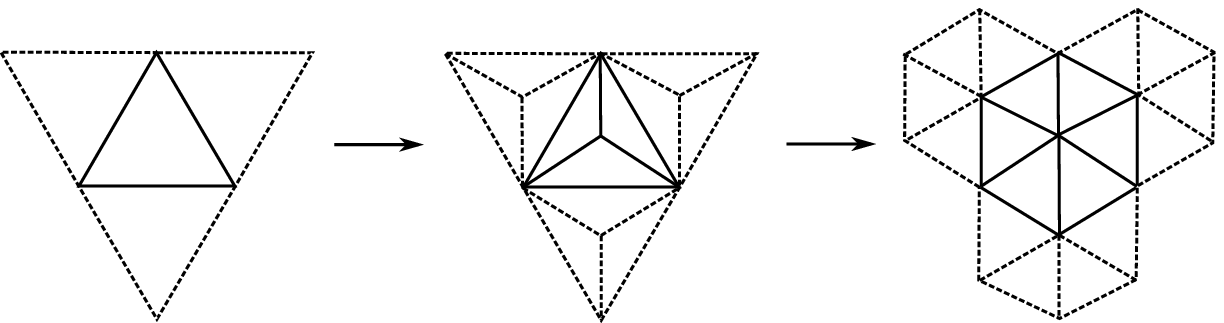}
\caption{}
\label{fig:MakeCombinatorial}
\end{figure}

Using the properties that $\lambda'$ is combinatorial and that each edge of $\lambda'$ touches at most one edge of $V_\lambda$, the inductive process of the proof of Lemma~\ref{lem:EnhancementsExist} then proves the second statement. 
\end{proof}

Lemma~\ref{lem:MakeTriangCombin} and the results of \S\S \ref{subsect:FaceSubdivisions} and \ref{subsect:DiagonalExchanges}  enable us to extend Proposition~\ref{prop:OffDiagKernelInvariant} and \ref{prop:VertexSweep} to triangulations that are not necessarily combinatorial. 

\begin{thm}
\label{thm:ConstructRepSkeinAlgebraClosedSurface}
Given a triangulation $\lambda$ of the surface $S$ and a $\lambda$--enhancement $\xi \colon \widetilde V_\lambda \to \mathbb{CP}^1$ for the group homomorphism $r\colon \pi_1(S) \to \SL(\C)$,  let $\mu_\lambda \colon \ZZ(\lambda) \to \End({E_\lambda})$ be the irreducible representation associated to this data by Proposition~{\upshape \ref{prop:ConstructRepCheFock}}. Then, the total off-diagonal kernel ${F_\lambda} \subset {E_\lambda}$ of  $\mu_\lambda$ is invariant under the representation  $\rho_\lambda = \mu_\lambda \circ \Tr_\lambda^\omega \colon \SSS(S_\lambda) \to \End({E_\lambda}) $ constructed in {\upshape \S \ref{subsect:RepresentingPunctSurfSkein}}, and $\rho_\lambda$ induces a representation $\check \rho_ \lambda \colon \SSS(S) \to \End({F_\lambda})$. 
\end{thm}

\begin{proof}
By Lemma~\ref{lem:MakeTriangCombin}, there exists a sequence of triangulations $\lambda=\lambda_0$, $\lambda_1$, \dots, $\lambda_{n-1}$, $\lambda_n = \lambda'$ such that $\lambda'$ is combinatorial, and such that each $\lambda_{i+1}$ is obtained from $\lambda_i$ by a face subdivision or by a diagonal exchange. In addition, for every $i$, the $\lambda$--enhancement $\xi \colon \widetilde V_\lambda \to \mathbb{CP}^1$ can be extended to a $\lambda_i$--enhancement $\xi_i \colon \widetilde V_{\lambda_i} \to \mathbb{CP}^1$ for $r$, in such a way that each $\xi_{i+1}$ restricts to $\xi_i$ on $\widetilde V_{\lambda_i}$. 

Since $\lambda'$ is combinatorial, the property sought holds for $\lambda'$ by Propositions~\ref{prop:OffDiagKernelInvariant} and \ref{prop:VertexSweep}. Propositions~\ref{prop:SubdivisionInvariance} and \ref{prop:DiagExRespectSkeinReps} assert that  the property will also hold for $\lambda_i$ if it holds for $\lambda_{i+1}$. The result then follows by induction. 
\end{proof}

\begin{thm}
\label{thm:ClassicalShadowCorrect}
The  representation $\check \rho_ \lambda \colon \SSS(S) \to \End({F_\lambda})$ provided by Theorem~{\upshape\ref{thm:ConstructRepSkeinAlgebraClosedSurface}} has classical shadow equal to the character $ r \in \RR(S)$ represented by the group homomorphism $r\colon \pi_1(S) \to \SL(\C)$, in the sense that 
$$
T_N\bigl( \check \rho_ \lambda ([K] )\bigr) = - \Tr \,r(K)\, \Id_{F_\lambda}
$$
for every knot $K \subset S\times [0,1]$ whose projection to $S$ has no crossing and whose framing is vertical. 
\end{thm}
\begin{proof} We again use a sequence of triangulations  $\lambda=\lambda_0$, $\lambda_1$, \dots, $\lambda_{n-1}$, $\lambda_n = \lambda'$ and $\lambda_i$--enhancements $\xi_i \colon \widetilde V_{\lambda_i} \to \mathbb{CP}^1$ such that $\lambda'$ is combinatorial,  each $\lambda_{i+1}$ is obtained from $\lambda_i$ by a face subdivision or by a diagonal exchange, and each $\xi_{i+1}$ restricts to $\xi_i$ on $\widetilde V_{\lambda_i}$. 

Since $\lambda'$ is combinatorial, Proposition~\ref{prop:ClassicalShadowOKCombinatorial} shows that $\check \rho_{\lambda'} = \check \rho_{\lambda_n}$ has classical shadow equal to $r\in \RR(S)$.  Propositions~\ref{prop:SubdivisionInvariance} and \ref{prop:DiagExRespectSkeinReps}  then inductively show that the  $\check \rho_{\lambda_i}$ are all isomorphic, and consequently also have classical shadow $r\in \RR(S)$. In particular, $\check \rho_{\lambda}=\check \rho_{\lambda_0}$ has classical shadow $r\in \RR(S)$. 
\end{proof}

\subsection{Independence of choices}
\label{subsect:IndependenceChoices}

We now prove that the construction of the representation  $\check \rho_ \lambda \colon \SSS(S) \to \End({F_\lambda})$ of Theorem~\ref{thm:ConstructRepSkeinAlgebraClosedSurface} is very natural. 

\begin{lem}
\label{lem:ConnectTriangulations}
Let $\lambda$ and $\lambda'$ be two  triangulations of $S$ whose vertex sets are disjoint, and let $\xi:\widetilde V_\lambda \to \CP$ and $\xi':\widetilde V_{\lambda'} \to \CP$ be $\lambda$-- and $\lambda'$--enhancements, respectively, for the homomorphism $r \colon \pi_1(S) \to \SL(\C)$. Then $\lambda$ and $\lambda'$ can be connected by a sequence of triangulations $\lambda= \lambda_0$, $\lambda_1$, \dots, $\lambda_{n-1}$, $\lambda_n=\lambda'$, each equipped with a $\lambda_i$--enhancement $\xi_i:V_{\lambda_i} \to \CP$ for $r$,  such that:
\begin{enumerate}
\item each $\lambda_{i+1}$ is obtained from $\lambda_i$ by a face subdivision as introduced in   {\S \upshape \ref{subsect:FaceSubdivisions}}, the inverse of a face subdivision, or a diagonal exchange as in  {\S \upshape \ref{subsect:DiagonalExchanges}};
\item $\xi_0=\xi$ and $\xi_n = \xi'$;
\item for every $i$, $\xi_i$ and $\xi_{i+1}$ coincide on the intersection $\widetilde V_{\lambda_i} \cap \widetilde V_{\lambda_{i+1}}$. 
\end{enumerate}
\end{lem}
For the third condition, note that the vertex sets $V_{\lambda_i} $ and $ V_{\lambda_{i+1}}$ differ by at most one vertex, so that $\widetilde V_{\lambda_i} $ and $\widetilde V_{\lambda_{i+1}}$ differ by at most one $\pi_1(S)$--orbit. 
\begin{proof}
By Lemma~\ref{lem:MakeTriangCombin}, we can assume without loss of generality that $\lambda$ and $\lambda'$ are combinatorial. The existence of the sequence $\lambda= \lambda_0$, $\lambda_1$, \dots, $\lambda_{n-1}$, $\lambda_n=\lambda'$ in this combinatorial setup is then  the 2--dimensional case of Pachner's theorem \cite{Pach1, Pach2} (which of course predates the full generality of Pachner's theorem by many decades). In addition, the $\lambda_i$ provided by this statement are all combinatorial.

To construct the enhancements $\xi_i$, note that $\pi_1(S)$ acts on each $\widetilde V_{\lambda_i} $, and therefore on the union $\bigcup_{i=1}^n \widetilde V_{\lambda_i} $. Extend $\xi$ and $\xi'$ to an $r$--equivariant map $\xi'' \colon \bigcup_{i=1}^n \widetilde V_{\lambda_i} \to \CP$ (this is where we use the fact that the vertex sets $V_\lambda$ and $V_{\lambda'}$ are disjoint), orbit by orbit as in the proof of Lemma~\ref{lem:EnhancementsExist}.  In this construction, we can require that distinct $\pi_1(S)$--orbits in $\bigcup_{i=1}^n \widetilde V_{\lambda_i} $ have disjoint images in $\CP$, since we only need to avoid countably many values at each step. Then, the restriction $\xi_i$ of $\xi''$ to $ \widetilde V_{\lambda_i} $ is a $\lambda_i$--enhancement for $r$; indeed, because $\lambda_i$ is combinatorial, the endpoints of each edge of its lift $\widetilde \lambda_i$ to the universal cover $\widetilde S$ belong to distinct $\pi_1(S)$--orbits, and in particular have distinct images under $\xi''$. 
\end{proof}

\begin{thm}
\label{thm:RepIndependentChoices}
Up to isomorphism and up to the action of a sign-reversal symmetry of $r\in \RR(S)$ on $\SSS(S)$, the representation $\check \rho_ \lambda \colon \SSS(S) \to \End({F_\lambda})$ provided by Theorem~{\upshape\ref{thm:ConstructRepSkeinAlgebraClosedSurface}} depends only on the group homomorphism $r\colon \pi_1(S) \to \SL(\C)$, not on the triangulation $\lambda$ or the $\lambda$--enhancement $\xi$ used in the construction. 
\end{thm}

\begin{proof}
Consider two triangulations $\lambda$ and $\lambda'$, with respective enhancements $\xi:\widetilde V_\lambda \to \CP$ and $\xi':\widetilde V_{\lambda'} \to \CP$ for the homomorphism $r \colon \pi_1(S) \to \SL(\C)$.  Modifying $\lambda'$ by a small isotopy does not change the associated representations $\mu_{\lambda'} \colon \ZZ(\lambda') \to \End(E_{\lambda'})$ and $\check \rho_ {\lambda'} \colon \SSS(S) \to \End({F_{\lambda'}})$, so we can assume that the vertex sets $V_\lambda$ and $V_{\lambda'}$ are disjoint. We can then consider the sequences of triangulations $\lambda= \lambda_0$, $\lambda_1$, \dots, $\lambda_{n-1}$, $\lambda_n=\lambda'$ and $\lambda_i$--enhancement $\xi_i:V_{\lambda_i} \to \CP$ provided by Lemma~\ref{lem:ConnectTriangulations}. Theorem~\ref{thm:ConstructRepSkeinAlgebraClosedSurface} associates to each triangulation $\lambda_i$ and $\lambda_i$--enhancement $\xi_i$  a representation $\check \rho_ {\lambda_i} \colon \SSS(S) \to \End(F_i)$. 
Propositions~\ref{prop:SubdivisionInvariance} and \ref{prop:DiagExRespectSkeinReps} show that each $\check \rho_ {\lambda_i} \colon \SSS(S) \to \End(F_i)$ is isomorphic to $\check \rho_ {\lambda_{i+1}} \colon \SSS(S) \to \End(F_{i+1})$ after possible composition with a sign-reversal symmetry.  It follows that $\check \rho_ \lambda \colon \SSS(S) \to \End({F_\lambda})$ is isomorphic to $\check \rho_{ \lambda'} \colon \SSS(S) \to \End({F_{\lambda'}})$ after possible composition with a sign-reversal symmetry. 
\end{proof}

\begin{rem}
Conjugating the homomorphism $r\colon \pi_1(S) \to \SL(\C)$ by an element $\theta\in\SL(\C)$ also leaves the isomorphism class of the representation $\check \rho_ \lambda \colon \SSS(S) \to \End(F_{\lambda})$ invariant; indeed, the $\lambda$--enhancement $\theta \xi \colon \widetilde V_\lambda \to \CP$ for $\theta r \theta^{-1}$ induces the same edge weights $x_i$ as $\xi$ in the construction of $\mu_\lambda \colon \ZZ(\lambda) \to \End({E_\lambda})$ in Proposition~{\upshape \ref{prop:ConstructRepCheFock}}. For irreducible homomorphisms $r \colon \pi_1(S) \to \SL(\C)$, being conjugate  by an element of $\SL(\C)$ is equivalent to defining the same character $ r \in \RR(S)$. However, for reducible homomorphisms $r $, we do not know if the representation $\check \rho_ \lambda \colon \SSS(S) \to \End(F_\lambda)$ depends only on the induced character $ r \in \RR(S)$  or on  subtler properties of the conjugacy class of $r \colon \pi_1(S) \to \SL(\C)$. 
\end{rem}

\section{The dimension of the total off-diagonal kernel}
\label{bigsect:DimensionOffDiagonalKernel}

We now have associated to each group homomorphism $r\colon \pi_1(S) \to \SL(\C)$ as representation $\check \rho_ \lambda \colon \SSS(S) \to \End(F_\lambda)$ of the skein algebra $\SSS(S)$, with classical shadow equal to the character $r\in \RR(S)$ represented by $r$. This construction is very natural as, up to isomorphism, $\check\rho_\lambda$ is independent of the triangulation $\lambda$ and of the $\lambda$--enhancement $\xi$.

However, we still do not know that this representation is non-trivial, namely that the total off-diagonal kernel $F_\lambda$ is non-trivial. This section is devoted to proving the non-triviality of $F_\lambda$, and to estimate its dimension. 

\begin{thm}
\label{thm:DimOffDiagKernel}
Let $\check \rho_ \lambda \colon \SSS(S) \to \End({F_\lambda})$ of the Kauffman bracket skein algebra of the closed oriented surface $S$ associated to the group homomorphism $r\colon \pi_1(S) \to \SL(\C)$ by Theorem~{\upshape\ref{thm:ConstructRepSkeinAlgebraClosedSurface}}. Then, the dimension of the off-diagonal kernel $F_{\lambda}$ is such that
$$
\dim F_{\lambda} \geq
\begin{cases}
N^{3(g-1)} &\text{ if } g\geq 2\\
N &\text{ if } g=1\\
1 &\text{ if } g=0
\end{cases}
$$
where $g$ is the genus of  $S$. In addition, the above inequality is an equality when the character $r\in \RR(S)$ represented by  $r$ is sufficiently generic, in the sense that it belongs to an explicit Zariski open dense subset of $\RR(S)$.  
\end{thm}

When the surface $S$ is not the sphere, the proof of Theorem~\ref{thm:DimOffDiagKernel} is based on explicit computations for triangulations $\lambda$ that have only one vertex. In particular, these triangulations cannot  be combinatorial.  This proof is the only reason why we struggled to include non-combinatorial triangulations in the previous sections. 

\subsection{Proof of Theorem~\ref{thm:DimOffDiagKernel} when the surface  $S$ has genus $g\geq 2$}

Let $\lambda$ be a triangulation of the surface $S$ with only one vertex $v$. In particular, every edge of $\lambda$ is a loop. Because $S$ has genus $g\geq 2$, we can choose $\lambda$ so that, in addition, there is an edge $e_{i_0}$ of $\lambda$ that  separates the surface $S$ into two subsurfaces  $S_1$ and $S_2$. Because of our conventions for triangulations, the three sides of each face of $\lambda$ are distinct, and an Euler characteristic argument shows that each of $S_1$ and $S_2$ has positive genus.

We first consider the case where the  homomorphism $r\colon \pi_1(S) \to \SL(\C)$ admits a $\lambda$--enhancement $\xi$. By Lemma~\ref{lem:EnhancementsExist}, this is equivalent to the property that $r(e_i)\neq \pm \Id$ for every edge $e_i$ of the triangulation $\lambda$. Let $\mu_\lambda \colon \ZZ(\lambda) \to \End({E_\lambda})$ be the representation associated   to the enhanced homomorphism $(r, \xi)$ by Proposition~\ref{prop:ConstructRepCheFock}, and consider the representation $\rho_\lambda = \mu_\lambda \circ \Tr_\lambda^\omega \colon \SSS(S_\lambda) \to \End({E_\lambda})$ as in \S \ref{subsect:RepresentingPunctSurfSkein}. Note that $S_\lambda$ is here the punctured surface $S-\{v\}$, obtained by removing from $S$ the vertex $v$ of $\lambda$. In particular, the total off-diagonal kernel $F_\lambda \subset E_\lambda$ of $\mu_\lambda$ is equal to the off-diagonal kernel $F_v$ of $v$. 

Let $K_1 \subset S_1$ be the closed curve obtained by pushing the edge loop $e_{i_0}$ inside of the subsurface $S_1$, and let $K_2 \subset S_2$ be similarly defined. In particular, $K_1$ and $K_2$ are both contained in the punctured surface $S_\lambda = S- \{ v \}$. When endowed with the vertical framing, $K_1$ and $K_2$ define skeins $[K_1]$, $[K_2] \in \SSS(S_\lambda)$. 

\begin{lem}
\label{lem:OffDiagKernelAndSweep}
The off-diagonal kernel $F_{\lambda}=\ker \mu_\lambda(Q_v)$ is equal to the kernel of $\rho_\lambda   \bigl( [K_1] \bigr) - \rho_\lambda \bigl( [K_2] \bigr) $.  
\end{lem}

\begin{proof}
Since $K_1$ and $K_2$ are isotopic in $S\times [0,1]$, Theorem~\ref{thm:ConstructRepSkeinAlgebraClosedSurface}  shows that the restrictions $\rho_\lambda\bigl( [K_1] \bigr)_{|F_{\lambda}} = \rho_\lambda\bigl( [K_2] \bigr)_{|F_{\lambda}} $ coincide. The off-diagonal kernel $F_{\lambda}$ is therefore contained in the kernel of $\rho_\lambda   \bigl( [K_1] \bigr) - \rho_\lambda \bigl( [K_2] \bigr) $.  

Because $\lambda$ is not combinatorial, our proof of Theorem~\ref{thm:ConstructRepSkeinAlgebraClosedSurface} relied on the ``drill, baby, drill'' strategy to reduce the problem to a combinatorial triangulation, where we could apply Proposition~\ref{prop:VertexSweep}. We will here use a careful examination of an analogue of  Proposition~\ref{prop:VertexSweep} for the non-combinatorial triangulation $\lambda$.

\begin{figure}[htbp]
\vskip 10pt
\SetLabels
( .6* 1.05) $ K_1$ \\
( .6* .7) $K_2 $ \\
(.6* -.2) (a)\\
\small
(0.65 * .32) $e_{i_0} $ \\
(0.1 *.83) $e_{i_1} $ \\
( 0*.75 ) $e_{i_2}$ \\
( .12* .15) $ e_{i_t}$ \\
( .35* .31) $e_{j_1}$ \\
(.4 * .42) $e_{j_2}$ \\
( .38*.66 ) $ e_{j_u}$ \\
\endSetLabels
\centerline{
\AffixLabels{\includegraphics{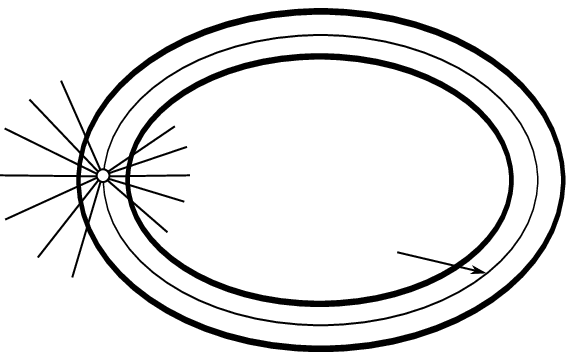}}\hskip 2cm
\SetLabels
( .6* 1.05) $ K_1'$ \\
(.6* -.2) (b)\\
\small
(0.65 * .15) $e_{i_0} $ \\
(0.1 *.83) $e_{i_1} $ \\
( 0*.75 ) $e_{i_2}$ \\
( .12* .15) $ e_{i_t}$ \\
(.98 * .7) $p$\\
(.98 * .27) $q$ \\
(.96 * .48) $B$ \\
\endSetLabels\AffixLabels{\includegraphics{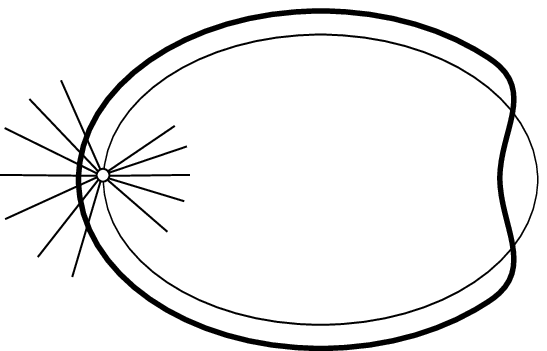}}
}
\vskip 10pt
\caption{}
\label{fig:OneHole}
\end{figure}

We first need to compute $ \Tr_\lambda^\omega \bigl( [K_1] \bigr)$ and $ \Tr_\lambda^\omega \bigl( [K_2] \bigr)$. For this, index the edges around the vertex $v$ as $e_{i_0}$, $e_{i_1}$, $e_{i_2}$, \dots $e_{i_t}$,  $e_{i_0}$, $e_{j_1}$, $e_{j_2}$, \dots $e_{j_{u}}$, counterclockwise in this order, so that all edges $e_{i_k}$ are contained in the subsurface  $S_1$ and all edges $e_{j_k}$ are contained in $S_2$. See Figure~\ref{fig:OneHole}(a). 

The computation of $ \Tr_\lambda^\omega \bigl( [K_1] \bigr)$ given by \cite{BonWon1} can be somewhat  complicated, because the projection of $K_1$ to $S$ cuts some edges of $\lambda$ more than once; this usually introduces correction factors in bigon neighborhoods of these edges. A convenient way to avoid these correction factors is to  isotop $K_1$ to a  framed knot $K_1'\subset S_\lambda\times[0,1]$ whose projection to $S$ coincides with the projection of $K_1$ for most of its length, except for a small interval that is pushed across the edge $e_{i_0}$ to create a small bigon $B\subset S_2$ bounded by an arc in $e_{i_0}$ and an arc in the projection of $K_1'$. In particular, the projection of $K_1'$ to $S$ cuts the edge $e_{i_0}$ in two points $p$ and $q$ occuring in this order for the orientation of $e_{i_0}$ coming from the boundary orientation of $K_1$. See Figure~\ref{fig:OneHole}(b).

In addition, we can arrange that the elevation on $K_1 \subset S_\lambda \times [0,1]$ steadily increases as one goes around $K_1$ from $p$ to $q$, crossing the preimage of the edges  $e_{i_1}$, $e_{i_2}$, \dots $e_{i_t}$, and then steeply goes down from $q$ to $p$ along the bigon $B$ to return to the starting point $p$. 

Then, there is no need for correction terms, except for the contribution of the bigon $B$. More precisely, the construction of the quantum trace in \cite{BonWon1} gives in this case:
\begin{align*}
 \Tr_\lambda^\omega \bigl( [K_1] \bigr) = \Tr_\lambda^\omega \bigl( [K_1'] \bigr) &=
 \omega^{-1}
 \sum_{k=0}^t \omega^{-t+1} Z_{i_0}Z_{i_1}\dots Z_{i_k} Z_{i_{k+1}}^{-1}Z_{i_{k+2}}^{-1}\dots Z_{i_t}^{-1}Z_{i_0}^{-1} \\
 &=
 \omega^{-t}
 \sum_{k=0}^t Z_{i_0}Z_{i_1}\dots Z_{i_k} Z_{i_{k+1}}^{-1}Z_{i_{k+2}}^{-1}\dots Z_{i_t}^{-1}Z_{i_0}^{-1} 
\end{align*}
where the factor $ \omega^{-1}$ is the contribution of the bigon $B$.

We will use the computation of Weyl quantum orderings in Lemma~\ref{lem:QuantumOrderAtVertex} to rearrange this expression. By the first case of Lemma~\ref{lem:QuantumOrderAtVertex},
\begin{align*}
&[Z_{i_0}Z_{i_1}\dots Z_{i_k}] = \omega^{-k} Z_{i_0}Z_{i_1}\dots Z_{i_k},\\
&[Z_{i_0}Z_{i_1}\dots Z_{i_k}]^{-1} = [Z_{i_0}^{-1}Z_{i_1}^{-1}\dots Z_{i_k}^{-1}]= \omega^{-k} Z_{i_0}^{-1}Z_{i_1}^{-1}\dots Z_{i_k}^{-1},\\
\text{and }&[Z_{i_0}Z_{i_1}\dots Z_{i_k}]^{2}  =[Z_{i_0}^{2}Z_{i_1}^{2}\dots Z_{i_k}^{2}]= \omega^{-4k} Z_{i_0}^{2}Z_{i_1}^{2}\dots Z_{i_k}^{2}
\end{align*}
It follows that
\begin{align*}
Z_{i_0}Z_{i_1}\dots Z_{i_k} &= \omega^k [Z_{i_0}Z_{i_1}\dots Z_{i_k}] \\
&=  \omega^k [Z_{i_0}Z_{i_1}\dots Z_{i_k}]^{2} [Z_{i_0}Z_{i_1}\dots Z_{i_k}]^{-1} \\
&= \omega^{-4k} Z_{i_0}^{2}Z_{i_1}^{2}\dots Z_{i_k}^{2}Z_{i_0}^{-1}Z_{i_1}^{-1}\dots Z_{i_k}^{-1}.
\end{align*}
This enables us to write 
\begin{align*}
 \Tr_\lambda^\omega \bigl( [K_1] \bigr) &= \omega^{-t}
\biggl( \sum_{k=0}^t \omega^{-4k} Z_{i_0}^2Z_{i_1}^2\dots Z_{i_k}^2\biggr) \Bigl( Z_{i_0}^{-1}Z_{i_1}^{-1}\dots Z_{i_t}^{-1} Z_{i_0}^{-1}\Bigr)\\
&= \biggl( \sum_{k=0}^t \omega^{-4k} Z_{i_0}^2Z_{i_1}^2\dots Z_{i_k}^2\biggr) Z_{i_0}^{-1}[Z_{i_1}^{-1}\dots Z_{i_t}^{-1}] Z_{i_0}^{-1},
\end{align*}
where the second equality follows from  an application of the third case of Lemma~\ref{lem:QuantumOrderAtVertex} to $[Z_{i_1}^{-1}\dots Z_{i_t}^{-1}]  = \omega^{-t} Z_{i_1}^{-1}\dots Z_{i_t}^{-1}$.

Consider the term  $[Z_{i_1} Z_{i_2} \dots Z_{i_t}]$. First of all, note that its exponents satisfy the parity condition required to belong to the balanced Chekhov-Fock algebra $\ZZ(\lambda)$. Also, $[Z_{i_1} Z_{i_2} \dots Z_{i_t}]$ commutes with $Z_{i_0}$; indeed,  the only $Z_{i_k}$ that do not commute with $Z_{i_0}$ are $Z_{i_1}$, $Z_{i_t}$, and a pair of consecutive elements $Z_{i_{k_1}}=Z_{i_1}$ and $Z_{i_{k_1+1}}=Z_{i_t}$ corresponding to the third vertex of the face of $\lambda$ that is contained in the subsurface $S_1$ and is adjacent to the edge $e_{i_0}$. It also commutes with all generators $Z_{j_l}$ since the corresponding edges are located in the interior of the surface $S_2$. 

Similarly $[Z_{j_1} Z_{j_2} \dots Z_{j_u}]$ is an element of  $\ZZ(\lambda)$ that commutes with $Z_{i_0}$ and with all $Z_{i_k}$. 

In particular, $[Z_{i_1} Z_{i_2} \dots Z_{i_t}]$, $[Z_{j_1} Z_{j_2} \dots Z_{j_u}]$ and $Z_{i_0}$ commute with each other, and the central element $H_v$ associated to the vertex $v$ is equal to
$$
H_v = [ Z_{i_0} Z_{i_1}  \dots Z_{i_t} Z_{i_0}  Z_{j_1}  \dots Z_{j_u}] = Z_{i_0}^2 [Z_{i_1} Z_{i_2} \dots Z_{i_t}][Z_{j_1} Z_{j_2} \dots Z_{j_u}].
$$

Returning to $\Tr_\lambda^\omega \bigl( [K_1] \bigr) $ and remembering that $H_v$ and $[Z_{j_1} Z_{j_2} \dots Z_{j_u}]$ commute with all $Z_{i_k}$, we conclude that
\begin{align*}
 \Tr_\lambda^\omega \bigl( [K_1] \bigr) 
&= \biggl( \sum_{k=0}^t \omega^{-4k} Z_{i_0}^2Z_{i_1}^2 Z_{i_2}^2 \dots Z_{i_k}^2\biggr) Z_{i_0}^{-2}[Z_{i_1}Z_{i_2}\dots Z_{i_t}]^{-1} \\
&=  \biggl( \sum_{k=0}^t \omega^{-4k} Z_{i_0}^2Z_{i_1}^2 Z_{i_2}^2\dots Z_{i_k}^2\biggr) H_v^{-1} [Z_{j_1} Z_{j_2} \dots Z_{j_u}] \\
&=  H_v^{-1} [Z_{j_1} Z_{j_2} \dots Z_{j_u}]  \biggl( \sum_{k=0}^t \omega^{-4k} Z_{i_0}^2Z_{i_1}^2  Z_{i_2}^2 \dots Z_{i_k}^2\biggr)\\
&=  [Z_{i_1} Z_{i_2} \dots Z_{i_t}]^{-1} Z_{i_0}^{-2}  \biggl( \sum_{k=0}^t \omega^{-4k} Z_{i_0}^2Z_{i_1}^2  Z_{i_2}^2 \dots Z_{i_k}^2\biggr).
\end{align*}

The same arguments applied to the framed knot $K_2$ give
$$
 \Tr_\lambda^\omega \bigl( [K_2] \bigr) 
 =  [Z_{j_1} Z_{j_2} \dots Z_{j_u}]^{-1} Z_{i_0}^{-2}   \biggl( \sum_{l=0}^u \omega^{-4l} Z_{i_0}^2Z_{j_1}^2Z_{j_2}^2\dots Z_{j_l}^2\biggr) .
$$

In particular,
$$
 [Z_{i_1} Z_{i_2} \dots Z_{i_t}]  \Tr_\lambda^\omega \bigl( [K_1] \bigr)  
=  1+  \sum_{k=1}^t \omega^{-4k} Z_{i_1}^2Z_{i_2}^2 \dots Z_{i_k}^2
$$
and 
\begin{align*}
 [Z_{i_1} Z_{i_2}& \dots Z_{i_t}]    \Tr_\lambda^\omega \bigl( [K_2] \bigr)  \\
&= Z_{i_0}^{-2} [Z_{j_1} Z_{j_2} \dots Z_{j_u}] ^{-1} [Z_{i_1} Z_{i_2} \dots Z_{i_t}]  \biggl( \sum_{l=0}^u \omega^{-4l} Z_{i_0}^2Z_{j_1}^2Z_{j_2}^2\dots Z_{j_l}^2\biggr) \\
&= H_v^{-1} [Z_{i_1} Z_{i_2} \dots Z_{i_t}]^2  \biggl( \sum_{l=0}^u \omega^{-4l} Z_{i_0}^2Z_{j_1}^2Z_{j_2}^2\dots Z_{j_l}^2\biggr) \\
&= H_v^{-1}   \biggl( \sum_{l=0}^u \omega^{-4t - 4l} Z_{i_1}^2 Z_{i_2}^2 \dots Z_{i_t}^2 Z_{i_0}^2Z_{j_1}^2Z_{j_2}^2\dots Z_{j_l}^2\biggr) \\
\end{align*}
by using again the third case of Lemma~\ref{lem:QuantumOrderAtVertex}, in addition to  the fact that $[Z_{i_1} Z_{i_2} \dots Z_{i_t}]$, $[Z_{j_1} Z_{j_2} \dots Z_{j_u}]$ and $Z_{i_0}$ commute with each other. 

This is beginning to look a lot like the off-diagonal term $Q_v \in \ZZ(\lambda)$ associated to the counterclockwise indexing of the edges of $\lambda$ around $v$ as $e_{i_1}$, $e_{i_2}$, \dots, $e_{i_t}$, $e_{i_0}$, $e_{j_1}$, $e_{j_2}$, \dots, $e_{j_u}$, $e_{i_0}$. Indeed, this off-diagonal term can be written as
$$
Q_v =  1+  \sum_{k=1}^t \omega^{-4k} Z_{i_1}^2Z_{i_2}^2 \dots Z_{i_k}^2  + \sum_{l=0}^u \omega^{-4t - 4l+4} Z_{i_1}^2 Z_{i_2}^2 \dots Z_{i_t}^2 Z_{i_0}^2Z_{j_1}^2Z_{j_2}^2\dots Z_{j_l}^2
$$

If we apply the representation $\mu_\lambda \colon \ZZ(\lambda) \to \End({E_\lambda})$ and  remember that $\mu_\lambda(H_v^{-1})=-\omega^{-4} \Id_{E_\lambda}$, this proves that
$$
\mu_\lambda \Bigl( [Z_{i_1} Z_{i_2} \dots Z_{i_t}] \Bigr)\circ  \biggl( \mu_\lambda  \Bigl(  \Tr_\lambda^\omega \bigl( [K_1] \bigr) \Bigr) - \mu_\lambda  \Bigl(  \Tr_\lambda^\omega \bigl( [K_1] \bigr) \Bigr) \biggr)= \mu_\lambda(Q_v) .
$$

Since $ [Z_{i_1} Z_{i_2} \dots Z_{i_t}]$ is invertible in $\ZZ(\lambda)$, the linear map $\mu_\lambda \Bigl(  [Z_{i_1} Z_{i_2} \dots Z_{i_t}] \Bigr)\in \End({E_\lambda})$ is invertible. It follows that the kernel of $\mu_\lambda  \Bigl(  \Tr_\lambda^\omega \bigl( [K_1] \bigr) \Bigr) - \mu_\lambda  \Bigl(  \Tr_\lambda^\omega \bigl( [K_1] \bigr) \Bigr) $ is equal to the kernel of $ \mu_\lambda(Q_v)$, namely to the off-diagonal kernel $F_{\lambda}=F_v \subset {E_\lambda}$. 

Since $\rho_\lambda = \mu_\lambda \circ \Tr_\lambda^\omega$, this  completes the proof of Lemma~\ref{lem:OffDiagKernelAndSweep}. 
\end{proof}

We now consider the algebraic structure of the balanced Chekhov-Fock algebra $\ZZ(\lambda)$ and of the irreducible representation $\mu_\lambda \colon \ZZ(\lambda) \to \End({E_\lambda})$. 

Let $\lambda_1$ and $\lambda_2$ be the triangulations of the surfaces $S_1$ and $S_2$ respectively induced by the triangulation $\lambda$. Define the balanced Chekhov-Fock algebra $\ZZ(\lambda_1)$ as the subalgebra of $\ZZ(\lambda)$ generated by all monomials in the generators $Z_{i_1}$, $Z_{i_2}$, \dots, $Z_{i_t}$ satisfying the appropriate exponent parity condition. Similarly, $\ZZ(\lambda_2)\subset \ZZ(\lambda)$ is generated by all monomials in the generators $Z_{j_1}$, $Z_{j_2}$, \dots, $Z_{j_u}$ with the appropriate exponent parity condition.

Because each $Z_{i_k}$ (with $k>0$) commutes with each $Z_{j_l}$ and because the element  $H_v$ is central, the inclusion maps $\ZZ(\lambda_1) \to \ZZ(\lambda)$, $\ZZ(\lambda_2) \to \ZZ(\lambda)$ and $\C[H_v^{\pm1}] \to \ZZ(\lambda)$ define an algebra homomorphism
$$
\ZZ(\lambda_1) \otimes \ZZ(\lambda_2) \otimes \C[H_v^{\pm1}]   \to \ZZ(\lambda).
$$

\begin{lem}
\label{lem:CheFockOneHole}
The above homomorphism defines an isomorphism
$$
\ZZ(\lambda) \cong \ZZ(\lambda_1) \otimes \ZZ(\lambda_2) \otimes \C[H_v^{\pm1}]   .
$$
\end{lem}
\begin{proof}
We need to show that the algebra homomorphism above is a linear isomorphism. 

The key observation for this is the following. For every monomial $$Z_{i_0}^n Z_{i_1}^{m_1}Z_{i_2}^{m_2} \dots Z_{i_t}^{m_t} Z_{j_1}^{n_1}Z_{j_2}^{n_2} \dots Z_{j_u}^{n_u}$$ of $\ZZ(\lambda)$, the exponent parity condition defining the balanced Chekhov-Fock algebra implies that the exponent $n$ of $Z_{i_0}$ is even, because the edge $e_0$ separates the surface $S$. As a consequence, such a monomial can be uniquely split as the product of a monomial of $\ZZ(\lambda_1)$, a monomial of $\ZZ(\lambda_2)$, and a power of $H_v = [Z_{i_0}^2 Z_{i_1}Z_{i_2}\dots Z_{i_t} Z_{j_1}Z_{j_2} \dots Z_{j_u}]$. 

Since these monomials $Z_{i_0}^n Z_{i_1}^{n_1}Z_{i_2}^{n_2} \dots Z_{i_t}^{n_t} Z_{j_1}^{m_1}Z_{j_2}^{m_2} \dots Z_{j_u}^{m_u}$ form a basis for $\ZZ(\lambda)$, while the monomials $ Z_{i_1}^{n_1}Z_{i_2}^{n_2} \dots Z_{i_t}^{n_t} $ form a basis for $\ZZ(\lambda_1)$ and  the monomials $ Z_{j_1}^{m_1}Z_{j_2}^{m_2} \dots Z_{j_u}^{m_u}$ form a basis for $\ZZ(\lambda_1)$, the result immediately follows. 
\end{proof}

The structure theorem provided by Lemma~\ref{lem:CheFockOneHole} enables us to split the irreducible representation $\mu_\lambda \colon \ZZ(\lambda) \to \End({E_\lambda})$ as a tensor product. Indeed, by elementary linear algebra (see for instance \cite[\S 4]{BonLiu}) or a careful analysis of the proof of Proposition~\ref{prop:ConstructRepCheFock}, there exists  irreducible representations $\mu_1  \colon \ZZ(\lambda_1) \to \End(E_1)$, $\mu_2 \colon \ZZ(\lambda_2) \to \End(E_2)$, $\mu_0 \colon \C[H_v^{\pm1}]   \to \End(\C)$, and an isomorphism ${E_\lambda} \cong E_1 \otimes E_2$ for which $\mu_\lambda$ corresponds to
$$
\mu_1 \otimes \mu_2 \otimes \mu_0 \colon 
 \ZZ(\lambda_1) \otimes \ZZ(\lambda_2) \otimes \C[H_v^{\pm1}]  \to 
\End(E_1 \otimes E_2 \otimes \C) = \End({E_\lambda}).
$$
In fact, since $\mu_\lambda(H_v) = -\omega^4 \Id_{E_\lambda}$, $\mu_0$ is the unique algebra homomorphism such that $\mu_0(H_v) = -\omega^4 \Id_\C$.

We now return to the knots $K_1$, $K_2\subset S\times [0,1]$. The knot $K_1$ is  contained in $S_1 \times [0,1]$, so that the quantum trace $\Tr_\lambda^\omega\bigl( [K_1] \bigr)$ belongs to the subalgebra $\ZZ(\lambda_1) \subset \ZZ(\lambda)$ corresponding to the subsurface $S_1$. In particular,  
$$
\rho_\lambda \bigl( [K_1] \bigr) = \mu_\lambda \bigl( \Tr_\lambda^\omega\bigl( [K_1] \bigr) \bigr) = \mu_1 \bigl( \Tr_\lambda^\omega\bigl( [K_1] \bigr) \bigr) \otimes \Id_{E_2}
$$
in $\End({E_\lambda}) = \End(E_1 \otimes E_2) = \End(E_1) \otimes \End(E_2)$. Similarly, 
$$
\rho_\lambda \bigl( [K_2] \bigr) =  \mu_\lambda \bigl( \Tr_\lambda^\omega\bigl( [K_2] \bigr) \bigr) =   \Id_{E_1} \otimes \mu_2 \bigl( \Tr_\lambda^\omega\bigl( [K_2] \bigr) \bigr).
$$

By Lemma~\ref{lem:OffDiagKernelAndSweep}, the off-diagonal kernel $F_{\lambda}\subset {E_\lambda}$ is equal to the kernel of $\rho_\lambda \bigl( [K_1] \bigr) -\rho_\lambda \bigl( [K_2] \bigr) $. The following statement is then an immediate consequence of the above observations. 

\begin{lem}
\label{lem:EigenvaluesTensorProduct}
The off-diagonal kernel $F_{\lambda} \subset {E_\lambda}=E_1 \otimes E_2$ is equal to 
$$
F_{\lambda}= \bigoplus_{a\in \C} \,E_1^{(a)} \otimes E_2^{(a)}
$$
where, for each $a\in \C$, $E_i^{(a)} = \bigl\{ w\in E_i, \mu_i \bigl( \Tr_\lambda^\omega\bigl( [K_i] \bigr) \bigr) (w) = aw \bigr \} $ is the eigenspace of $ \mu_i \bigl( \Tr_\lambda^\omega\bigl( [K_i] \bigr) \bigr) \in \End(E_\lambda)$ corresponding to $a$ if $a$ is an eigenvalue of this endomorphism, and is $0$ otherwise. \qed 
\end{lem}

This reduces the problem to the determination of the eigenvalues and eigenspaces of the homomorphisms $\mu_1 \bigl( \Tr_\lambda^\omega\bigl( [K_1] \bigr) \bigr) $ and $\mu_2 \bigl( \Tr_\lambda^\omega\bigl( [K_2] \bigr) \bigr) $.

Let us focus attention on the first homomorphism. The eigenvalues and eigenspaces of $\mu_1 \bigl( \Tr_\lambda^\omega\bigl( [K_1] \bigr) \bigr)$ are easily deduced from those of $\rho_\lambda \bigl( [K_1] \bigr) = \mu_1 \bigl( \Tr_\lambda^\omega\bigl( [K_1] \bigr) \bigr) \otimes \Id_{E_1}$. 

\begin{lem}
\label{lem:EigenspaceDimensions}
Suppose that the homomorphism $r \colon \pi_1(S) \to \SL(\C)$ is  generic enough that $\Tr\, r(e_{i_0}) \neq \pm 2$. Then the homomorphism $\rho_\lambda \bigl( [K_1] \bigr) \in \End(E_\lambda)$ is diagonalizable, its eigenvalues are the $N$ distinct solutions of the equation $T_N(x)= -\Tr\, r(K_1)$,  and all of its eigenspaces have the same dimension $\frac1N \dim E_\lambda$. 
\end{lem}

\begin{proof} We begin with a simple observation about the Chebyshev polynomial $T_N(x)$. If $y\neq \pm2$, the equation $T_N(x)=y$ has $N$ distinct solutions. Indeed, if we write $y$ as $y=b+b^{-1}$ for some $b$, the solutions to the equation $T_N(x)=y$ are of the form $x=a + a^{-1}$ as $a$ ranges over all $N$--roots of $b$. A little algebraic manipulation shows that these solutions are all distinct unless $b =\pm1$, which is excluded by our hypothesis that $y\neq \pm2$. 

The fact that $\rho_\lambda \bigl( [K_1] \bigr) \in \End(E_\lambda)$ is diagonalizable is then an immediate consequence of this observation and of the property, provided by Conclusion~(4) of Proposition~\ref{prop:ConstructRepCheFock}, that 
$$
T_N \bigl(  \rho_\lambda \bigl( [K_1] \bigr) \bigr) =  -\Tr\, r(K_1) \,\Id_{E_\lambda} =  -\Tr\, r(e_{i_0}) \,\Id_{E_\lambda} . 
$$
This proves that all eigenvalues of $ \rho_\lambda \bigl( [K_1] \bigr)$ are solutions of the equation $T_N(x) =  -\Tr\, r(e_{i_0})$ and, since all solutions to this equation are simple by our hypothesis that $\Tr\, r(e_{i_0}) \neq \pm 2$, that $ \rho_\lambda \bigl( [K_1] \bigr)$ is diagonalizable. 

Showing that all solutions of the above equation occur as eigenvalues, and computing the dimension of the corresponding eigenspaces, will require a more elaborate argument.

By Complement~\ref{comp:ConstructRepCheFockContinuous},  if we vary the enhanced homomorphism $(r, \xi)$ over a small open subset in the space of such pairs, the representation $\mu_\lambda \colon \ZZ(\lambda) \to \End({E_\lambda})$ can be chosen so that, for every monomial $Z_1^{k_1}Z_2^{k_2} \dots Z_n^{k_n} \in \ZZ(\lambda)$, 
$$
\mu_\lambda \bigl( [Z_1^{k_1}Z_2^{k_2} \dots Z_n^{k_n}] \bigr) = u_1^{k_1}u_2^{k_2} \dots u_n^{k_n} \, A_{k_1k_2\dots k_n} 
$$
where each $u_i=\sqrt[2N]{x_i}$ is a local determination of the $2N$--root of the crossratio weight $x_i$ defined by $(r, \xi)$, and where the endomorphisms  $A_{k_1k_2\dots k_n} \in \End({E_\lambda})$ are independent of $(r,\xi)$. 


We can now reverse the process and add more generality to it in order to give ourselves some flexibility. Consider  the space $\mathcal W = (\C-\{0\})^n$ of weight systems $\mathcal u$ assigning a  weight $u_i$ to each edge $e_i$ of $\lambda$, with no specific relation between these edge weights. (The edge weights $u_i = \sqrt[2N]{x_i}$ associated to an enhanced homomorphism $(r, \xi)$ that we  considered so far were constrained by the relations of \S \ref{subsect:ClassicalOffDiagonalTerm}.) An edge weight system $\mathbf u \in \mathcal W$ determines a representation $\mu_\lambda^{\mathbf u} \colon \ZZ(\lambda) \to \End({E_\lambda})$ by the property that 
$$
\mu_\lambda^{\mathbf u} (Z_1^{k_1}Z_2^{k_2} \dots Z_n^{k_n}) = u_1^{k_1}u_2^{k_2} \dots u_n^{k_n} \, A_{k_1k_2\dots k_n} 
$$
for every monomial $Z_1^{k_1}Z_2^{k_2} \dots Z_n^{k_n} \in \ZZ(\lambda)$, where the endomorphisms  $A_{k_1k_2\dots k_n} \in \End({E_\lambda})$ are the ones occurring above. 

This associates to $\mathbf u \in \mathcal W$ a representation $\rho_\lambda^{\mathbf u} = \mu_\lambda^{\mathbf u} \circ \Tr_\lambda^\omega \colon \SSS(S_\lambda) \to \End(E_\lambda)$, and the miraculous cancellations of \cite{BonWon3} (as used in \cite[\S4]{BonWon4}) provide a homomorphism $r^{\mathbf u} \colon \pi_1(S_\lambda) \to \SL(\C)$ such that 
$$
T_N \bigl( \rho_\lambda^{\mathbf u}  \bigl([K] \bigr)  \bigr)=
T_N \bigl( \mu_\lambda^{\mathbf u} \circ \Tr_\lambda^\omega \bigl([K] \bigr)  \bigr) = - \Tr \,  r^{\mathbf u}(K) \, \Id_{E_\lambda}
$$
for every framed knot $K\subset S_\lambda \times [0,1]$ whose projection to $S_\lambda$ has no crossing and whose framing is vertical.

In particular,  if we return to  the formula
$$
 \Tr_\lambda^\omega \bigl( [K_1] \bigr) = \omega^{-t}
 \sum_{k=0}^t Z_{i_0}Z_{i_1}\dots Z_{i_k} Z_{i_{k+1}}^{-1}Z_{i_{k+2}}^{-1}\dots Z_{i_t}^{-1}Z_{i_0}^{-1} 
$$
used in the proof of Lemma~\ref{lem:OffDiagKernelAndSweep}, this gives 
$$
 \rho_\lambda^{\mathbf u} \bigl( [K_1] \bigr)   =  \mu_\lambda^{\mathbf u} \circ \Tr_\lambda^\omega \bigl( [K_1] \bigr)  =  \sum_{k=0}^t u_{i_1} u_{i_2}\dots u_{i_k} u_{i_{k+1}}^{-1} u_{i_{k+2}}^{-1}\dots  u_{i_t}^{-1} A_k
$$
where $A_k  \in \End(E_\lambda)$ is the product of a suitable term $A_{k_1k_2\dots k_n}$ with a power of $\omega$. In particular, we will use the observation that for $k=t$
$$
u_{i_1} u_{i_2}\dots u_{i_t} A_t = \mu_\lambda^{\mathbf u} \bigl( \omega^{-t} Z_{i_0}Z_{i_1}Z_{i_2}\dots Z_{i_t} Z_{i_0}^{-1}  \bigr) = \mu_\lambda^{\mathbf u} \bigl( [Z_{i_1}Z_{i_2}\dots Z_{i_t} ]\bigr)
$$
where the quantum ordering computation comes from Lemma~\ref{lem:QuantumOrderAtVertex} and the fact that $Z_{i_0}$ commutes with $Z_{i_1}Z_{i_2}\dots Z_{i_t}$ in $\TT(\lambda)$. Similarly, for $k=0$, 
$$
u_{i_1}^{-1} u_{i_2}^{-1}\dots u_{i_t}^{-1} A_0 = \mu_\lambda^{\mathbf u} \bigl( \omega^{-t} Z_{i_0}Z_{i_1}^{-1}Z_{i_2}^{-1}\dots Z_{i_t}^{-1} Z_{i_0}^{-1}  \bigr) = \mu_\lambda^{\mathbf u} \bigl( [Z_{i_1}^{-1}Z_{i_2}^{-1}\dots Z_{i_t}^{-1} ]\bigr),
$$
from which it follows that $A_0 = A_t^{-1}$. 

Also, by our determination of the algebraic structure of $\ZZ(\lambda)$ in \cite[\S 2.2]{BonWon4} (and in particular Lemma~10 of that article),  $ [Z_{i_1}^{N}Z_{i_2}^{N}\dots Z_{i_t}^{N} ]$ is central in $\ZZ(\lambda)$ since $\omega^{4N}=1$. By irreducibility of the representation $\mu_\lambda^{\mathbf u}$, there consequently exists a number $x\in \C^*$ such that $\mu_\lambda^{\mathbf u} \bigl( [Z_{i_1}^{N}Z_{i_2}^{N}\dots Z_{i_t}^{N} ]=x\,\Id_{E_\lambda}$. Taking the square of this equation and using the property that $\mu_\lambda^{\mathbf u}(Z_i^{2N}) = u_i^{2N}\,\Id_{E_\lambda}$, we conclude that $x = \pm u_{i_1}^{N}u_{i_2}^{N}\dots u_{i_t}^{N} $ and that 
$$
A_t^{N} = u_{i_1}^{-N} u_{i_2}^{-N} \dots u_{i_t}^{-N} \mu_\lambda^{\mathbf u}\bigl([Z_{i_1}^{N} Z_{i_2}^{N} \dots Z_{i_t}^{N}]\bigr) =\pm \Id_{E_\lambda}.
$$

After these preliminary observations, we now return to the main line of our proof. 
If $\Tr \,  r^{\mathbf u}(K_1) \neq \pm2$, the same argument as before shows that $ \rho_\lambda^{\mathbf u} \bigl( [K_1] \bigr) $ is diagonalizable, and that all its eigenvalues are solutions of the equation $T_N(x) = - \Tr \,  r^{\mathbf u}(K_1)$. Our strategy will be to determine the dimension of the eigenspaces of $ \rho_\lambda^{\mathbf u} \bigl( [K_1] \bigr) $ for one specific value of $\mathbf u$, and then to conclude by a connectedness property that these dimensions are the same for all $\mathbf u \in \mathcal W$ with $\Tr \,  r^{\mathbf u}(K_1) \neq \pm2$.

For this, we borrow two distinct ideas from Julien Roger \cite{Roger1, Roger2}. The first one is a  result of \cite[Lemma~19]{Roger1}, where Roger considers  monomials of $\ZZ(\lambda)$ associated to simple closed curves in the punctured surface $S_\lambda$. In the case of $K_1$, the corresponding monomial is $[Z_{i_1} Z_{i_2} \dots Z_{i_t}] $ and Roger produces a monomial $B\in \TT(\lambda_1)$ such that $B[Z_{i_1} Z_{i_2} \dots Z_{i_t}] = \omega^4 [Z_{i_1} Z_{i_2} \dots Z_{i_t}] B$. Taking the square $B^2$ to make sure that we  have an element of the balanced Chekhov-Fock algebra $\ZZ(\lambda)$, applying the representation $\mu_\lambda^{\mathbf u} \colon \ZZ(\lambda) \to \End(E_\lambda)$ associated to $\mathbf u \in \mathcal W$, and remembering that $\mu_\lambda^{\mathbf u} \bigl( [Z_{i_1}Z_{i_2}\dots Z_{i_t} ]\bigr) = u_{i_1} u_{i_2}\dots u_{i_t} A_t $, it follows that $\mu_\lambda^{\mathbf u}(B^2) A_t = \omega^8 A_t\, \mu_\lambda^{\mathbf u}(B^2)$. As a consequence,  $\mu_\lambda^{\mathbf u}(B^2)$ sends the eigenspace of $ A_t $ corresponding to the eigenvalue $a$ to the eigenspace corresponding to the eigenvalue $\omega^8a$. Since we observed that $A_t^N = \pm \Id_{E_\lambda}$ and since $\omega^8$ is a primitive $N$--root of unity (as $A=\omega^2$ is a primitive $N$--root of $-1$ and $N$ is odd), it follows that the  eigenvalues of $A_t$ are all $N$--roots of  $\pm1$, and that  its eigenspaces have the same dimension $\frac 1N \dim E_\lambda$. 

We now follow another idea first exploited in \cite[\S 2.2]{Roger1} and \cite[Appendix~B]{Roger2}, except that the broader context  of $\mathcal W$ enables us to use an explicit argument without having to rely on results of \cite{Roger1, Roger2}. To construct a suitable edge weight system $\widehat{\mathbf u} \in \mathcal W$, pick  an arbitrary number $u_0\in \C-\{0\}$ such that $u_0^{4N}\neq 1$, and another number $\epsilon\in \C-\{0\}$  close to 0.  Then define $\widehat{\mathbf u}$ to assign weight  $\widehat u_{i_1}=u_0\epsilon$ to the edge $e_{i_1}$, weight $\widehat u_{i_t} = \epsilon^{-1}$ to $e_{i_t}$, and weight $\widehat u_i=1$ to all other $e_i$. Remember that there exists an index $k_1$ such that the edge $e_{i_{k_1}}$ is equal to $e_{i_1}$ and $e_{i_{k_1+1}}$ is equal to $e_{i_t}$. In particular, $\widehat u_{i_{k_1}}= u_0\epsilon$ and $\widehat u_{i_{k_1+1}}=\epsilon^{-1}$. It follows that
$$
\widehat u_{i_1}\dots \widehat u_{i_k} \widehat u_{i_{k+1}}^{-1}\widehat u_{i_{k+2}}^{-1}\dots \widehat u_{i_t}^{-1} =
\begin{cases}
u_0^{-2} &\text{ if } k=0\\
\epsilon^2 &\text{ if } 0<k<k_1\\
u_0^2 \epsilon^4 &\text{ if } k=k_1\\
u_0^2 \epsilon^2 &\text{ if } k_1<k<t\\
u_0^2 &\text{ if } k=t
\end{cases}
$$

Then, if $\epsilon$ is sufficiently small,
$$
 \rho_\lambda^{\widehat{\mathbf u}} \bigl( [K_1] \bigr)  
= \sum_{k=0}^t \widehat u_{i_1}\dots \widehat u_{i_k} \widehat u_{i_{k+1}}^{-1} \widehat u_{i_{k+2}}^{-1}\dots \widehat u_{i_t}^{-1} A_k
$$
is very close to
$$
C= u_0^{-2} A_0 + u_0^2 A_t = u_0^{-2} A_t^{-1} + u_0^2 A_t \in \End(E_\lambda). 
$$

We proved that the eigenvalues of $A_t$ are all $N$--roots of $\pm1$, where $\pm$ is the sign such that $A_t^N = \pm \Id_{E_\lambda}$. Therefore, the eigenvalues of $C = u_0^{-2} A_t^{-1} + u_0^2 A_t$ are the numbers $\pm(u_0^{-2} \omega^{-4k} + u_0^2 \omega^{4k})$ with $k=0$, $1$, \dots, $N-1$. These $N$ numbers  are distinct by our assumption that $u_0^{4N} \neq 1$.  The eigenspaces of $C$ are the eigenspaces of $A_t$, which we proved all  have the same dimension $\frac1N \dim E_\lambda$. 

Therefore, for $\widehat{\mathbf u} \in \mathcal W$ associated to $u_0$ and $\epsilon$ as above, with $\epsilon$ small enough, the diagonalizable endomorphism $ \rho_\lambda^{\widehat{\mathbf u}} \bigl( [K_1] \bigr)  \in \End(E_\lambda)$ has $N$ distinct eigenvalues and the corresponding eigenspaces all have dimension $\frac 1N \dim E_{\lambda}$.

In the space $\mathcal W\cong (\C^*)^n$ of edge weight systems for $\lambda$, the subspace $\mathcal W'$ consisting of those $\mathbf u \in \mathcal W$ with $\Tr \,  r^{\mathbf u}(K_1) \neq \pm2$ is connected, since its complement has complex codimension 1. Note that the above point $\widehat{\mathbf u}$ belongs to $\mathcal W'$ since 
$$
\Tr \,  r^{\widehat{\mathbf u}}(K_1) \,\Id_{E_\lambda} = T_N \bigl (\rho_\lambda^{\widehat{\mathbf u}} \bigl( [K_1] \bigr)  \bigr)
$$
is very close to 
$$
T_N(u_0^{-2} A_t^{-1} + u_0^2 A_t) = u_0^{-2N} A_t^{-N} + u_0^{2N} A_t^N = \pm (u_0^{-2N}  + u_0^{2N}) \Id_{E_\lambda}.
$$
Therefore, the trace  $\Tr \,  r^{\widehat{\mathbf u}}(K_1)$ is very close to $\pm (u_0^{-2N}  + u_0^{2N}) $, and is consequently different from $\pm 2\, \Id_{E_\lambda}$ by our assumption that $u_0^{4N} \neq 1$. 

We saw that, for all $\mathbf u \in \mathcal W'$, the endomorphism $ \rho_\lambda^{\mathbf u} \bigl( [K_1] \bigr) $ is diagonalizable and its eigenvalues are solutions of the equation $T_N(x) =- \Tr \,  r^{\mathbf u}(K_1) $. Since the solutions of that equation are always simple for $\mathbf u \in \mathcal W'$, the dimension of the eigenspaces is a locally constant function of $\mathbf u$, and is therefore constant by connectedness of $\mathcal W'$. We found one point $\widehat {\mathbf u} \in \mathcal W$ such that all eigenspaces of  $\rho_\lambda^{\widehat{\mathbf u}} \bigl( [K_1] \bigr) $ have dimension  $\frac1N \dim E_\lambda$. Therefore, the eigenspaces of  $\rho_\lambda^{\mathbf u} \bigl( [K_1] \bigr) $ have dimension  $\frac1N \dim E_\lambda$ for every $\mathbf u \in \mathcal W'$. 

In particular, this property holds for $\mathbf u \in \mathcal W'$ defined by the edge weights $u_i = \sqrt[2N]{x_i}$ associated to the enhanced homomorphism $(r, \xi)$ of the hypotheses of Lemma~\ref{lem:EigenspaceDimensions}, which proves this statement. 
\end{proof}

\begin{prop}
\label{prop:DimensionOffDiagKernelHighGenusGenericCase}
Let $S$ be a closed oriented surface of genus $g\geq 2$, and consider a homomorphism $r\colon \pi_1(S) \to \SL(\C)$. Suppose that there exists a triangulation $\lambda_0$  of $S$ with exactly one vertex $v$ and with at least one separating edge $e_{i_0}$, such that  $r(e_i)\neq \pm \Id$ for every edge $e_i$ of $\lambda_0$ and  $\Tr \, r(e_{i_0}) \neq \pm 2$. Then, for every triangulation $\lambda$ of $S$ and every  $\lambda$--enhancement $\xi$ for $r$, the off-diagonal kernel $F_{\lambda}$ associated to this data has dimension
$$
\dim F_{\lambda} = N^{3(g-1)}. 
$$
\end{prop}
\begin{proof} The hypotheses on $r$ and $\lambda_0$ guarantee that, by Lemma~\ref{lem:EnhancementsExist},  $r$ admits at east one  $\lambda_0$--enhancement $\xi_0$. By Theorem~\ref{thm:RepIndependentChoices}, the total off-diagonal kernel $F_{\lambda_0}$ is isomorphic to $F_{\lambda}$ and  we can consequently restrict attention to the case where $\lambda = \lambda_0$ and $\xi = \xi_0$.

Namely, we assume that $\lambda$ has exactly one vertex $v$, and that $r$ admits a $\lambda$--enhancement $\xi$; in addition, an edge $e_{i_0}$ of $\lambda$ separates $S$ into two subsurfaces $S_1$ and $S_2$, and $\Tr\, r(e_{i_0}) \neq \pm2$. In this case and with the  notations of this section, recall that we have split the representation $\mu_\lambda \colon \ZZ(\lambda) \to \End(E_\lambda)$ associated to the enhanced homomorphism $(r, \xi)$ as a tensor product
$$
\mu_1 \otimes \mu_2 \otimes \mu_0 \colon 
 \ZZ(\lambda_1) \otimes \ZZ(\lambda_2) \otimes \C[H_v^{\pm1}]  \to 
\End(E_1 \otimes E_2 \otimes \C) = \End({E_\lambda}).
$$
of three irreducible representations $\mu_1  \colon \ZZ(\lambda_1) \to \End(E_1)$, $\mu_2 \colon \ZZ(\lambda_2) \to \End(E_2)$, $\mu_0 \colon \C[H_v^{\pm1}]   \to \End(\C)$, for  isomorphisms $\ZZ(\lambda) \cong \ZZ(\lambda_1) \otimes \ZZ(\lambda_2) \otimes \C[H_v^{\pm1}]   $ and  ${E_\lambda} \cong E_1 \otimes E_2$. In addition,  $\mu_0$ is the unique algebra homomorphism such that $\mu_0(H_v) = -\omega^4 \Id_\C$.

By Lemma~\ref{lem:EigenvaluesTensorProduct}
$$
\dim F_{\lambda} = \sum_{a\in \C} \dim E_1^{(a)} \dim E_2^{(a)} .
$$
where $E_i^{(a)}$ is the eigenspace of $\mu_i \bigl( \Tr_{\lambda_i}^\omega \bigr( [K_i]\bigr) \bigr)$ corresponding to the eigenvalue $a$ (and is 0 if $a$ is not an eigenvalue). 

Since $
\rho_\lambda \bigl( [K_1] \bigr) = \mu_\lambda \bigl( \Tr_\lambda^\omega\bigl( [K_1] \bigr) \bigr) = \mu_1 \bigl( \Tr_\lambda^\omega\bigl( [K_1] \bigr) \bigr) \otimes \Id_{E_2}
$, the $a$--eigenspace of $\rho_\lambda \bigl( [K_1] \bigr) $ is equal to the tensor product  $E_1^{(a)} \otimes E_2$.  By Lemma~\ref{lem:EigenspaceDimensions}, we conclude that 
$$
\dim E_1^{(a)} \dim E_2 = {\textstyle\frac1N} \dim E_\lambda = {\textstyle\frac1N}  \dim E_1 \dim E_2
$$
when $T_N(a) = - \Tr\,r(K_1)$, and $E_1^{(a)}=0$ otherwise. As a consequence, $\dim E_1^{(a)} $ is equal to $ {\textstyle\frac1N}  \dim E_1 $ if  $T_N(a) = - \Tr\,r(K_1)$ and to $0$ otherwise.

Similarly, $\dim E_2^{(a)} $ is equal to $ {\textstyle\frac1N}  \dim E_2 $ if  $T_N(a) = - \Tr\,r(K_2)$ and to $0$ otherwise.

By hypothesis, $\Tr\,r(K_1) = \Tr\,r(K_2) \neq \pm2$, so there are exactly $N$ values of $a$ that have non-zero contributions to the sum
$$
\dim F_{\lambda} = \sum_{a\in \C} \dim E_1^{(a)} \dim E_2^{(a)}  = N ( {\textstyle\frac1N}  \dim E_1)( {\textstyle\frac1N}  \dim E_2) =  {\textstyle\frac1N}  \dim E_\lambda = N^{3g-3}
$$
since $\dim E_\lambda = N^{3g-2}$ by Proposition~\ref{prop:ConstructRepCheFock}. 
\end{proof}

\begin{rem}
\label{rem:EnhanceOneVertexTriangGeneric}
If we fix a triangulation $\lambda_0$ with exactly one vertex and with at least one separating edge $e_{i_0}$, the homomorphisms $r$ satisfying the hypotheses of Proposition~\ref{prop:DimensionOffDiagKernelHighGenusGenericCase} form a Zariski open dense subset of the space of all group homomorphisms $\pi_1(S) \to \SL(\C)$. Indeed, for a simple closed curve $\gamma$, many possible arguments show that the set of characters $r\in \RR(S)$ such that $\Tr\, r(\gamma) \neq \pm2$ is Zariski open and dense in $\RR(S)$. 
\end{rem}

\begin{prop}
\label{prop:DimensionOffDiagKernelHighGenusGeneralCase}
Let $S$ be a closed oriented surface of genus $g\geq 2$. Then, for every homomorphism $r\colon \pi_1(S) \to \SL(\C)$ and for every triangulation $\lambda$ such that $r$ admits a $\lambda$--enhancement $\xi$, the total off-diagonal kernel $F_\lambda$ defined by this enhanced homomorphism $(r, \xi)$ has dimension
$$
\dim F_{\lambda} \geq N^{3(g-1)}. 
$$
\end{prop}

\begin{proof} By Theorem~\ref{thm:RepIndependentChoices}, the dimension of $F_\lambda$ depends only on the group homomorphism $r$, not on the triangulation $\lambda$ or the enhancement $\xi$. In particular, we can assume without loss of generality that $\lambda$ is combinatorial, so that every homomorphism $r\colon \pi_1(S) \to \SL(\C)$ admits an enhancement $\xi$ by Lemma~\ref{lem:EnhancementsExist}. 
If we locally vary $r$, the proof of Lemma~\ref{lem:EnhancementsExist} shows that we can choose the enhancement $\xi$ so that it varies continuously with $r$. Then, the representation $\mu_\lambda \colon \ZZ(\lambda) \to \End(E_\lambda)$ associated to $(r,\xi)$ by Proposition~\ref{prop:ConstructRepCheFock} depends continuously on $r$ by Complement~\ref{comp:ConstructRepCheFockContinuous}. 

The total off-diagonal kernel $F_\lambda$ is defined as an intersection of kernels $\ker \mu_\lambda(Q_v)$. Its dimension is therefore a lower semi-continuous function of the representation $\mu_\lambda$, thus of the homomorphism $r$. Proposition~\ref{prop:DimensionOffDiagKernelHighGenusGenericCase} (see also Remark~\ref{rem:EnhanceOneVertexTriangGeneric}) asserts that the dimension of $F_\lambda$ is equal to $N^{3g-3}$ for generic homomorphisms $r\colon \pi_1(S) \to \SL(\C)$. By lower semi-continuity, it follows that $\dim F_\lambda  \geq N^{3g-3}$ for all $r$. 
\end{proof}

\subsection{Proof of Theorem~\ref{thm:DimOffDiagKernel} when the surface  $S$ is the torus}

\begin{prop}
\label{prop:DimensionOffDiagKernelTorusGenericCase}
Suppose that the surface $S$ is a torus, and that the image of the homomorphism $\bar r \colon \pi_1(S) \to \PSL(\C)$ induced by $r \colon \pi_1(S) \to \SL(\C)$ has more than two elements. Then, for every triangulation $\lambda$ of $S$ and every $\lambda$--enhancement $\xi$ for $r$, the associated off-diagonal kernel has dimension
$$
\dim F_\lambda = N.
$$
\end{prop}

\begin{proof} By Theorem~\ref{thm:RepIndependentChoices}, the dimension of $F_\lambda$ is independent of the triangulation $\lambda$ and of the enhancement $\xi$. This provides us with  flexibility in the choice of $\lambda$ to perform computations. 

By hypothesis, the image of $\bar r \colon \pi_1(S)  \to \PSL(\C)$ is neither trivial nor isomorphic to $\Z_2$.  A simple algebraic manipulation then provides a set of generators $a_1$, $a_2$ of $ \pi_1(S) \cong \Z^2$ such that $\bar r(a_1)$, $\bar r(a_2)$ and $\bar r(a_1a_2)\in \PSL(\C)$ are non-trivial. Then there exists a triangulation $\lambda$ with one vertex $v$,  and whose edges $e_1$, $e_2$ and $e_3$ respectively represent the classes $a_1$, $a_2$ and $a_1a_2$ in $\pi_1(S)$. By Lemma~\ref{lem:EnhancementsExist}, this guarantees that there exists  a $\lambda$--enhancement $\xi$ for $r$. 

In the Chekhov-Fock algebra $\TT(\lambda)$, let $Z_1$, $Z_2$, $Z_3$ be the generators respectively associated to the edges $e_1$, $e_2$ and $e_3$. Exchanging the r\^oles of $e_1$ and $e_2$ if necessary, we can assume that $e_1$, $e_2$, $e_3$ arise in this order clockwise around each of the two faces of $\lambda$. Then the skew-commutativity relations satisfied by the $Z_i$ are that $Z_iZ_{i+1} = \omega^4 Z_{i+1}Z_i$ for every $i$ (considering indices modulo 3).

The  central element $H_v \in \ZZ(\lambda)$ associated to the vertex $v$ is equal to
$$
H_v = [Z_1^2 Z_2^2 Z_3^2] = \omega^{-8} Z_1^2 Z_2^2 Z_3^2.
$$
while its off-diagonal term is 
\begin{align*}
Q_v &= 1 + \omega^{-4}Z_1^2 + \omega^{-8} Z_1^2Z_2^2 + \omega^{-12} Z_1^2 Z_2^2 Z_3^2\\
&\qquad\qquad \qquad\qquad+ \omega^{-12} Z_1^2 Z_2^2 Z_3^2Z_1^2 + \omega^{-12} Z_1^2 Z_2^2 Z_3^2Z_1^2Z_2^2 \\
&= (1 +  \omega^{-12} Z_1^2 Z_2^2 )(1 + \omega^{-4}Z_1^2 + \omega^{-8} Z_1^2Z_2^2 )= (1 +  \omega^{-4} H_v )(1 + \omega^{-4}Z_1^2 + \omega^{-8} Z_1^2Z_2^2 ).
\end{align*}

The representation $\mu_\lambda \colon \ZZ(\lambda) \to \End(E_\lambda)$ associated to the enhanced homomorphism $(r, \xi)$ by Proposition~\ref{prop:ConstructRepCheFock} has dimension $\dim E_\lambda =N$, and sends $H_v$ to $-\omega^4 \Id_{E_\lambda}$. The above computation shows that $\mu_\lambda(Q_v) =0 \in \End(E_\lambda)$. Therefore, the off-diagonal kernel is equal to 
$$
F_\lambda = \ker \mu_\lambda(Q_v) = \ker 0 = E_\lambda
$$
and has dimension $N$. 
\end{proof}

The hypotheses of Proposition~\ref{prop:DimensionOffDiagKernelTorusGenericCase} are realized on a Zariski open dense subset of the space of homomorphisms $r\colon \pi_1(S) \to \SL(\C)$. The same lower semi-continuity argument used in the proof of Proposition~\ref{prop:DimensionOffDiagKernelHighGenusGeneralCase}  gives the following general statement.

\begin{prop}
\label{prop:DimensionOffDiagKernelTorusGeneralCase}
Suppose that the surface $S$ is a torus. Then, for every triangulation $\lambda$ of the torus and every $\lambda$--enhancement $\xi$ for $r$, the associated off-diagonal kernel $F_\lambda$ has dimension at least $N$. \qed
\end{prop}

\subsection{Proof of Theorem~\ref{thm:DimOffDiagKernel} when the surface  $S$ is the sphere} In this case, every homomorphism $r\colon \pi_1(S) \to \SL(\C)$ is of course trivial.

\begin{prop}
\label{prop:DimensionOffDiagKernelSphere}
Suppose that the surface $S$ is a sphere. Then, for every triangulation $\lambda$ of $S$ and every $\lambda$--enhancement $\xi$ for the trivial homomorphism, the  associated total off-diagonal kernel $F_\lambda$ has dimension equal to $1$. 
\end{prop}

\begin{proof} By Theorem~\ref{thm:RepIndependentChoices}, it suffices to check this for any triangulation $\lambda$ for which the trivial homomorphism admits a $\lambda$--enhancement; in this case, this just means that every edge of $\lambda$ has distinct endpoints. 
We use the smallest such triangulation $\lambda$, with exactly three vertices and two faces glued along their boundary.  

For this triangulation, the generators $Z_1$, $Z_2$, $Z_3$ of the (unbalanced) Chekhov-Fock algebra $\TT(\lambda)$ commute, and the balanced Chekhov-Fock algebra $\ZZ(\lambda)$ is isomorphic to the Laurent polynomial algebra $\C[H_1^{\pm1}, H_2^{\pm1}, H_3^{\pm1}]$, where the $H_i = Z_{i+1}Z_{i+2}$ are the central elements associated to the three vertices of $\lambda$ (counting indices modulo 3). In particular, the representation $\mu_\lambda $ provided by Proposition~\ref{prop:ConstructRepCheFock} is 1--dimensional, and is the unique representation sending each $H_i$ to $-\omega^4 \Id_{E_\lambda}$. 

Each off-diagonal term is of the form $Q_i = 1 + \omega^{-4}Z_i^2 = 1 + \omega^{-4} H_{i+1}H_{i+2}H_i^{-1}$. The off-diagonal kernel of each vertex is therefore $\ker \mu_\lambda(Q_i) = \ker 0 = E_\lambda$, and the total off-diagonal kernel $F_\lambda$ has dimension
\begin{equation*}
\dim F_\lambda = \dim \bigcap_{i=1}^3 \ker \mu_\lambda(Q_i) = \dim E_\lambda =1 \qedhere
\end{equation*}
\end{proof}

The combination of Propositions~\ref{prop:DimensionOffDiagKernelHighGenusGenericCase}, \ref{prop:DimensionOffDiagKernelHighGenusGeneralCase}, \ref{prop:DimensionOffDiagKernelTorusGenericCase}, \ref{prop:DimensionOffDiagKernelTorusGeneralCase} and \ref{prop:DimensionOffDiagKernelSphere} completes the proof of Theorem~\ref{thm:DimOffDiagKernel}. 

\subsection{Proof of the Realization Theorem~\ref{thm:RealizeInvariantIntro}} We are now ready to complete the proof of  the Realization Theorem~\ref{thm:RealizeInvariantIntro}. 

Given a group homomorphism $r\colon \pi_1(S) \to \SL(\C)$ and a combinatorial triangulation $\lambda$ of the surface $S$, Proposition~\ref{prop:ClassicalShadowOKCombinatorial} provided a representation $\check \rho _\lambda \colon \SSS(S) \to \End(F_\lambda)$ whose classical shadow is equal to the character $r\in \RR(S)$. Theorem~\ref{thm:DimOffDiagKernel} shows that $F_\lambda$ is different from 0, so that this representation is non-trivial. The representation $\check\rho_\lambda$ may or may not be irreducible, but it admits at least one irreducible component $\rho_r \colon \SSS(S) \to \End(E)$ with $E\subset F_\lambda$. This irreducible representation satisfies the conclusions of  Theorem~\ref{thm:RealizeInvariantIntro}. 

\newcommand{\bibsub}[1]{\kern-0pt${}_{#1}$\kern -0.5pt}

\end{document}